\documentclass{amsart}    

\usepackage{graphicx}
\usepackage{amsmath}
\usepackage{amsfonts}
\usepackage{amssymb}
\usepackage{enumerate}
\usepackage{hyperref}
\usepackage{amsrefs}

\newtheorem{theorem}{Theorem}[section]
\newtheorem{lemma}[theorem]{Lemma}

\theoremstyle{definition}

\numberwithin{equation}{section}

\newcommand{\g}{\mathfrak{g}}
\newcommand{\m}{\mathfrak{m}}
\newcommand{\M}{\mathbb{M}}

\newcommand{\modspace}{\hskip 4pt}

\newcommand{\si}{{[1]}}
\newcommand{\sii}{{[2]}}

\DeclareMathOperator{\Endo}{End}
\DeclareMathOperator{\Id}{Id}
\DeclareMathOperator{\Jac}{J}
\DeclareMathOperator{\modulus}{mod}
\DeclareMathOperator{\op}{{op}}

\DeclareMathOperator{\Nalt}{N_{\rm alt}}

\input{xy}
\xyoption{all}

 \newcount\grcalca
 \newcount\grcalcb
 \newcount\grcalcc
 \newcount\grcalcd
 \newcount\grcalce
 \newcount\grcalcf
 \newcount\grcalcg
 \newcount\grcalch
 \newcount\grcalci
 \newcount\grrow
 \newcount\grcolumn
 \newcount\grwidth
 \newcount\factor
 \newcount\hfactor
 \newcount\qfactor
 \newcount\tfactor
 \newcount\sfactor
 \newcount\dfactor
 \hfactor = \factor
 \divide    \hfactor by 2
 \qfactor = \hfactor
 \divide    \qfactor by 2
 \tfactor = \factor
 \divide    \tfactor by 3
 \sfactor = \factor
 \divide    \sfactor by 6
 \dfactor = \factor
 \divide    \dfactor by 12

 \newcommand{\gbeg}[2]{
   \unitlength=1pt
   \grrow = #2
   \grcolumn = 0
   \grcalca = #1
   \grcalcb = #2
   \multiply \grcalca by \factor
   \grwidth = \grcalca
   \multiply \grcalcb by \factor
   \begin{minipage}{\grcalca pt}
   \begin{picture}(\grcalca,\grcalcb)
   \advance \grcalcb by -\factor
   }

 \newcommand{\gend}{
   \end{picture}
   {\vskip2.5ex}
   \end{minipage} }

 \newcommand{\gnl}{
   \advance \grrow by -1
   \grcolumn = 0}

 \newcommand{\gvac}[1]{
   \advance \grcolumn by #1}

 \newcommand{\gcl}[1]{
   \grcalca = \grcolumn
   \multiply \grcalca by \factor
   \advance \grcalca by \hfactor
   \grcalcb = \grrow
   \multiply \grcalcb by \factor
   \grcalcc = #1
   \multiply \grcalcc by \factor
   \put(\grcalca,\grcalcb) {\line(0,-1){\grcalcc}}
   \advance \grcolumn by 1}

 \newcommand{\gcn}[4]{
   \grcalca = \grcolumn
   \multiply \grcalca by \factor
   \grcalci = #3
   \multiply \grcalci by \hfactor
   \advance \grcalca by \grcalci
   \grcalcb = \grcolumn
   \multiply \grcalcb by \factor
   \grcalci = #3
   \advance \grcalci by #4
   \multiply \grcalci by \qfactor
   \advance \grcalcb by \grcalci
   \grcalcc = \grcolumn
   \multiply \grcalcc by \factor
   \grcalci = #4
   \multiply \grcalci by \hfactor
   \advance \grcalcc by \grcalci
   \grcalcd = \grrow
   \multiply \grcalcd by \factor
   \grcalce = \grrow
   \multiply \grcalce by \factor
   \grcalci = #2
   \multiply \grcalci by \tfactor
   \advance \grcalce by -\grcalci
   \grcalcf = \grrow
   \multiply \grcalcf by \factor
   \grcalci = #2
   \multiply \grcalci by \hfactor
   \advance \grcalcf by -\grcalci
   \grcalcg = \grrow
   \multiply \grcalcg by \factor
   \grcalci = #2
   \multiply \grcalci by \tfactor
   \multiply \grcalci by 2
   \advance \grcalcg by -\grcalci
   \grcalch = \grrow
   \advance \grcalch by -#2
   \multiply \grcalch by \factor
   \qbezier(\grcalca,\grcalcd)(\grcalca,\grcalce)(\grcalcb,\grcalcf)
   \qbezier(\grcalcb,\grcalcf)(\grcalcc,\grcalcg)(\grcalcc,\grcalch)
   \advance \grcolumn by #1}

 \newcommand{\gnot}[1]{
   \grcalca = \grcolumn
   \multiply \grcalca by \factor
   \advance \grcalca by \hfactor
   \grcalcb = \grrow
   \multiply \grcalcb by \factor
   \advance \grcalcb by -\hfactor
   \put(\grcalca,\grcalcb) {\makebox(0,0){$\scriptstyle #1$}} }

 \newcommand{\got}[2]{
   \grcalca = \grcolumn
   \multiply \grcalca by \factor
   \grcalcc = #1
   \multiply \grcalcc by \hfactor
   \advance \grcalca by \grcalcc
   \grcalcb = \grrow
   \multiply \grcalcb by \factor
   \advance \grcalcb by -\tfactor
   \advance \grcalcb by -\tfactor
   \put(\grcalca,\grcalcb){\makebox(0,0)[b]{$#2$}}
   \advance \grcolumn by #1}

 \newcommand{\gob}[2]{
   \grcalca = \grcolumn
   \multiply \grcalca by \factor
   \grcalcc = #1
   \multiply \grcalcc by \hfactor
   \advance \grcalca by \grcalcc
   \put(\grcalca,0){\makebox(0,0)[b]{$#2$}}
   \advance \grcolumn by #1}

 \newcommand{\gmu}{
   \grcalca = \grcolumn
   \advance \grcalca by 1
   \multiply \grcalca by \factor
   \grcalcb = \grrow
   \multiply \grcalcb by \factor
   \grcalcc = \factor
   \advance \grcalcc by \hfactor
   \put(\grcalca,\grcalcb){\oval(\factor,\grcalcc)[b]}
   \advance \grcalcb by -\hfactor
   \advance \grcalcb by -\qfactor
   \put(\grcalca,\grcalcb) {\line(0,-1){\qfactor}}
   \advance \grcolumn by 2}

 \newcommand{\gcmu}{
   \grcalca = \grcolumn
   \advance \grcalca by 1
   \multiply \grcalca by \factor
   \grcalcb = \grrow
   \advance \grcalcb by -1
   \multiply \grcalcb by \factor
   \grcalcc = \factor
   \advance \grcalcc by \hfactor
   \put(\grcalca,\grcalcb){\oval(\factor,\grcalcc)[t]}
   \advance \grcalcb by \factor
   \put(\grcalca,\grcalcb) {\line(0,-1){\qfactor}}
   \advance \grcolumn by 2}

 \newcommand{\glm}{
   \grcalca = \grcolumn
   \multiply \grcalca by \factor
   \advance \grcalca by \hfactor
   \grcalcb = \grcalca
   \advance \grcalcb by \factor
   \grcalcc = \grrow
   \multiply \grcalcc by \factor
   \grcalcd = \grcalcc
   \advance \grcalcd by -\tfactor
   \grcalce = \grcalcd
   \advance \grcalce by -\tfactor
   \put(\grcalca, \grcalcc){\line(0,-1){\tfactor}}
   \put(\grcalca, \grcalcd){\line(1,0){\factor}}
   \put(\grcalca, \grcalcd){\line(3,-1){\factor}}
   \put(\grcalcb, \grcalcc){\line(0,-1){\factor}}
   \advance \grcolumn by 2}

 \newcommand{\grm}{
   \grcalcb = \grcolumn
   \multiply \grcalcb by \factor
   \advance \grcalcb by \hfactor
   \grcalca = \grcalcb
   \advance \grcalca by \factor
   \grcalcc = \grrow
   \multiply \grcalcc by \factor
   \grcalcd = \grcalcc
   \advance \grcalcd by -\tfactor
   \grcalce = \grcalcd
   \advance \grcalce by -\tfactor
   \put(\grcalca, \grcalcc){\line(0,-1){\tfactor}}
   \put(\grcalca, \grcalcd){\line(-1,0){\factor}}
   \put(\grcalca, \grcalcd){\line(-3,-1){\factor}}
   \put(\grcalcb, \grcalcc){\line(0,-1){\factor}}
   \advance \grcolumn by 2}

 \newcommand{\glcm}{
   \grcalca = \grcolumn
   \multiply \grcalca by \factor
   \advance \grcalca by \hfactor
   \grcalcb = \grcalca
   \advance \grcalcb by \factor
   \grcalcc = \grrow
   \advance \grcalcc by -1
   \multiply \grcalcc by \factor
   \grcalcd = \grcalcc
   \advance \grcalcd by \tfactor
   \grcalce = \grcalcd
   \advance \grcalce by \tfactor
   \put(\grcalca, \grcalcc){\line(0,1){\tfactor}}
   \put(\grcalca, \grcalcd){\line(1,0){\factor}}
   \put(\grcalca, \grcalcd){\line(3,1){\factor}}
   \put(\grcalcb, \grcalcc){\line(0,1){\factor}}
   \advance \grcolumn by 2}

 \newcommand{\grcm}{
   \grcalcb = \grcolumn
   \multiply \grcalcb by \factor
   \advance \grcalcb by \hfactor
   \grcalca = \grcalcb
   \advance \grcalca by \factor
   \grcalcc = \grrow
   \advance \grcalcc by -1
   \multiply \grcalcc by \factor
   \grcalcd = \grcalcc
   \advance \grcalcd by \tfactor
   \grcalce = \grcalcd
   \advance \grcalce by \tfactor
   \put(\grcalca, \grcalcc){\line(0,1){\tfactor}}
   \put(\grcalca, \grcalcd){\line(-1,0){\factor}}
   \put(\grcalca, \grcalcd){\line(-3,1){\factor}}
   \put(\grcalcb, \grcalcc){\line(0,1){\factor}}
   \advance \grcolumn by 2}

 \newcommand{\gwmu}[1]{
   \grcalca = \grcolumn
   \multiply \grcalca by \factor
   \grcalcd = \hfactor
   \multiply \grcalcd by #1
   \advance \grcalca by \grcalcd
   \grcalcb = \grrow
   \multiply \grcalcb by \factor
   \grcalcc = \factor
   \advance \grcalcc by \hfactor
   \grcalcd = #1
   \advance \grcalcd by -1
   \multiply \grcalcd by \factor
   \put(\grcalca,\grcalcb){\oval(\grcalcd,\grcalcc)[b]}
   \advance \grcalcb by -\hfactor
   \advance \grcalcb by -\qfactor
   \put(\grcalca,\grcalcb) {\line(0,-1){\qfactor}}
   \advance \grcolumn by #1}

 \newcommand{\gwcm}[1]{
   \grcalca = \grcolumn
   \multiply \grcalca by \factor
   \grcalcd = \hfactor
   \multiply \grcalcd by #1
   \advance \grcalca by \grcalcd
   \grcalcb = \grrow
   \advance \grcalcb by -1
   \multiply \grcalcb by \factor
   \grcalcc = \factor
   \advance \grcalcc by \hfactor
   \grcalcd = #1
   \advance \grcalcd by -1
   \multiply \grcalcd by \factor
   \put(\grcalca,\grcalcb){\oval(\grcalcd,\grcalcc)[t]}
   \advance \grcalcb by \factor
   \put(\grcalca,\grcalcb) {\line(0,-1){\qfactor}}
   \advance \grcolumn by #1}

 \newcommand{\gwmuc}[1]{
   \grcalca = \grcolumn
   \multiply \grcalca by \factor
   \advance \grcalca by \hfactor
   \grcalcb = \grrow
   \multiply \grcalcb by \factor
   \grcalcc = #1
   \advance \grcalcc by -1
   \multiply \grcalcc by \factor
   \put(\grcalca,\grcalcb){\line(1,0){\grcalcc}}
   \advance \grcalca by -\hfactor
   \grcalcd = \hfactor
   \multiply \grcalcd by #1
   \advance \grcalca by \grcalcd
   \grcalcc = \factor
   \advance \grcalcc by \hfactor
   \grcalcd = #1
   \advance \grcalcd by -1
   \multiply \grcalcd by \factor
   \put(\grcalca,\grcalcb){\oval(\grcalcd,\grcalcc)[b]}
   \advance \grcalcb by -\hfactor
   \advance \grcalcb by -\qfactor
   \put(\grcalca,\grcalcb) {\line(0,-1){\qfactor}}
   \advance \grcolumn by #1}

 \newcommand{\gwcmc}[1]{
   \grcalca = \grcolumn
   \multiply \grcalca by \factor
   \advance \grcalca by \hfactor
   \grcalcb = \grrow
   \multiply \grcalcb by \factor
   \advance \grcalcb by -\factor
   \grcalcc = #1
   \advance \grcalcc by -1
   \multiply \grcalcc by \factor
   \put(\grcalca,\grcalcb){\line(1,0){\grcalcc}}
   \grcalcd = #1
   \advance \grcalcd by -1
   \multiply \grcalcd by \hfactor
   \advance \grcalca by \grcalcd
   \grcalcc = \factor
   \advance \grcalcc by \hfactor
   \grcalcd = #1
   \advance \grcalcd by -1
   \multiply \grcalcd by \factor
   \put(\grcalca,\grcalcb){\oval(\grcalcd,\grcalcc)[t]}
   \advance \grcalcb by \factor
   \put(\grcalca,\grcalcb) {\line(0,-1){\qfactor}}
   \advance \grcolumn by #1}

 \newcommand{\gev}{
   \grcalca = \grcolumn
   \advance \grcalca by 1
   \multiply \grcalca by \factor
   \grcalcb = \grrow
   \multiply \grcalcb by \factor
   \grcalcc = \factor
   \advance \grcalcc by \hfactor
   \put(\grcalca,\grcalcb){\oval(\factor,\grcalcc)[b]}
   \advance \grcolumn by 2}

 \newcommand{\gdb}{
   \grcalca = \grcolumn
   \advance \grcalca by 1
   \multiply \grcalca by \factor
   \grcalcb = \grrow
   \advance \grcalcb by -1
   \multiply \grcalcb by \factor
   \grcalcc = \factor
   \advance \grcalcc by \hfactor
   \put(\grcalca,\grcalcb){\oval(\factor,\grcalcc)[t]}
   \advance \grcolumn by 2}

 \newcommand{\gwev}[1]{
   \grcalca = \grcolumn
   \multiply \grcalca by \factor
   \grcalcd = \hfactor
   \multiply \grcalcd by #1
   \advance \grcalca by \grcalcd
   \grcalcb = \grrow
   \multiply \grcalcb by \factor
   \grcalcc = \factor
   \advance \grcalcc by \hfactor
   \grcalcd = #1
   \advance \grcalcd by -1
   \multiply \grcalcd by \factor
   \put(\grcalca,\grcalcb){\oval(\grcalcd,\grcalcc)[b]}
   \advance \grcolumn by #1}

 \newcommand{\gwdb}[1]{
   \grcalca = \grcolumn
   \multiply \grcalca by \factor
   \grcalcd = \hfactor
   \multiply \grcalcd by #1
   \advance \grcalca by \grcalcd
   \grcalcb = \grrow
   \advance \grcalcb by -1
   \multiply \grcalcb by \factor
   \grcalcc = \factor
   \advance \grcalcc by \hfactor
   \grcalcd = #1
   \advance \grcalcd by -1
   \multiply \grcalcd by \factor
   \put(\grcalca,\grcalcb){\oval(\grcalcd,\grcalcc)[t]}
   \advance \grcolumn by #1}

 \newcommand{\gbr}{
   \grcalca = \grcolumn
   \multiply \grcalca by \factor
   \advance \grcalca by \hfactor
   \grcalcb = \grcalca
   \advance \grcalcb by \hfactor
   \grcalcc = \grcalca
   \advance \grcalcc by \factor
   \grcalcd = \grrow
   \multiply \grcalcd by \factor
   \grcalce = \grcalcd
   \advance \grcalce by -\tfactor
   \grcalcf = \grcalcd
   \advance \grcalcf by -\hfactor
   \grcalcg = \grcalce
   \advance \grcalcg by -\tfactor
   \grcalch = \grcalcd
   \advance \grcalch by -\factor
   \qbezier(\grcalca,\grcalcd)(\grcalca,\grcalce)(\grcalcb,\grcalcf)
   \qbezier(\grcalcb,\grcalcf)(\grcalcc,\grcalcg)(\grcalcc,\grcalch)
   \advance \grcalcf by -\dfactor
   \advance \grcalcb by -\sfactor
   \qbezier(\grcalca,\grcalch)(\grcalca,\grcalcg)(\grcalcb,\grcalcf)
   \advance \grcalcf by \sfactor
   \advance \grcalcb by \tfactor
   \qbezier(\grcalcc,\grcalcd)(\grcalcc,\grcalce)(\grcalcb,\grcalcf)
   \advance \grcolumn by 2}

 \newcommand{\gibr}{
   \grcalca = \grcolumn
   \multiply \grcalca by \factor
   \advance \grcalca by \hfactor
   \grcalcb = \grcalca
   \advance \grcalcb by \hfactor
   \grcalcc = \grcalca
   \advance \grcalcc by \factor
   \grcalcd = \grrow
   \multiply \grcalcd by \factor
   \grcalce = \grcalcd
   \advance \grcalce by -\tfactor
   \grcalcf = \grcalcd
   \advance \grcalcf by -\hfactor
   \grcalcg = \grcalce
   \advance \grcalcg by -\tfactor
   \grcalch = \grcalcd
   \advance \grcalch by -\factor
   \qbezier(\grcalcc,\grcalcd)(\grcalcc,\grcalce)(\grcalcb,\grcalcf)
   \qbezier(\grcalcb,\grcalcf)(\grcalca,\grcalcg)(\grcalca,\grcalch)
   \advance \grcalcf by -\dfactor
   \advance \grcalcb by \sfactor
   \qbezier(\grcalcc,\grcalch)(\grcalcc,\grcalcg)(\grcalcb,\grcalcf)
   \advance \grcalcf by \sfactor
   \advance \grcalcb by -\tfactor
   \qbezier(\grcalca,\grcalcd)(\grcalca,\grcalce)(\grcalcb,\grcalcf)
   \advance \grcolumn by 2}

 \newcommand{\gbrc}{
   \grcalca = \grcolumn
   \multiply \grcalca by \factor
   \advance \grcalca by \hfactor
   \grcalcb = \grcalca
   \advance \grcalcb by \hfactor
   \grcalcc = \grcalca
   \advance \grcalcc by \factor
   \grcalcd = \grrow
   \multiply \grcalcd by \factor
   \grcalce = \grcalcd
   \advance \grcalce by -\tfactor
   \grcalcf = \grcalcd
   \advance \grcalcf by -\hfactor
   \grcalcg = \grcalce
   \advance \grcalcg by -\tfactor
   \grcalch = \grcalcd
   \advance \grcalch by -\factor
   \put(\grcalcb,\grcalcf){\circle{\hfactor}}
   \qbezier(\grcalca,\grcalcd)(\grcalca,\grcalce)(\grcalcb,\grcalcf)
   \qbezier(\grcalcb,\grcalcf)(\grcalcc,\grcalcg)(\grcalcc,\grcalch)
   \advance \grcalcf by -\dfactor
   \advance \grcalcb by -\sfactor
   \qbezier(\grcalca,\grcalch)(\grcalca,\grcalcg)(\grcalcb,\grcalcf)
   \advance \grcalcf by \sfactor
   \advance \grcalcb by \tfactor
   \qbezier(\grcalcc,\grcalcd)(\grcalcc,\grcalce)(\grcalcb,\grcalcf)
   \advance \grcolumn by 2}

 \newcommand{\gibrc}{
   \grcalca = \grcolumn
   \multiply \grcalca by \factor
   \advance \grcalca by \hfactor
   \grcalcb = \grcalca
   \advance \grcalcb by \hfactor
   \grcalcc = \grcalca
   \advance \grcalcc by \factor
   \grcalcd = \grrow
   \multiply \grcalcd by \factor
   \grcalce = \grcalcd
   \advance \grcalce by -\tfactor
   \grcalcf = \grcalcd
   \advance \grcalcf by -\hfactor
   \grcalcg = \grcalce
   \advance \grcalcg by -\tfactor
   \grcalch = \grcalcd
   \advance \grcalch by -\factor
   \put(\grcalcb,\grcalcf){\circle{\hfactor}}
   \qbezier(\grcalcc,\grcalcd)(\grcalcc,\grcalce)(\grcalcb,\grcalcf)
   \qbezier(\grcalcb,\grcalcf)(\grcalca,\grcalcg)(\grcalca,\grcalch)
   \advance \grcalcf by -\dfactor
   \advance \grcalcb by \sfactor
   \qbezier(\grcalcc,\grcalch)(\grcalcc,\grcalcg)(\grcalcb,\grcalcf)
   \advance \grcalcf by \sfactor
   \advance \grcalcb by -\tfactor
   \qbezier(\grcalca,\grcalcd)(\grcalca,\grcalce)(\grcalcb,\grcalcf)
   \advance \grcolumn by 2}

 \newcommand{\gu}[1]{
   \grcalca = \grcolumn
   \multiply \grcalca by \factor
   \grcalcd = \hfactor
   \multiply \grcalcd by #1
   \advance \grcalca by \grcalcd
   \grcalcb = \grrow
   \advance \grcalcb by -1
   \multiply \grcalcb by \factor
   \put(\grcalca,\grcalcb) {\line(0,1){\hfactor}}
   \advance \grcalcb by \hfactor
   \put(\grcalca,\grcalcb) {\circle*{3}}
   \advance \grcolumn by #1}

 \newcommand{\gcu}[1]{
   \grcalca = \grcolumn
   \multiply \grcalca by \factor
   \grcalcd = \hfactor
   \multiply \grcalcd by #1
   \advance \grcalca by \grcalcd
   \grcalcb = \grrow
   \multiply \grcalcb by \factor
   \put(\grcalca,\grcalcb) {\line(0,-1){\hfactor}}
   \advance \grcalcb by -\hfactor
   \put(\grcalca,\grcalcb) {\circle*{3}}
   \advance \grcolumn by #1}

 \newcommand{\gmp}[1]{
   \grcalca = \grcolumn
   \multiply \grcalca by \factor
   \advance \grcalca by \hfactor
   \grcalcb = \grrow
   \multiply \grcalcb by \factor
   \put(\grcalca,\grcalcb) {\line(0,-1){\dfactor}}
   \advance \grcalcb by -\factor
   \put(\grcalca,\grcalcb) {\line(0,1){\dfactor}}
   \advance \grcalcb by \hfactor
   \grcalcc = \factor
   \advance \grcalcc by -\qfactor
   \put(\grcalca,\grcalcb) {\circle{\grcalcc}}
   \put(\grcalca,\grcalcb) {\makebox(0,0){$\scriptstyle #1$}}
   \advance \grcolumn by 1}

 \newcommand{\gbmp}[1]{
   \grcalca = \grcolumn
   \multiply \grcalca by \factor
   \advance \grcalca by \hfactor
   \grcalcb = \grrow
   \multiply \grcalcb by \factor
   \put(\grcalca,\grcalcb) {\line(0,-1){\dfactor}}
   \advance \grcalcb by -\factor
   \put(\grcalca,\grcalcb) {\line(0,1){\dfactor}}
   \advance \grcalca by -\hfactor
   \advance \grcalca by \dfactor
   \advance \grcalcb by \dfactor
   \grcalcc = \factor
   \advance \grcalcc by -\sfactor
   \put(\grcalca,\grcalcb) {\framebox(\grcalcc,\grcalcc){$\scriptstyle #1$}}
   \advance \grcolumn by 1}

 \newcommand{\gbmpt}[1]{
   \grcalca = \grcolumn
   \multiply \grcalca by \factor
   \advance \grcalca by \hfactor
   \grcalcb = \grrow
   \multiply \grcalcb by \factor
   \put(\grcalca,\grcalcb) {\line(0,-1){\dfactor}}
   \advance \grcalcb by -\factor
   \advance \grcalca by -\hfactor
   \advance \grcalca by \dfactor
   \advance \grcalcb by \dfactor
   \grcalcc = \factor
   \advance \grcalcc by -\sfactor
   \put(\grcalca,\grcalcb) {\framebox(\grcalcc,\grcalcc){$\scriptstyle #1$}}
   \advance \grcolumn by 1}

 \newcommand{\gbmpb}[1]{
   \grcalca = \grcolumn
   \multiply \grcalca by \factor
   \advance \grcalca by \hfactor
   \grcalcb = \grrow
   \multiply \grcalcb by \factor
   \advance \grcalcb by -\factor
   \put(\grcalca,\grcalcb) {\line(0,1){\dfactor}}
   \advance \grcalca by -\hfactor
   \advance \grcalca by \dfactor
   \advance \grcalcb by \dfactor
   \grcalcc = \factor
   \advance \grcalcc by -\sfactor
   \put(\grcalca,\grcalcb) {\framebox(\grcalcc,\grcalcc){$\scriptstyle #1$}}
   \advance \grcolumn by 1}

 \newcommand{\gbmpn}[1]{
   \grcalca = \grcolumn
   \multiply \grcalca by \factor
   \advance \grcalca by \hfactor
   \grcalcb = \grrow
   \multiply \grcalcb by \factor
   \advance \grcalcb by -\factor
   \advance \grcalca by -\hfactor
   \advance \grcalca by \dfactor
   \advance \grcalcb by \dfactor
   \grcalcc = \factor
   \advance \grcalcc by -\sfactor
   \put(\grcalca,\grcalcb) {\framebox(\grcalcc,\grcalcc){$\scriptstyle #1$}}
   \advance \grcolumn by 1}

 \newcommand{\glmptb}{
   \grcalca = \grcolumn
   \multiply \grcalca by \factor
   \advance \grcalca by \hfactor
   \grcalcb = \grrow
   \multiply \grcalcb by \factor
   \put(\grcalca,\grcalcb) {\line(0,-1){\dfactor}}
   \advance \grcalcb by -\factor
   \put(\grcalca,\grcalcb) {\line(0,1){\dfactor}}
   \advance \grcalca by -\hfactor
   \advance \grcalca by \dfactor
   \advance \grcalcb by \dfactor
   \put(\grcalca,\grcalcb) {\line(1,0){\factor}}
   \advance \grcalcb by \factor
   \advance \grcalcb by -\sfactor
   \put(\grcalca,\grcalcb) {\line(1,0){\factor}}
   \grcalcc = \factor
   \advance \grcalcc by -\sfactor
   \put(\grcalca,\grcalcb) {\line(0,-1){\grcalcc}}
   \advance \grcolumn by 1}

 \newcommand{\glmpt}{
   \grcalca = \grcolumn
   \multiply \grcalca by \factor
   \advance \grcalca by \hfactor
   \grcalcb = \grrow
   \multiply \grcalcb by \factor
   \put(\grcalca,\grcalcb) {\line(0,-1){\dfactor}}
   \advance \grcalca by -\hfactor
   \advance \grcalca by \dfactor
   \advance \grcalcb by -\dfactor
   \put(\grcalca,\grcalcb) {\line(1,0){\factor}}
   \advance \grcalcb by -\factor
   \advance \grcalcb by \sfactor
   \put(\grcalca,\grcalcb) {\line(1,0){\factor}}
   \grcalcc = \factor
   \advance \grcalcc by -\sfactor
   \put(\grcalca,\grcalcb) {\line(0,1){\grcalcc}}
   \advance \grcolumn by 1}

 \newcommand{\glmpb}{
   \grcalca = \grcolumn
   \multiply \grcalca by \factor
   \advance \grcalca by \hfactor
   \grcalcb = \grrow
   \multiply \grcalcb by \factor
   \advance \grcalcb by -\factor
   \put(\grcalca,\grcalcb) {\line(0,1){\dfactor}}
   \advance \grcalca by -\hfactor
   \advance \grcalca by \dfactor
   \advance \grcalcb by \dfactor
   \put(\grcalca,\grcalcb) {\line(1,0){\factor}}
   \advance \grcalcb by \factor
   \advance \grcalcb by -\sfactor
   \put(\grcalca,\grcalcb) {\line(1,0){\factor}}
   \grcalcc = \factor
   \advance \grcalcc by -\sfactor
   \put(\grcalca,\grcalcb) {\line(0,-1){\grcalcc}}
   \advance \grcolumn by 1}

 \newcommand{\glmp}{
   \grcalca = \grcolumn
   \multiply \grcalca by \factor
   \advance \grcalca by \dfactor
   \grcalcb = \grrow
   \multiply \grcalcb by \factor
   \advance \grcalcb by -\dfactor
   \put(\grcalca,\grcalcb) {\line(1,0){\factor}}
   \advance \grcalcb by -\factor
   \advance \grcalcb by \sfactor
   \put(\grcalca,\grcalcb) {\line(1,0){\factor}}
   \grcalcc = \factor
   \advance \grcalcc by -\sfactor
   \put(\grcalca,\grcalcb) {\line(0,1){\grcalcc}}
   \advance \grcolumn by 1}

 \newcommand{\gcmptb}{
   \grcalca = \grcolumn
   \multiply \grcalca by \factor
   \advance \grcalca by \hfactor
   \grcalcb = \grrow
   \multiply \grcalcb by \factor
   \put(\grcalca,\grcalcb) {\line(0,-1){\dfactor}}
   \advance \grcalcb by -\factor
   \put(\grcalca,\grcalcb) {\line(0,1){\dfactor}}
   \advance \grcalca by -\hfactor
   \advance \grcalcb by \dfactor
   \put(\grcalca,\grcalcb) {\line(1,0){\factor}}
   \advance \grcalcb by \factor
   \advance \grcalcb by -\sfactor
   \put(\grcalca,\grcalcb) {\line(1,0){\factor}}
   \advance \grcolumn by 1}

 \newcommand{\gcmpt}{
   \grcalca = \grcolumn
   \multiply \grcalca by \factor
   \advance \grcalca by \hfactor
   \grcalcb = \grrow
   \multiply \grcalcb by \factor
   \put(\grcalca,\grcalcb) {\line(0,-1){\dfactor}}
   \advance \grcalcb by -\factor
   \advance \grcalca by -\hfactor
   \advance \grcalcb by \dfactor
   \put(\grcalca,\grcalcb) {\line(1,0){\factor}}
   \advance \grcalcb by \factor
   \advance \grcalcb by -\sfactor
   \put(\grcalca,\grcalcb) {\line(1,0){\factor}}
   \advance \grcolumn by 1}

 \newcommand{\gcmpb}{
   \grcalca = \grcolumn
   \multiply \grcalca by \factor
   \advance \grcalca by \hfactor
   \grcalcb = \grrow
   \multiply \grcalcb by \factor
   \advance \grcalcb by -\factor
   \put(\grcalca,\grcalcb) {\line(0,1){\dfactor}}
   \advance \grcalca by -\hfactor
   \advance \grcalcb by \dfactor
   \put(\grcalca,\grcalcb) {\line(1,0){\factor}}
   \advance \grcalcb by \factor
   \advance \grcalcb by -\sfactor
   \put(\grcalca,\grcalcb) {\line(1,0){\factor}}
   \advance \grcolumn by 1}

 \newcommand{\gcmp}{
   \grcalca = \grcolumn
   \multiply \grcalca by \factor
   \grcalcb = \grrow
   \multiply \grcalcb by \factor
   \advance \grcalcb by -\factor
   \advance \grcalcb by \dfactor
   \put(\grcalca,\grcalcb) {\line(1,0){\factor}}
   \advance \grcalcb by \factor
   \advance \grcalcb by -\sfactor
   \put(\grcalca,\grcalcb) {\line(1,0){\factor}}
   \advance \grcolumn by 1}

 \newcommand{\grmptb}{
   \grcalca = \grcolumn
   \multiply \grcalca by \factor
   \advance \grcalca by \hfactor
   \grcalcb = \grrow
   \multiply \grcalcb by \factor
   \put(\grcalca,\grcalcb) {\line(0,-1){\dfactor}}
   \advance \grcalcb by -\factor
   \put(\grcalca,\grcalcb) {\line(0,1){\dfactor}}
   \advance \grcalca by \hfactor
   \advance \grcalca by -\dfactor
   \advance \grcalcb by \dfactor
   \put(\grcalca,\grcalcb) {\line(-1,0){\factor}}
   \advance \grcalcb by \factor
   \advance \grcalcb by -\sfactor
   \put(\grcalca,\grcalcb) {\line(-1,0){\factor}}
   \grcalcc = \factor
   \advance \grcalcc by -\sfactor
   \put(\grcalca,\grcalcb) {\line(0,-1){\grcalcc}}
   \advance \grcolumn by 1}

 \newcommand{\grmpt}{
   \grcalca = \grcolumn
   \multiply \grcalca by \factor
   \advance \grcalca by \hfactor
   \grcalcb = \grrow
   \multiply \grcalcb by \factor
   \put(\grcalca,\grcalcb) {\line(0,-1){\dfactor}}
   \advance \grcalca by \hfactor
   \advance \grcalca by -\dfactor
   \advance \grcalcb by -\dfactor
   \put(\grcalca,\grcalcb) {\line(-1,0){\factor}}
   \advance \grcalcb by -\factor
   \advance \grcalcb by \sfactor
   \put(\grcalca,\grcalcb) {\line(-1,0){\factor}}
   \grcalcc = \factor
   \advance \grcalcc by -\sfactor
   \put(\grcalca,\grcalcb) {\line(0,1){\grcalcc}}
   \advance \grcolumn by 1}

 \newcommand{\grmpb}{
   \grcalca = \grcolumn
   \multiply \grcalca by \factor
   \advance \grcalca by \hfactor
   \grcalcb = \grrow
   \multiply \grcalcb by \factor
   \advance \grcalcb by -\factor
   \put(\grcalca,\grcalcb) {\line(0,1){\dfactor}}
   \advance \grcalca by \hfactor
   \advance \grcalca by -\dfactor
   \advance \grcalcb by \dfactor
   \put(\grcalca,\grcalcb) {\line(-1,0){\factor}}
   \advance \grcalcb by \factor
   \advance \grcalcb by -\sfactor
   \put(\grcalca,\grcalcb) {\line(-1,0){\factor}}
   \grcalcc = \factor
   \advance \grcalcc by -\sfactor
   \put(\grcalca,\grcalcb) {\line(0,-1){\grcalcc}}
   \advance \grcolumn by 1}

 \newcommand{\grmp}{
   \grcalca = \grcolumn
   \multiply \grcalca by \factor
   \advance \grcalca by \factor
   \advance \grcalca by -\dfactor
   \grcalcb = \grrow
   \multiply \grcalcb by \factor
   \advance \grcalcb by -\dfactor
   \put(\grcalca,\grcalcb) {\line(-1,0){\factor}}
   \advance \grcalcb by -\factor
   \advance \grcalcb by \sfactor
   \put(\grcalca,\grcalcb) {\line(-1,0){\factor}}
   \grcalcc = \factor
   \advance \grcalcc by -\sfactor
   \put(\grcalca,\grcalcb) {\line(0,1){\grcalcc}}
   \advance \grcolumn by 1}
\newcommand{\gsy}{
   \grcalca = \grcolumn
   \multiply \grcalca by \factor
   \advance \grcalca by \hfactor
   \grcalcb = \grcalca
   \advance \grcalcb by \hfactor
   \grcalcc = \grcalca
   \advance \grcalcc by \factor
   \grcalcd = \grrow
   \multiply \grcalcd by \factor
   \grcalce = \grcalcd
   \advance \grcalce by -\tfactor
   \grcalcf = \grcalcd
   \advance \grcalcf by -\hfactor
   \grcalcg = \grcalce
   \advance \grcalcg by -\tfactor
   \grcalch = \grcalcd
   \advance \grcalch by -\factor
   \qbezier(\grcalcc,\grcalcd)(\grcalcc,\grcalce)(\grcalcb,\grcalcf)
   \qbezier(\grcalcb,\grcalcf)(\grcalca,\grcalcg)(\grcalca,\grcalch)
   \advance \grcalcf by -\dfactor
   \advance \grcalcb by \sfactor
   \qbezier(\grcalcc,\grcalch)(\grcalcc,\grcalcg)(\grcalcb,\grcalcf)
   \qbezier(\grcalca,\grcalcd)(\grcalca,\grcalce)(\grcalcb,\grcalcf)
   \advance \grcolumn by 2}

 \newcommand{\gwmuh}[3]{
   \grcalca = \grcolumn
   \multiply \grcalca by \factor
   \grcalcb = #2
   \advance \grcalcb by #3
   \multiply \grcalcb by \qfactor
   \advance \grcalca by \grcalcb
   \grcalcb = \grrow
   \multiply \grcalcb by \factor
   \grcalcc = #3
   \advance \grcalcc by -#2
   \multiply \grcalcc by \hfactor
   \grcalcd = \factor
   \advance \grcalcd by \hfactor
   \put(\grcalca,\grcalcb){\oval(\grcalcc,\grcalcd)[b]}
   \grcalca = \grcolumn
   \multiply \grcalca by \factor
   \grcalcc = #1
   \multiply \grcalcc by \hfactor
   \advance \grcalca by \grcalcc
   \advance \grcalcb by -\hfactor
   \advance \grcalcb by -\qfactor
   \put(\grcalca,\grcalcb) {\line(0,-1){\qfactor}}
   \advance \grcolumn by #1}

 \newcommand{\gwcmh}[3]{
   \grcalca = \grcolumn
   \multiply \grcalca by \factor
   \grcalcb = #2
   \advance \grcalcb by #3
   \multiply \grcalcb by \qfactor
   \advance \grcalca by \grcalcb
   \grcalcb = \grrow
   \advance \grcalcb by -1
   \multiply \grcalcb by \factor
   \grcalcc = #3
   \advance \grcalcc by -#2
   \multiply \grcalcc by \hfactor
   \grcalcd = \factor
   \advance \grcalcd by \hfactor
   \put(\grcalca,\grcalcb){\oval(\grcalcc,\grcalcd)[t]}
   \grcalca = \grcolumn
   \multiply \grcalca by \factor
   \grcalcc = #1
   \multiply \grcalcc by \hfactor
   \advance \grcalca by \grcalcc
   \advance \grcalcb by \factor
   \put(\grcalca,\grcalcb) {\line(0,-1){\qfactor}}
   \advance \grcolumn by #1}

 \newcommand{\gsbox}[1]{
   \grcalca = \grcolumn
   \multiply \grcalca by \factor
   \grcalcb = \grrow
   \multiply \grcalcb by \factor
   \advance \grcalcb by -\factor
   \grcalcc = #1
   \multiply \grcalcc by \factor
   \grcalcd = \factor
   \put(\grcalca,\grcalcb){\framebox(\grcalcc,\grcalcd){}}}

   \newcommand{\gbox}[2]{
   \grcalca = \grcolumn
   \multiply \grcalca by \factor
   \grcalcb = \grrow
   \multiply \grcalcb by \factor
   \advance \grcalcb by -\factor
   \grcalcc = #1
   \multiply \grcalcc by \factor
   \grcalcd = #2
   \multiply \grcalcd by \factor
   \put(\grcalca,\grcalcb){\framebox(\grcalcc,\grcalcd){}}}

\allowdisplaybreaks[4]

\newcommand{\linea}{\gcl{1}}
\newcommand{\cruce}{\gbr}
\newcommand{\etiqueta}[1]{\gnot{#1\,\,\,\,}}

\begin{document}

\title[Co-Moufang deformations of $U(\M(\alpha,\beta,\gamma))$]{Co-Moufang deformations of the universal enveloping algebra of the algebra of traceless octonions}

\author{Jos\'e M. P\'erez-Izquierdo}
\address{Departamento de Matem\'aticas y Computaci\'on\\ Universidad de La Rioja, 26004 \\ Lo\-gro\-\~no, Spain}
\email{jm.perez@unirioja.es}  
              
\author{Ivan P. Shestakov}        
\address{Instituto de Matematica e Estat{\'\i}stica, Universidade de S\~ao Paulo, Caixa postal 66281, CEP 05315-970\\ S\~ao Paulo, Brazil}
\email{ivan.shestakov@gmail.com}

\thanks{The first author thanks support from the Spanish Ministry of Science and Innovation
(MTM2010-18370-C04-03) and from FAPESP (2011/51553-4). He also thanks the Instituto de Matem\'atica e Estat{\'\i}stica/USP of S\~{a}o Paulo for its kind hospitality during his stay in 2011. The second author thanks support from CNPq, Proc. 456698/2014-0 and FAPESP, Proc. 2014/09310-5. Both authors also thank Bodo Pareigis for sharing his beautiful ``diagrams'' package.}



\begin{abstract}
By means of graphical calculus we prove that, over fields of characteristic zero, any bialgebra deformation of the universal enveloping algebra of the algebra of traceless octonions satisfying the dual of the left and right Moufang identities must be coassociative and cocommutative.
\end{abstract}
\maketitle

\section{Introduction}
\begin{center}
 \emph{Throughout this paper the base field $k$ will be assumed to be of characteristic zero}
\end{center}

The only spheres that are H-spaces are $S^0, S^1, S^3$ and $S^7$. While $S^0, S^1$ and $S^3$  are Lie groups, the product on $S^7$, inherited from the octonions, is no longer associative but satisfies the \emph{left, middle and right Moufang identities}
\begin{displaymath}
  a(x(ay)) = ((ax)a)y,\quad (ax)(ya) = (a(xy))a\quad \mbox{and}\quad  ((xa)y)a = x(a(ya)).
\end{displaymath}
A Moufang loop is a set $(M,xy)$ with a binary product $xy$ such that: 1) it satisfies any of the Moufang identities 2) the left and right multiplication operators by any $x \in M$, $L_x \colon y \mapsto xy$ and $R_x\colon y \mapsto yx$, are bijective and 3) $M$ contains a unit element for the product $xy$.

The tangent space $\m$ at the unit element of any analytic Moufang loop $M$ can be endowed with a skew-symmetric product $[a,b]$ that satisfies a generalization of the Jacobi identity:
\begin{displaymath}
  \Jac(a,b,[a,c]) = [\Jac(a,b,c),a]
\end{displaymath}
where $\Jac(a,b,c) := [[a,b],c] + [[b,c],a] +[[c,a],b]$ is the \emph{Jacobian} of $a,b$ and $c$. Such algebraic structures $(\m,[a,b])$ are called \emph{Malcev algebras} \cite{M55,S61} and they locally classify analytic Moufang loops up to isomorphism. The proof of this connection between Moufang loops and Malcev algebras initiated the development of a non-associative Lie correspondence finally achieved by Mikheev and Sabinin in 1987 \cite{MS87} with techniques of differential geometry. The tangent space of $S^7$ at the unit element is the Malcev algebra of traceless octonions. 

Given any not necessarily associative algebra $(A,xy)$, the \emph{generalized alternative nucleus} of $A$ is
\begin{displaymath}
  \Nalt(A) := \{ a \in A \mid (a,x,y) = - (x,a,y) = (x,y,a) \quad \forall_{x,y \in A}\}.
\end{displaymath}
where $(x,y,z) := (xy)z - x(yz)$ stands for the \emph{associator} of $x,y$ and $z$. The generalized alternative nucleus  is always closed under the commutator product $[x,y] := xy - yx$ and it becomes a Malcev algebra with this product. \emph{Alternative algebras} are those algebras $A$ such that, like the octonions, satisfy $\Nalt(A) = A$, so they provide natural examples of Malcev algebras. Unfortunately, it remains an open problem whether or not any Malcev algebra is a subalgebra of $(A,[x,y])$ for some alternative algebra $(A,xy)$. In \cite{PS04} it was proved that any Malcev algebra appears as a Malcev subalgebra of the generalized alternative nucleus of some algebra. This result is similar to the corresponding result for Lie algebras, namely, the Poincar\'e-Birkhoff-Witt Theorem: any Lie algebra appears as a Lie subalgebra of an associative algebra with the commutator product. As in the case of Lie algebra, the result is based on the explicit construction of a universal enveloping algebra $U(\m)$ for the Malcev algebra $(\m,[a,b])$. In the case that $(\m,[a,b])$ is a Lie algebra then $U(\m)$ turns out to be the usual associative universal enveloping algebra of $\m$.

The universal enveloping algebras of Malcev algebras are quite close to Hopf algebras. $U(\m)$ admits a cocommutative and coassociative coalgebra structure $(U(\m),\Delta,\epsilon)$ with the additional property that the comultiplication, or coproduct, $\Delta$ and the counit $\epsilon$ are homomorphisms of algebras. The Malcev algebra $\m$ is recovered as the space of primitive elements, i.e., elements $a$ such that $\Delta(a) = a \otimes 1 + 1 \otimes a$. However, in general,  $U(\m)$ is non-associative but it satisfies some Hopf version of the Moufang identities, the \emph{left, middle and right Hopf-Moufang identities}:
\begin{eqnarray*}
  \sum a_{(1)}(x(a_{(2)}y)) &=& \sum ((a_{(1)}x)a_{(2)})y, \\
  \sum (a_{(1)}x)(ya_{(2)}) &=& \sum (a_{(1)}(xy))a_{(2)} \text{ and }\\
  \sum ((xa_{(1)})y)a_{(2)} &=& \sum x(a_{(1)}(ya_{(2)})).
\end{eqnarray*}

The study of $U(\m)$ led in \cite{PS04} to a generalization for Malcev algebras of the well-known Ado-Iwasawa Theorem and the Chevalley and Konstant Theorems \cite{SZ,PS10}. Explicit formulas for the product of $U(\m)$ in some low-dimensional cases are also known \cite{BHP09,BHPTU,BU10,BT11}. Another approach to the construction of the product of $U(\m)$ was carried out in \cite{BMP}. This approach unified the connections between groups with triality and Moufang loops \cite{Gl68,Do78,Mi93,GrZa06}, and Lie algebras with triality and Malcev algebras \cite{Mi92,Gr03} by means of the notion of Hopf algebra with triality. For an account of universal enveloping algebras of generalizations of Lie and Malcev algebras see \cite{P07,MP10,MPS14}.

In \cite{MP} a study of the possible coassociative bialgebra deformations of $U(\m)$ was carried out for the Malcev algebra of traceless octonions $\m = \M(\alpha,\beta,\gamma)$ \cite{ZSSS82}. The reason for focusing on these algebras is that any central simple Malcev algebra is either a Lie algebra or isomorphic to an algebra $\M(\alpha,\beta,\gamma)$. Bialgebra deformations of $U(\m)$ should be an analogue of quantized enveloping algebras of Lie algebras in a non-associative setting. The development of quantized enveloping algebras and quantum groups during the last two decades has been spectacular so the search for bialgebra deformations of enveloping algebras of Malcev is challenging.

Given a Malcev algebra $(\m,[a,b])$ over a field $k$, a \emph{bialgebra deformation} of $U(\m)$ over the ring of formal power series $K = k[[h]]$ is a topologically free $K$-module $U_h(\m)$ endowed with four maps of $K$-modules
\begin{center}
\begin{tabular}{llll}
 \emph{(unit)} & $\iota_h \colon K \rightarrow U_h(\m)$,  $1 \mapsto 1$,& \emph{(product)} & $p_h \colon U_h(\m) \tilde{\otimes} U_h(\m) \rightarrow U_h(\m)$,\\
\emph{(counit)} & $\epsilon_h \colon U_h(\m) \rightarrow K$, &
 \emph{(coproduct)}& $\Delta_h \colon U_h(\m) \rightarrow U_h(\m) \tilde{\otimes} U_h(\m)$,
\end{tabular}
\end{center}
where $\tilde{\otimes}$ stands for the completed tensor product in the $h$-adic topology, such that
\begin{enumerate}
 \item  $(U_h(\m),\iota_h,p_h,\epsilon_h,\Delta_h)$ satisfies the axioms of bialgebra over the commutative ring $K$ (see {Remarks to Definition 4.1.3} in \cite{CP95}) but with the algebraic tensor product replaced by its completion and without assuming the coassociativity,
 \item $U_h(\m)/hU_h(\m) \cong U(\m)$ as a $k$-vector space and, with this identification,
 \item $p_h \equiv p \modspace (\modulus h)$ and $\Delta_h \equiv \Delta \modspace (\modulus h)$
\end{enumerate}
with $p$ and $\Delta$ the multiplication and comultiplication of $U(\m)$ respectively. Since $U_h(\m)$ is topologically free and $U_h(\m)/hU_h(\m) \cong U(\m)$, we can identify $U_h(\m)$ with $U(\m)[[h]]$ as a $K$-module and we will do so. The unit and counit are assumed to be the natural extensions to $U(\m)[[h]]$ of the unit and counit of $U(\m)$. The \emph{null deformation} of $U(\m)$ is obtained by extending $K$-linearly the structure maps of $U(\m)$. \emph{Trivial deformations} are those that are isomorphic to the null deformation under a $K$-linear bialgebra isomorphism which is the identity modulo $h$. To avoid confusion in this context where two multiplications, $p$ and $p_h$, and two comultiplications, $\Delta$ and $\Delta_h$, appear we will stick to the following notation:
\begin{eqnarray*}
xy := p(x \otimes y), && \sum x_{(1)} \otimes x_{(2)} := \Delta(x) \text{ for all } x \in U(\m),\\
x\bullet y := p_h(x\tilde{\otimes} y) &\text{\ and\ }&  \sum x_{[1]} \tilde{\otimes} x_{[2]} := \Delta_h(x) \text{ for all } x \in U_h(\m).
\end{eqnarray*}

Bialgebra deformations of $U(\m)$ are rather general objects so some extra restrictions are imposed either on the multiplication or on the comultiplication. Quantized universal enveloping algebras of Lie algebras are assumed to be associative and coassociative \cite{CP95}, i.e., the associativity and coassociativity are properties that one wants to preserve when deforming these structures. However, $U(\m)$ is no longer associative so, it is more appropriate to preserve the Hopf-Moufang identities.

In \cite{MP} it was proved that, in contrast to the Lie case, any coassociative but possibly non-cocommutative bialgebra deformation of $U(\M(\alpha,\beta,\gamma))$ satisfying the Hopf-Moufang identities is trivial. The coassociativity allowed an approach to the study of bialgebra deformations similar to that carried out for Lie algebras. However, since the product of $U(\m)$ is not associative this restriction on the coalgebra structure seems artificial. According to the self-dual structure of Hopf algebras, the dual version of the Hopf-Moufang identities, that we will call \emph{co-Moufang identities}, is probably the natural restriction on the coalgebra structure of any bialgebra deformation of $U(\m)$. Co-Moufang identities are satisfied, for instance, by the coordinate algebra $k[S^7]$ of the seven-dimensional sphere where they have been used to develop differential geometry \cite{KM09}.

Unfortunately, when dealing with non-coassociative coalgebras Sweedler's notation for the comultiplication turns out to be very obscure and cumbersome. Graphical calculus is preferable in this case:
\begin{displaymath}
   \gbeg{2}{3}
    \got{1}{x} \got{1}{y}   \gnl
    \gmu                    \gnl
    \gend := xy,
\quad
     \gbeg{2}{3}
    \got{2}{x}\gnl
    \gcmu                   \gnl
    \gend := \Delta(x),
\quad
    \gbeg{1}{3}
    \got{1}{x}\gnl
    \gcu{1}                 \gnl
    \gend := \epsilon(x)
\quad
    \gbeg{2}{3}
    \got{1}{x} \got{1}{y}   \gnl
    \cruce                   \gnl
    \gend := x\otimes y \mapsto y \otimes x,
\quad\mbox{and}\quad
    \gbeg{1}{3}
    \got{1}{x}\gnl
    \linea                 \gnl
    \gend := x
\end{displaymath}
Compositions of maps are written from top to bottom. For instance, the equalities
\begin{equation}
\sum \epsilon(x_{(1)})x_{(2)} = x = \sum \epsilon(x_{(2)})x_{(1)}
\end{equation}
are encoded in the following equalities of diagrams
\begin{equation}
\label{eq:counit}
\gbeg{2}{3}
\got{2}{x}              \gnl
\gcmu                   \gnl
\gcu{1}         \linea  \gnl
\gend
=
\gbeg{1}{3}
\got{1}{x}              \gnl
\linea  \gnl
\linea  \gnl
\gend
=
\gbeg{2}{3}
\got{2}{x}          \gnl
\gcmu               \gnl
\linea      \gcu{1} \gnl
\gend
\end{equation}
The left and right Hopf-Moufang identities are represented by
\begin{displaymath}
\gbeg{4}{9}
\got{2}{x} \got{1}{y} \got{1}{z}    \gnl
\gcmu       \gcl{1}     \gcl{1}     \gnl
\gcl{1}     \gbr        \gcl{1}     \gnl
\gcl{1}     \gcl{1}     \gmu        \gnl
\gcl{1}     \gcl{1}     \gcn{1}{1}{2}{1}    \gnl
\gcl{1}     \gmu                    \gnl
\gcl{1}     \gcn{1}{1}{2}{1}        \gnl
\gmu                                \gnl
\gend
=
\gbeg{4}{9}
\got{2}{x} \got{1}{y} \got{1}{z}    \gnl
\gcmu       \linea      \linea      \gnl
\linea      \cruce      \linea      \gnl
\gmu        \linea      \linea      \gnl
\gcn{2}{1}{2}{3}        \linea      \linea  \gnl
\gvac{1}    \gmu        \linea      \gnl
\gvac{1}    \gcn{2}{1}{2}{3}        \linea  \gnl
\gvac{2}    \gmu         \gnl
\gend
\quad\quad\mbox{and}\quad\quad
\gbeg{4}{9}
\got{1}{x} \got{1}{y} \got{2}{z}    \gnl
\linea \linea \gcmu \gnl
\linea \cruce \linea \gnl
\linea \linea \gmu \gnl
\linea \linea \gcn{1}{1}{2}{1}\gnl
\linea \gmu \gnl
\linea \gcn{1}{1}{2}{1}\gnl
\gmu \gnl
\gend
=
\gbeg{4}{9}
\got{1}{x} \got{1}{y} \got{2}{z}    \gnl
\linea \linea \gcmu \gnl
\linea \cruce \linea \gnl
\gmu \linea \linea \gnl
\gcn{2}{1}{2}{3} \linea \linea \gnl
\gvac{1}\gmu \linea \gnl
\gcn{3}{1}{4}{5} \linea\gnl
\gvac{2}\gmu \gnl
\gend
\end{displaymath}
 The left and right co-Moufang identities are obtained by turning upside-down the diagrams corresponding to the left and right Moufang identities:
\begin{displaymath}
    \mbox{(left co-Moufang)}
        \gbeg{4}{9}
        \gcmu \gnl
        \gcl{1} \gcn{1}{1}{1}{2}        \gnl
        \gcl{1} \gcmu                   \gnl
        \gcl{1} \gcl{1} \gcn{1}{1}{1}{2}\gnl
        \gcl{1} \gcl{1} \gcmu           \gnl
        \gcl{1} \gbr    \gcl{1}         \gnl
        \gmu    \gcl{1} \gcl{1}         \gnl
        \gend
        =
        \gbeg{4}{9}
        \gvac{2}    \gcmu       \gnl
        \gcn{3}{1}{5}{4}        \gcl{1}     \gnl
        \gvac{1}    \gcmu       \gcl{1}     \gnl
        \gcn{2}{1}{3}{2}        \gcl{1}     \gcl{1} \gnl
        \gcmu       \gcl{1}     \gcl{1}     \gnl
        \gcl{1}     \gbr        \gcl{1}     \gnl
        \gmu        \gcl{1}     \gcl{1}     \gnl
        \gend
        \quad\quad
        \mbox{(right co-Moufang)}
        \gbeg{4}{9}
        \gcmu \gnl
        \gcl{1} \gcn{1}{1}{1}{2}        \gnl
        \gcl{1} \gcmu                   \gnl
        \gcl{1} \gcl{1} \gcn{1}{1}{1}{2}\gnl
        \gcl{1} \gcl{1} \gcmu           \gnl
        \gcl{1} \gbr    \gcl{1}         \gnl
        \gcl{1} \gcl{1} \gmu            \gnl
        \gend
        =
        \gbeg{4}{9}
        \gvac{2}    \gcmu       \gnl
        \gcn{3}{1}{5}{4}        \gcl{1}     \gnl
        \gvac{1}    \gcmu       \gcl{1}     \gnl
        \gcn{2}{1}{3}{2}        \gcl{1}     \gcl{1} \gnl
        \gcmu       \gcl{1}     \gcl{1}     \gnl
        \gcl{1}     \gbr        \gcl{1}     \gnl
        \gcl{1}     \gcl{1}     \gmu        \gnl
        \gend
\end{displaymath}
In other words \cite{KM09},
\begin{equation}
      \sum x_{(1)}x_{(2)(2)(1)} \otimes x_{(2)(1)} \otimes x_{(2)(2)(2)} =
      \sum x_{(1)(1)(1)} x_{(1)(2)} \otimes x_{(1)(1)(2)} \otimes x_{(2)}
\end{equation}
and
\begin{equation}
      \sum x_{(1)} \otimes x_{(2)(2)(1)} \otimes x_{(2)(1)} x_{(2)(2)(2)} =
      \sum x_{(1)(1)(1)} \otimes x_{(1)(2)} \otimes x_{(1)(1)(2)}x_{(2)}.
\end{equation}

In this paper we will prove that any bialgebra deformation $U_h(\m)$ of $U(\m)$, where $\m$ is a finite-dimensional central simple Malcev algebra, satisfying the left and right co-Moufang identities is coassociative. In the  case that $\m = \M(\alpha,\beta,\gamma)$ these deformations are also cocommutative. Hence, any bialgebra deformation of $U(\M(\alpha,\beta,\gamma))$ satisfying the left and right Moufang and co-Moufang identities is trivial. This result reveals an extraordinary rigidity of the universal enveloping algebra of non-Lie tangent algebras.

\section{Cocommutativity}
Some consequences of the co-Moufang identities will be frequently used in our arguments, where we will always assume that $\m$ is a Malcev algebra. We collect them in the following lemma.

\begin{lemma}\label{lem:basic}
  Any bialgebra deformation of $U(\m)$ satisfying the left and right co-Moufang identities also satisfies
  \begin{displaymath}
    \begin{array}{cccccc}
      \mbox{1)} &
            \gbeg{3}{4}
            \gcmu                               \gnl
            \linea      \gcn{1}{1}{1}{2}        \gnl
            \linea      \gcmu                   \gnl
            \gmu        \linea                  \gnl
            \gend
            =
            \gbeg{3}{4}
            \gvac{1}            \gcmu       \gnl
            \gcn{2}{1}{3}{2}    \linea      \gnl
            \gcmu               \linea      \gnl
            \gmu                \linea      \gnl
            \gend
            \quad\quad &
      \mbox{2)} &
            \gbeg{3}{4}
            \gvac{1}            \gcmu       \gnl
            \gcn{2}{1}{3}{2}    \linea      \gnl
            \gcmu               \linea      \gnl
            \linea              \gmu        \gnl
            \gend
            =
            \gbeg{3}{4}
            \gcmu       \gnl
            \linea      \gcn{1}{1}{1}{2}        \gnl
            \linea      \gcmu                   \gnl
            \linea      \gmu                    \gnl
            \gend
            &
      \mbox{3)} &
            \gbeg{3}{5}
            \gcmu       \gnl
            \linea      \gcn{1}{1}{1}{2}        \gnl
            \linea      \gcmu                   \gnl
            \linea      \cruce                  \gnl
            \gmu        \linea                  \gnl
            \gend
            =
            \gbeg{3}{5}
            \gvac{1}            \gcmu       \gnl
            \gcn{2}{1}{3}{2}    \linea      \gnl
            \gcmu               \linea      \gnl
            \linea              \cruce      \gnl
            \gmu                \linea      \gnl
            \gend
            \\ &&&&& \\
      \mbox{4)} &
            \gbeg{3}{5}
            \gvac{1}            \gcmu       \gnl
            \gcn{2}{1}{3}{2}    \linea      \gnl
            \gcmu               \linea      \gnl
            \cruce              \linea      \gnl
            \linea              \gmu        \gnl
            \gend
            =
            \gbeg{3}{5}
            \gcmu                               \gnl
            \linea      \gcn{1}{1}{1}{2}        \gnl
            \linea      \gcmu                   \gnl
            \cruce      \linea                  \gnl
            \linea      \gmu                    \gnl
            \gend
            &
      \mbox{5)} &
            \gbeg{4}{7}
            \gcmu                           \gnl
            \gcl{1} \gcn{1}{1}{1}{2}        \gnl
            \gcl{1} \gcmu                   \gnl
            \gcl{1} \gcl{1} \gcn{1}{1}{1}{2}\gnl
            \gcl{1} \gcl{1} \gcmu           \gnl
            \linea  \cruce  \linea          \gnl
            \gmu    \linea  \linea          \gnl
            \gend
            =
            \gbeg{4}{7}
            \gvac{1}    \gwcm{3}       \gnl
            \gcn{3}{1}{3}{2}        \gcl{1}     \gnl
            \gcmu      \gvac{1}     \linea     \gnl
            \linea      \gcn{1}{1}{1}{2} \gvac{1}       \linea \gnl
            \linea      \gcmu       \linea      \gnl
            \linea      \cruce       \linea     \gnl
            \gmu        \linea      \linea       \gnl
            \gend
             &
      \mbox{6)} &

            \gbeg{4}{7}
            \gvac{2}            \gcmu                   \gnl
            \gvac{1}            \gcn{2}{1}{3}{2}\linea  \gnl
            \gvac{1}            \gcmu           \linea  \gnl
            \gcn{2}{1}{3}{2}    \linea          \linea  \gnl
            \gcmu               \linea           \linea \gnl
            \linea              \cruce           \linea \gnl
            \linea              \linea           \gmu   \gnl
            \gend
            =
            \gbeg{4}{7}
            \gwcm{3}                                        \gnl
            \linea  \gvac{1}            \gcn{1}{1}{1}{2}    \gnl
            \linea  \gvac{1}            \gcmu               \gnl
            \linea  \gcn{2}{1}{3}{2}    \linea              \gnl
            \linea  \gcmu               \linea              \gnl
            \linea  \cruce              \linea              \gnl
            \linea  \linea              \gmu                \gnl
            \gend
    \end{array}
    \end{displaymath}
\end{lemma}
\begin{proof}
  By evaluating $\Id \tilde{\otimes} \epsilon \tilde{\otimes} \Id$ on both sides of the left co-Moufang identity and using the property (\ref{eq:counit}) of the counit we get 1):
  \begin{displaymath}
    \gbeg{4}{7}
    \gcmu                                       \gnl
    \gcl{1} \gcn{1}{1}{1}{2}                    \gnl
    \gcl{1} \gcmu                               \gnl
    \gcl{1} \gcl{1}             \gcn{1}{1}{1}{2}\gnl
    \gcl{1} \gcl{1}             \gcmu           \gnl
    \gcl{1} \gbr                \gcl{1}         \gnl
    \gmu    \gcu{1}             \gcl{1}         \gnl
    \gend
    -
    \gbeg{4}{7}
    \gvac{2}                \gcmu                   \gnl
    \gcn{3}{1}{5}{4}        \gcl{1}                 \gnl
    \gvac{1}                \gcmu       \gcl{1}     \gnl
    \gcn{2}{1}{3}{2}        \gcl{1}     \gcl{1}     \gnl
    \gcmu                   \gcl{1}     \gcl{1}     \gnl
    \gcl{1}                 \gbr        \gcl{1}     \gnl
    \gmu                    \gcu{1}     \gcl{1}     \gnl
    \gend
    =
    \gbeg{3}{4}
    \gcmu                               \gnl
    \linea      \gcn{1}{1}{1}{2}        \gnl
    \linea      \gcmu                   \gnl
    \gmu        \linea                  \gnl
    \gend
    -
    \gbeg{3}{4}
    \gvac{1}            \gcmu       \gnl
    \gcn{2}{1}{3}{2}    \linea      \gnl
    \gcmu               \linea      \gnl
    \gmu                \linea      \gnl
    \gend
  \end{displaymath}
  The proof of equalities 2), 3) and 4) is similar. Equality 5) is a consequence of 3) and the left co-Moufang identity. Identity 6) is a consequence of 4) and the right co-Moufang identity. \qed
\end{proof}
 All the results can be established without diagrams, and we will do so sometimes to compare both approaches. However, it will become more and more apparent that diagrams are a valuable notation for dealing with non-coassociative bialgebras. For instance, part 1) in Lemma~\ref{lem:basic} could be rewritten as
 \begin{displaymath}
   \sum x_{\si}\bullet x_{\sii\si} \tilde{\otimes} x_{\sii\sii} = \sum x_{\si\si}\bullet x_{\si\sii} \tilde{\otimes} x_{\sii}
 \end{displaymath}
and, as we have seen, it is derived from the left co-Moufang identity by
\begin{eqnarray*}
    && \sum x_{\si} \bullet x_{\sii\si}\tilde{\otimes} 1 \tilde{\otimes} x_{\sii\sii}
     =\\
    && \quad\quad \sum x_{\si} \bullet x_{\sii\sii\si}\tilde{\otimes} \epsilon(x_{\sii\si})1 \tilde{\otimes} x_{\sii\sii\sii}  = \\
    && \quad\quad \sum x_{\si\si\si} \bullet x_{\si\sii} \tilde{\otimes} \epsilon(x_{\si\si\sii})1 \tilde{\otimes} x_{\sii}  =\\
    && \quad\quad \sum x_{\si\si} \bullet x_{\si\sii} \tilde{\otimes} 1 \tilde{\otimes} x_{\sii}.
\end{eqnarray*}
When more intensive calculus is needed this approach is less useful.

The following lemma is well-known and its proof is obvious.
\begin{lemma}\label{lem:induction}
  Let $(A_h,\bullet)$ and $(B_h,\bullet)$ be topologically free algebras, that we identify with $A[[h]]$ and $B[[h]]$ as $k[[h]]$-modules for some $k$-algebras $A$ and $B$, and $\varphi_n,\psi_n \colon A \rightarrow B$ be linear maps ($n\geq 0$) such that the maps $\varphi_h = \sum_{n\geq 0} \varphi_n h^n$, $\psi_h = \sum_{n\geq 0} \psi_n h^n$ satisfy $\varphi_h(x \bullet y) = \varphi_h(x)\bullet \varphi_h(y)$ and $\psi_h(x \bullet y) = \psi_h(x) \bullet \psi_h(y)$ for all $x,y \in A_h$. If $\varphi_i = \psi_i$ $i=0,\dots, n-1$ then
  \begin{displaymath}
    (\varphi_n - \psi_n)(xy) = (\varphi_n - \psi_n)(x) \varphi_0(y) + \psi_0(x)(\varphi_n - \psi_n)(y)
  \end{displaymath}
  for al $x,y \in A$. In particular, if $A_h = U_h(\m)$, $\varphi_i = \psi_i$ $i=0,\dots, n-1$ and $\varphi_n\vert_\m = \psi_n\vert_\m$ then $\varphi_n = \psi_n$.
\end{lemma}

The map $Q\colon U(\m) \rightarrow U(\m)$ given by
\begin{displaymath}
Q(x) := \sum x_{(1)}x_{(2)} \quad \quad Q = \gbeg{2}{2} \gcmu \gnl \gmu \gnl \gend
\end{displaymath}
will play an important role. Since $U(\m)$ is spanned by $\{ a^n \mid a \in \m, n\geq 0\}$ \cite{PS04} and
\begin{displaymath}
  Q(a^n) = \sum \binom{n}{i} a^i a^{n-i} = 2^n a^n
\end{displaymath}
then $Q$ is semisimple with eigenvalues $\{2^n \mid n \geq 0\}$--notice that there is no ambiguity in the notation $a^n$ when $a\in \m$ since $ka$ is an abelian Lie subalgebra of $\m$, so the subalgebra generated by $a$ in $U(\m)$ is isomorphic to the associative and commutative algebra $U(k a)$. The primitive elements $\m$ form the subspace spanned by the eigenvectors of eigenvalue $2$.

The comultiplication  $\Delta_h\vert_{U(\m)}$ and the multiplication $p_h\vert_{U(\m)}$ can be developed as a power series on $h$
\begin{displaymath}
  \Delta_h\vert_{U(\m)} = \Delta_0 + h \Delta_1 + h^2 \Delta_2 + \cdots,\quad p_h\vert_{U(\m)} = p_0 + h p_1 + h^2 p_2 + \cdots
\end{displaymath}
where $p_i \colon U(\m) \otimes U(\m) \rightarrow U(\m)$ and $\Delta_i \colon U(\m) \rightarrow U(\m) \otimes U(\m)$ are linear maps and $\Delta_0 = \Delta$, $p_0 = p$ are the comultiplication and the multiplication on $U(\m)$.
The following notation will be very helpful:
\begin{displaymath}
  \gbeg{2}{2}
  \etiqueta{0}\gcmu\gnl
  \linea \linea \gnl
  \gend
  :=
  \Delta_0,
  \quad
  \gbeg{2}{2}
  \etiqueta{+}\gcmu\gnl
  \linea  \linea \gnl
  \gend
  :=
  h \Delta_1 + h^2 \Delta_2 + \cdots
\end{displaymath}
so, when restricted to $U(\m)$,
\begin{equation}
\label{eq:decomposition}
\gbeg{2}{2} \gcmu\gnl \linea \linea \gnl\gend = \gbeg{2}{2} \etiqueta{0}\gcmu\gnl \linea \linea \gnl \gend + \gbeg{2}{2} \etiqueta{+}\gcmu\gnl \linea \linea \gnl\gend.
\end{equation}
We will use a similar notation for $p_h\vert_{U(\m)}$ too.
\begin{theorem}\label{thm:cocommutative}
  Any bialgebra deformation of $U(\M(\alpha,\beta,\gamma))$ satisfying the left and right co-Moufang identities is cocommutative.
\end{theorem}
\begin{proof}
  We will prove by induction that $\Delta_n = \Delta^{op}_n$. Since $U(\m)$ is cocommutative then $\Delta_0 = \Delta^{\op}_0$. Let us assume that $\Delta_1 = \Delta^{\op}_1, \dots, \Delta_{n-1} = \Delta^{\op}_{n-1}$.  By Lemma~\ref{lem:basic} we have
  \begin{displaymath}
    D := \gbeg{3}{5}
    \gvac{1}            \gcmu   \gnl
    \gcn{2}{1}{3}{2}    \linea  \gnl
    \gcmu               \linea  \gnl
    \gmu                \linea  \gnl
    \gcn{2}{1}{2}{1}    \linea  \gnl
    \gend
    -
    \gbeg{3}{5}
    \gvac{1}    \gcmu        \gnl
    \gvac{1}    \cruce  \gnl
    \gcn{2}{1}{3}{2}    \linea      \gnl
    \gcmu               \linea      \gnl
    \gmu    \linea      \gnl
    \gend
    =
    \gbeg{3}{5}
    \gcmu       \gnl
    \linea      \gcn{1}{1}{1}{2}        \gnl
    \linea      \gcmu                   \gnl
    \gmu        \linea                  \gnl
    \gcn{2}{1}{2}{1}      \linea  \gnl
    \gend
    -
    \gbeg{3}{5}
    \gvac{1}    \gcmu        \gnl
    \gvac{1}    \cruce  \gnl
    \gcn{2}{1}{3}{2}    \linea      \gnl
    \gcmu               \linea      \gnl
    \gmu    \linea      \gnl
    \gend
    =
    \gbeg{3}{6}
    \gcmu       \gnl
    \linea      \gcn{1}{1}{1}{2}        \gnl
    \linea      \gcmu                   \gnl
    \gmu        \linea                  \gnl
    \gcn{2}{1}{2}{1}      \linea  \gnl
    \linea  \gvac{1} \linea  \gnl
    \gend
    -
    \gbeg{3}{6}
    \gvac{1}    \gcmu        \gnl
    \gcn{2}{1}{3}{2}    \linea      \gnl
    \gcmu               \linea      \gnl
    \cruce  \linea  \gnl
    \linea  \cruce  \gnl
    \gmu    \linea      \gnl
    \gend
    =
    \gbeg{3}{6}
    \gcmu       \gnl
    \linea      \gcn{1}{1}{1}{2}        \gnl
    \linea      \gcmu                   \gnl
    \gmu        \linea                  \gnl
    \gcn{2}{1}{2}{1}      \linea  \gnl
    \linea  \gvac{1} \linea  \gnl
    \gend
    -
    \gbeg{3}{6}
    \gvac{1}    \etiqueta{0}\gcmu        \gnl
    \gcn{2}{1}{3}{2}    \linea      \gnl
    \gcmu               \linea      \gnl
    \cruce  \linea  \gnl
    \linea  \cruce  \gnl
    \gmu    \linea      \gnl
    \gend
    -
    \gbeg{3}{6}
    \gvac{1}    \etiqueta{+}\gcmu        \gnl
    \gcn{2}{1}{3}{2}    \linea      \gnl
    \gcmu               \linea      \gnl
    \cruce  \linea  \gnl
    \linea  \cruce  \gnl
    \gmu    \linea      \gnl
    \gend
\end{displaymath}
Since $\Delta_i = \Delta^{\op}_i$ $i = 0, \dots, n-1$ then, modulo $h^{n+1}$,
\begin{displaymath}
  D \equiv \gbeg{3}{6}
    \gcmu       \gnl
    \linea      \gcn{1}{1}{1}{2}        \gnl
    \linea      \gcmu                   \gnl
    \gmu        \linea                  \gnl
    \gcn{2}{1}{2}{1}      \linea  \gnl
    \linea  \gvac{1} \linea  \gnl
    \gend
    -
    \gbeg{3}{6}
    \gvac{1}    \etiqueta{0}\gcmu        \gnl
    \gcn{2}{1}{3}{2}    \linea      \gnl
    \gcmu               \linea      \gnl
    \cruce  \linea  \gnl
    \linea  \cruce  \gnl
    \gmu    \linea      \gnl
    \gend
    -
    \gbeg{3}{6}
    \gvac{1}    \etiqueta{+}\gcmu        \gnl
    \gcn{2}{1}{3}{2}    \linea      \gnl
    \gcmu               \linea      \gnl
    \linea  \cruce  \gnl
    \gmu    \linea      \gnl
    \gcn{2}{1}{2}{1} \linea  \gnl
    \gend
    =
    \gbeg{3}{6}
    \gcmu       \gnl
    \linea      \gcn{1}{1}{1}{2}        \gnl
    \linea      \gcmu                   \gnl
    \gmu        \linea                  \gnl
    \gcn{2}{1}{2}{1}      \linea  \gnl
    \linea  \gvac{1} \linea  \gnl
    \gend
    -
    \gbeg{3}{6}
    \gvac{1}    \etiqueta{0}\gcmu        \gnl
    \gcn{2}{1}{3}{2}    \linea      \gnl
    \gcmu               \linea      \gnl
    \cruce  \linea  \gnl
    \linea  \cruce  \gnl
    \gmu    \linea      \gnl
    \gend
    +
    \gbeg{3}{6}
    \gvac{1}    \etiqueta{0}\gcmu        \gnl
    \gcn{2}{1}{3}{2}    \linea      \gnl
    \gcmu               \linea      \gnl
    \linea  \cruce  \gnl
    \gmu    \linea      \gnl
    \gcn{2}{1}{2}{1} \linea  \gnl
    \gend
    -
    \gbeg{3}{6}
    \gvac{1}    \gcmu        \gnl
    \gcn{2}{1}{3}{2}    \linea      \gnl
    \gcmu               \linea      \gnl
    \linea  \cruce  \gnl
    \gmu    \linea      \gnl
    \gcn{2}{1}{2}{1} \linea  \gnl
    \gend
\end{displaymath}
By Lemma~\ref{lem:basic}
\begin{displaymath}
  D \equiv \gbeg{3}{6}
    \gcmu       \gnl
    \linea      \gcn{1}{1}{1}{2}        \gnl
    \linea      \gcmu                   \gnl
    \gmu        \linea                  \gnl
    \gcn{2}{1}{2}{1}      \linea  \gnl
    \linea  \gvac{1} \linea  \gnl
    \gend
    -
    \gbeg{3}{6}
    \gvac{1}    \etiqueta{0}\gcmu        \gnl
    \gcn{2}{1}{3}{2}    \linea      \gnl
    \gcmu               \linea      \gnl
    \cruce  \linea  \gnl
    \linea  \cruce  \gnl
    \gmu    \linea      \gnl
    \gend
    +
    \gbeg{3}{6}
    \gvac{1}    \etiqueta{0}\gcmu        \gnl
    \gcn{2}{1}{3}{2}    \linea      \gnl
    \gcmu               \linea      \gnl
    \linea  \cruce  \gnl
    \gmu    \linea      \gnl
    \gcn{2}{1}{2}{1} \linea  \gnl
    \gend
    -
    \gbeg{3}{6}
    \gcmu       \gnl
    \linea      \gcn{1}{1}{1}{2}        \gnl
    \linea      \gcmu                   \gnl
    \linea      \cruce                  \gnl
    \gmu        \linea                  \gnl
    \gcn{2}{1}{2}{1}   \linea  \gnl
    \gend
    =
    \gbeg{3}{6}
    \gcmu       \gnl
    \linea      \gcn{1}{1}{1}{2}        \gnl
    \linea      \gcmu                   \gnl
    \gmu        \linea                  \gnl
    \gcn{2}{1}{2}{1}      \linea  \gnl
    \linea  \gvac{1} \linea  \gnl
    \gend
    -
    \gbeg{3}{6}
    \gvac{1}    \etiqueta{0}\gcmu        \gnl
    \gcn{2}{1}{3}{2}    \linea      \gnl
    \gcmu               \linea      \gnl
    \cruce  \linea  \gnl
    \linea  \cruce  \gnl
    \gmu    \linea      \gnl
    \gend
    +
    \gbeg{3}{6}
    \gvac{1}    \etiqueta{0}\gcmu        \gnl
    \gcn{2}{1}{3}{2}    \linea      \gnl
    \gcmu               \linea      \gnl
    \linea  \cruce  \gnl
    \gmu    \linea      \gnl
    \gcn{2}{1}{2}{1} \linea  \gnl
    \gend
    -
    \gbeg{3}{6}
    \etiqueta{0}\gcmu       \gnl
    \linea      \gcn{1}{1}{1}{2}        \gnl
    \linea      \gcmu                   \gnl
    \linea      \cruce                  \gnl
    \gmu        \linea                  \gnl
    \gcn{2}{1}{2}{1}   \linea  \gnl
    \gend
    -
    \gbeg{3}{6}
    \etiqueta{+}\gcmu       \gnl
    \linea      \gcn{1}{1}{1}{2}        \gnl
    \linea      \gcmu                   \gnl
    \linea      \cruce                  \gnl
    \gmu        \linea                  \gnl
    \gcn{2}{1}{2}{1}   \linea  \gnl
    \gend
\end{displaymath}
Since $\Delta_i = \Delta^{\op}_i$ $i= 0, \dots, n-1$ then, modulo $h^{n+1}$,
\begin{displaymath}
  D \equiv \gbeg{3}{6}
    \gcmu       \gnl
    \linea      \gcn{1}{1}{1}{2}        \gnl
    \linea      \gcmu                   \gnl
    \gmu        \linea                  \gnl
    \gcn{2}{1}{2}{1}      \linea  \gnl
    \linea  \gvac{1} \linea  \gnl
    \gend
    -
    \gbeg{3}{6}
    \gvac{1}    \etiqueta{0}\gcmu        \gnl
    \gcn{2}{1}{3}{2}    \linea      \gnl
    \gcmu               \linea      \gnl
    \cruce  \linea  \gnl
    \linea  \cruce  \gnl
    \gmu    \linea      \gnl
    \gend
    +
    \gbeg{3}{6}
    \gvac{1}    \etiqueta{0}\gcmu        \gnl
    \gcn{2}{1}{3}{2}    \linea      \gnl
    \gcmu               \linea      \gnl
    \linea  \cruce  \gnl
    \gmu    \linea      \gnl
    \gcn{2}{1}{2}{1} \linea  \gnl
    \gend
    -
    \gbeg{3}{6}
    \etiqueta{0}\gcmu       \gnl
    \linea      \gcn{1}{1}{1}{2}        \gnl
    \linea      \gcmu                   \gnl
    \linea      \cruce                  \gnl
    \gmu        \linea                  \gnl
    \gcn{2}{1}{2}{1}   \linea  \gnl
    \gend
    -
    \gbeg{3}{6}
    \etiqueta{+}\gcmu       \gnl
    \linea      \gcn{1}{1}{1}{2}        \gnl
    \linea      \gcmu                   \gnl
    \gmu        \linea                  \gnl
    \gcn{2}{1}{2}{1}      \linea  \gnl
    \linea  \gvac{1} \linea  \gnl
    \gend
    =
    \gbeg{3}{6}
    \etiqueta{0}\gcmu       \gnl
    \linea      \gcn{1}{1}{1}{2}        \gnl
    \linea      \gcmu                   \gnl
    \gmu        \linea                  \gnl
    \gcn{2}{1}{2}{1}      \linea  \gnl
    \linea  \gvac{1} \linea  \gnl
    \gend
    -
    \gbeg{3}{6}
    \gvac{1}            \etiqueta{0}\gcmu   \gnl
    \gcn{2}{1}{3}{2}    \linea              \gnl
    \gcmu               \linea              \gnl
    \cruce              \linea              \gnl
    \linea              \cruce              \gnl
    \gmu                \linea              \gnl
    \gend
    +
    \gbeg{3}{6}
    \gvac{1}            \etiqueta{0}\gcmu       \gnl
    \gcn{2}{1}{3}{2}    \linea                  \gnl
    \gcmu               \linea                  \gnl
    \linea              \cruce                  \gnl
    \gmu                \linea                  \gnl
    \gcn{2}{1}{2}{1}    \linea                  \gnl
    \gend
    -
    \gbeg{3}{6}
    \etiqueta{0}    \gcmu           \gnl
    \linea          \gcn{1}{1}{1}{2}\gnl
    \linea          \gcmu           \gnl
    \linea          \cruce          \gnl
    \gmu            \linea          \gnl
    \gcn{2}{1}{2}{1}\linea          \gnl
    \gend
\end{displaymath}
Any $a \in \m$ is primitive, i.e., $\Delta_0(a) = a \otimes 1 + 1 \otimes a$ thus
\begin{displaymath}
  \begin{array}{ll}
    \gbeg{3}{6}
    \got{2}{a}                              \gnl
    \etiqueta{0}\gcmu                       \gnl
    \linea              \gcn{1}{1}{1}{2}    \gnl
    \linea              \gcmu               \gnl
    \gmu                \linea              \gnl
    \gcn{2}{1}{2}{1}    \linea              \gnl
    \gend
    =
    \gbeg{2}{2}
    \got{2}{a}  \gnl
    \gcmu       \gnl
    \gend
    + a \otimes 1,
    \quad\quad &
    \gbeg{3}{7}
    \got{4}{a}                              \gnl
    \gvac{1}            \etiqueta{0}\gcmu   \gnl
    \gcn{2}{1}{3}{2}    \linea              \gnl
    \gcmu               \linea              \gnl
    \cruce              \linea              \gnl
    \linea              \cruce              \gnl
    \gmu                \linea              \gnl
    \gend
    =
    \gbeg{2}{3}
    \got{2}{a}  \gnl
    \gcmu       \gnl
    \cruce      \gnl
    \gend
    + a \otimes 1,
    \\ &
    \\
    \gbeg{3}{7}
    \got{4}{a}                                  \gnl
    \gvac{1}            \etiqueta{0}\gcmu       \gnl
    \gcn{2}{1}{3}{2}    \linea                  \gnl
    \gcmu               \linea                  \gnl
    \linea              \cruce                  \gnl
    \gmu                \linea                  \gnl
    \gcn{2}{1}{2}{1}    \linea                  \gnl
    \gend
    =
    \gbeg{2}{2}
    \got{2}{a}  \gnl
    \gcmu       \gnl
    \gend
    + a \otimes 1,
    &
    \gbeg{3}{7}
    \got{2}{a}                      \gnl
    \etiqueta{0}    \gcmu           \gnl
    \linea          \gcn{1}{1}{1}{2}\gnl
    \linea          \gcmu           \gnl
    \linea          \cruce          \gnl
    \gmu            \linea          \gnl
    \gcn{2}{1}{2}{1}\linea          \gnl
    \gend
    =
    \gbeg{2}{3}
    \got{2}{a}  \gnl
    \gcmu       \gnl
    \cruce      \gnl
    \gend
    + a \otimes 1
  \end{array}
\end{displaymath}
and
\begin{displaymath}
    \gbeg{3}{6}
    \got{4}{a}                  \gnl
    \gvac{1}            \gcmu   \gnl
    \gcn{2}{1}{3}{2}    \linea  \gnl
    \gcmu               \linea  \gnl
    \gmu                \linea  \gnl
    \gcn{2}{1}{2}{1}    \linea  \gnl
    \gend
    -
    \gbeg{3}{6}
    \got{4}{a}                  \gnl
    \gvac{1}            \gcmu   \gnl
    \gvac{1}            \cruce  \gnl
    \gcn{2}{1}{3}{2}    \linea  \gnl
    \gcmu               \linea  \gnl
    \gmu                \linea  \gnl
    \gend
    \equiv
    2\left(
    \gbeg{2}{2}
    \got{2}{a}  \gnl
    \gcmu       \gnl
    \gend
    -
    \gbeg{2}{3}
    \got{2}{a}  \gnl
    \gcmu       \gnl
    \cruce      \gnl
    \gend
    \right)\quad(\modulus h^{n+1})
\end{displaymath}
Since $\Delta_i = \Delta^{\op}_i$ $i = 0,\dots, n-1$ we finally get
\begin{displaymath}
  \gbeg{3}{6}
    \got{4}{a}                  \gnl
    \gvac{1}            \gcmu   \gnl
    \gcn{2}{1}{3}{2}    \linea  \gnl
    \etiqueta{0}\gcmu   \linea  \gnl
    \etiqueta{0}\gmu    \linea  \gnl
    \gcn{2}{1}{2}{1}    \linea  \gnl
    \gend
    -
    \gbeg{3}{6}
    \got{4}{a}                  \gnl
    \gvac{1}            \gcmu   \gnl
    \gvac{1}            \cruce  \gnl
    \gcn{2}{1}{3}{2}    \linea  \gnl
    \etiqueta{0}\gcmu   \linea  \gnl
    \etiqueta{0}\gmu    \linea  \gnl
    \gend
    \equiv
    2\left(
    \gbeg{2}{2}
    \got{2}{a}  \gnl
    \gcmu       \gnl
    \gend
    -
    \gbeg{2}{3}
    \got{2}{a}  \gnl
    \gcmu       \gnl
    \cruce      \gnl
    \gend
    \right)\quad(\modulus h^{n+1})
\end{displaymath}
This shows that $(\Delta_n - \Delta^{\op}_n)(a)$ is an eigenvector of $Q$ of eigenvalue $2$ so
\begin{displaymath}
(\Delta_n - \Delta^{\op}_n)(a) \in \m \wedge \m
\end{displaymath}
for any $a \in \m$. The result follows from \cite{MP}. \qed
\end{proof}

\section{Coassociativity}
Let $U_h(\m)$  be a bialgebra deformation of $U(\m)$ satisfying the left and right co-Moufang identities. The coassociator $C_h := (\Delta_h \tilde{\otimes} \Id) \Delta_h - (\Id \tilde{\otimes} \Delta_h)\Delta_h$ of $U_h(\m)$ measures the lack of coassociativity. It can be developed into a power series
\begin{displaymath}
  C_h\vert_{U(\m)} = hC_1 + h^2C_2 + \cdots.
\end{displaymath}
The constant term vanishes due to the coassociativity of $U(\m)$. We will prove that $C_n = 0$ using induction on $n$.
\begin{lemma}\label{lem:P1}
  Assume that $C_1 = \cdots = C_{n-1} = 0$. Then for any $a \in \m$
  \begin{displaymath}
    \gbeg{4}{7}
    \got{6}{a}           \gnl
    \gvac{2}            \gcmu                   \gnl
    \gcn{3}{1}{5}{4}    \gcl{1}                 \gnl
    \gvac{1}            \gcmu       \gcl{1}     \gnl
    \gcn{2}{1}{3}{2}    \gcl{1}     \gcl{1}     \gnl
    \gcmu               \gcl{1}     \gcl{1}     \gnl
    \gmu                \gcl{1}     \gcl{1}     \gnl
    \gend
    -
    \gbeg{4}{7}
    \got{4}{a}                                                  \gnl
    \gvac{1}            \gcmu                                   \gnl
    \gcn{1}{1}{3}{2}    \gcn{1}{1}{3}{4}                        \gnl
    \gcmu               \gcmu                                   \gnl
    \gmu                \linea              \linea              \gnl
    \gcn{2}{1}{2}{1}    \linea              \linea              \gnl
    \linea              \gvac{1}            \linea      \linea  \gnl
    \gend
    \equiv
    \gbeg{3}{5}
    \got{4}{a}                          \gnl
    \gvac{1}            \gcmu           \gnl
    \gcn{2}{1}{3}{2}    \linea          \gnl
    \gcmu               \linea          \gnl
    \linea              \linea   \linea \gnl
    \gend
    -
    \gbeg{4}{5}
    \got{3}{a}                  \gnl
    \gvac{1}            \gcmu   \gnl
    \gcn{2}{1}{3}{2}    \linea  \gnl
    \gcmu               \linea  \gnl
    \gbr                \linea \gnl
    \gend
    +
    \gbeg{3}{5}
    \got{2}{a}                  \gnl
    \gcmu                       \gnl
    \linea  \gcn{1}{1}{1}{2}    \gnl
    \linea  \gcmu               \gnl
    \gbr    \linea              \gnl
    \gend
    -
    \gbeg{3}{5}
    \got{2}{a}                  \gnl
    \gcmu                       \gnl
    \linea  \gcn{1}{1}{1}{2}    \gnl
    \linea  \gcmu               \gnl
    \linea  \linea      \linea  \gnl
    \gend
    \quad (\modulus h^{n+1})
\end{displaymath}
\end{lemma}
\begin{proof}
  By Lemma~\ref{lem:basic},
  \begin{displaymath}
    D :=
    \gbeg{4}{6}
    \gvac{2}    \gcmu       \gnl
    \gcn{3}{1}{5}{4}        \gcl{1}     \gnl
    \gvac{1}    \gcmu       \gcl{1}     \gnl
    \gcn{2}{1}{3}{2}        \gcl{1}     \gcl{1} \gnl
    \gcmu       \gcl{1}     \gcl{1}     \gnl
    \gmu        \gcl{1}     \gcl{1}     \gnl
    \gend
    -
    \gbeg{4}{6}
    \gvac{1}\gcmu \gnl
    \gcn{1}{1}{3}{2} \gcn{1}{1}{3}{4} \gnl
    \gcmu   \gcmu   \gnl
    \gmu \linea    \linea   \gnl
    \gcn{2}{1}{2}{1} \linea  \linea \gnl
    \linea  \gvac{1}    \linea  \linea  \gnl
    \gend
    =
    \gbeg{4}{6}
    \gvac{1}\gwcm{3}       \gnl
    \gcn{3}{1}{3}{2}   \linea \gnl
    \gcmu       \gvac{1}\linea \gnl
    \linea      \gcn{2}{1}{1}{2}        \linea\gnl
    \linea      \gcmu                   \linea \gnl
    \gmu        \linea                  \linea\gnl
    \gend
    -
    \gbeg{4}{6}
    \gvac{1}\gcmu \gnl
    \gcn{1}{1}{3}{2} \gcn{1}{1}{3}{4} \gnl
    \gcmu   \gcmu   \gnl
    \gmu \linea    \linea   \gnl
    \gcn{2}{1}{2}{1} \linea  \linea \gnl
    \linea  \gvac{1}    \linea  \linea  \gnl
    \gend
    =
    \gbeg{4}{6}
    \gvac{1}\gwcm{3}       \gnl
    \gcn{3}{1}{3}{2}   \linea \gnl
    \gcmu       \gvac{1}\linea \gnl
    \linea      \gcn{2}{1}{1}{2}        \linea\gnl
    \linea      \gcmu                   \linea \gnl
    \gmu        \linea                  \linea\gnl
    \gend
    -
    \gbeg{4}{6}
    \gcmu                                   \gnl
    \linea              \gcn{1}{1}{1}{2}    \gnl
    \linea              \gcmu               \gnl
    \gmu                \gcn{1}{1}{1}{2}    \gnl
    \gcn{2}{1}{2}{1}    \gcmu               \gnl
    \linea              \gvac{1}            \linea  \linea  \gnl
    \gend
  \end{displaymath}
  Decomposing the first summand and using $\Delta_0 = \Delta^{\op}_0$ we get
  \begin{eqnarray*}
  D &=& \gbeg{4}{6}
    \gvac{1}            \gwcm{3}            \gnl
    \gcn{3}{1}{3}{2}    \linea              \gnl
    \gcmu               \gvac{1}            \linea  \gnl
    \linea              \gcn{2}{1}{1}{2}    \linea  \gnl
    \linea              \etiqueta{+}\gcmu   \linea  \gnl
    \gmu                \linea              \linea  \gnl
    \gend
    +
    \gbeg{4}{6}
    \gvac{1}            \gwcm{3}                    \gnl
    \gcn{3}{1}{3}{2}    \linea                      \gnl
    \gcmu               \gvac{1}            \linea  \gnl
    \linea              \gcn{2}{1}{1}{2}    \linea  \gnl
    \linea              \etiqueta{0}\gcmu   \linea  \gnl
    \gmu                \linea              \linea  \gnl
    \gend
    -
    \gbeg{4}{6}
    \gcmu               \gnl
    \linea              \gcn{1}{1}{1}{2}                    \gnl
    \linea              \gcmu                               \gnl
    \gmu                \gcn{1}{1}{1}{2}                    \gnl
    \gcn{2}{1}{2}{1}    \gcmu                               \gnl
    \linea              \gvac{1}            \linea  \linea  \gnl
    \gend
    =
    \gbeg{4}{7}
    \gvac{1}            \gwcm{3}                    \gnl
    \gcn{3}{1}{3}{2}    \linea                      \gnl
    \gcmu               \gvac{1}            \linea  \gnl
    \linea              \gcn{2}{1}{1}{2}    \linea  \gnl
    \linea              \etiqueta{+}\gcmu   \linea  \gnl
    \gmu                \linea              \linea  \gnl
    \gcn{2}{1}{2}{1}    \linea              \linea  \gnl
    \gend
    +
    \gbeg{4}{7}
    \gvac{1}            \gwcm{3}                            \gnl
    \gcn{3}{1}{3}{2}    \linea                              \gnl
    \gcmu               \gvac{1}            \linea          \gnl
    \linea              \gcn{2}{1}{1}{2}    \linea          \gnl
    \linea              \etiqueta{0}        \gcmu   \linea  \gnl
    \linea              \gbr                \linea          \gnl
    \gmu                \linea              \linea          \gnl
    \gend
    -
    \gbeg{4}{7}
    \gcmu                                                   \gnl
    \linea              \gcn{1}{1}{1}{2}                    \gnl
    \linea              \gcmu                               \gnl
    \gmu                \gcn{1}{1}{1}{2}                    \gnl
    \gcn{2}{1}{2}{1}    \gcmu                               \gnl
    \linea              \gvac{1}            \linea  \linea  \gnl
    \linea              \gvac{1}            \linea  \linea  \gnl
    \gend
    \\
  &=&
  \gbeg{4}{7}
    \gvac{1}            \gwcm{3}                    \gnl
    \gcn{3}{1}{3}{2}    \linea                      \gnl
    \gcmu               \gvac{1}            \linea  \gnl
    \linea              \gcn{2}{1}{1}{2}    \linea  \gnl
    \linea              \etiqueta{+}\gcmu   \linea  \gnl
    \gmu                \linea              \linea  \gnl
    \gcn{2}{1}{2}{1}    \linea              \linea  \gnl
    \gend
    -
    \gbeg{4}{7}
    \gvac{1}            \gwcm{3}                    \gnl
    \gcn{3}{1}{3}{2}    \linea                      \gnl
    \gcmu               \gvac{1}            \linea  \gnl
    \linea              \gcn{2}{1}{1}{2}    \linea  \gnl
    \linea              \etiqueta{+}\gcmu   \linea  \gnl
    \linea              \gbr                \linea  \gnl
    \gmu                \linea              \linea  \gnl
    \gend
    +
    \gbeg{4}{7}
    \gvac{1}            \gwcm{3}                    \gnl
    \gcn{3}{1}{3}{2}    \linea                      \gnl
    \gcmu               \gvac{1}            \linea  \gnl
    \linea              \gcn{2}{1}{1}{2}    \linea  \gnl
    \linea              \gcmu               \linea  \gnl
    \linea              \gbr                \linea  \gnl
    \gmu                \linea              \linea  \gnl
    \gend
    -
    \gbeg{4}{7}
    \gcmu                                           \gnl
    \linea              \gcn{1}{1}{1}{2}            \gnl
    \linea              \gcmu                       \gnl
    \gmu                \gcn{1}{1}{1}{2}            \gnl
    \gcn{2}{1}{2}{1}    \gcmu                       \gnl
    \linea              \gvac{1}    \linea  \linea  \gnl
    \linea              \gvac{1}    \linea  \linea  \gnl
    \gend
  \end{eqnarray*}
  Lemma~\ref{lem:basic} implies that
  \begin{eqnarray*}
    D & = &
    \gbeg{4}{7}
    \gvac{1}            \gwcm{3}                    \gnl
    \gcn{3}{1}{3}{2}    \linea                      \gnl
    \gcmu               \gvac{1}            \linea  \gnl
    \linea              \gcn{2}{1}{1}{2}    \linea  \gnl
    \linea              \etiqueta{+}\gcmu   \linea  \gnl
    \gmu                \linea              \linea  \gnl
    \gcn{2}{1}{2}{1}    \linea              \linea  \gnl
    \gend
    -
    \gbeg{4}{7}
    \gvac{1}            \gwcm{3}                    \gnl
    \gcn{3}{1}{3}{2}    \linea                      \gnl
    \gcmu               \gvac{1}            \linea  \gnl
    \linea              \gcn{2}{1}{1}{2}    \linea  \gnl
    \linea              \etiqueta{+}\gcmu   \linea  \gnl
    \linea              \gbr                \linea  \gnl
    \gmu                \linea              \linea  \gnl
    \gend
    +
    \gbeg{4}{7}
    \gcmu \gnl
    \gcl{1} \gcn{1}{1}{1}{2}        \gnl
    \gcl{1} \gcmu                   \gnl
    \gcl{1} \gcl{1} \gcn{1}{1}{1}{2}\gnl
    \gcl{1} \gcl{1} \gcmu           \gnl
    \linea  \cruce  \linea          \gnl
    \gmu    \linea  \linea          \gnl
    \gend
    -
    \gbeg{4}{7}
    \gcmu                                           \gnl
    \linea              \gcn{1}{1}{1}{2}            \gnl
    \linea              \gcmu                       \gnl
    \gmu                \gcn{1}{1}{1}{2}            \gnl
    \gcn{2}{1}{2}{1}    \gcmu                       \gnl
    \linea              \gvac{1}    \linea  \linea  \gnl
    \linea              \gvac{1}    \linea  \linea  \gnl
    \gend
    \\
    &=&
    \gbeg{4}{7}
    \gvac{1}            \gwcm{3}                    \gnl
    \gcn{3}{1}{3}{2}    \linea                      \gnl
    \gcmu               \gvac{1}            \linea  \gnl
    \linea              \gcn{2}{1}{1}{2}    \linea  \gnl
    \linea              \etiqueta{+}\gcmu   \linea  \gnl
    \gmu                \linea              \linea  \gnl
    \gcn{2}{1}{2}{1}    \linea              \linea  \gnl
    \gend
    -
    \gbeg{4}{7}
    \gvac{1}\gwcm{3}                                \gnl
    \gcn{3}{1}{3}{2}    \linea                      \gnl
    \gcmu               \gvac{1}            \linea \gnl
    \linea              \gcn{2}{1}{1}{2}    \linea  \gnl
    \linea              \etiqueta{+}\gcmu   \linea  \gnl
    \linea              \gbr                \linea  \gnl
    \gmu                \linea              \linea  \gnl
    \gend
    +
    \gbeg{4}{7}
    \etiqueta{0}\gcmu \gnl
    \gcl{1} \gcn{1}{1}{1}{2}        \gnl
    \gcl{1} \gcmu                   \gnl
    \gcl{1} \gcl{1} \gcn{1}{1}{1}{2}\gnl
    \gcl{1} \gcl{1} \gcmu           \gnl
    \linea  \cruce  \linea          \gnl
    \gmu    \linea  \linea          \gnl
    \gend
    +
    \gbeg{4}{7}
    \etiqueta{+}\gcmu \gnl
    \gcl{1} \gcn{1}{1}{1}{2}        \gnl
    \gcl{1} \gcmu                   \gnl
    \gcl{1} \gcl{1} \gcn{1}{1}{1}{2}\gnl
    \gcl{1} \gcl{1} \gcmu           \gnl
    \linea  \cruce  \linea          \gnl
    \gmu    \linea  \linea          \gnl
    \gend
    -
    \gbeg{4}{7}
    \gcmu \gnl
    \linea  \gcn{1}{1}{1}{2}            \gnl
    \linea  \gcmu                       \gnl
    \gmu    \gcn{1}{1}{1}{2}            \gnl
    \gcn{2}{1}{2}{1}    \gcmu           \gnl
    \linea  \gvac{1}    \linea  \linea  \gnl
    \linea  \gvac{1}    \linea  \linea  \gnl
    \gend
  \end{eqnarray*}
  Since $C_1 = \cdots = C_{n-1} = 0$ then, modulo $h^{n+1}$,
  \begin{eqnarray*}
    D &\equiv&
    \gbeg{4}{7}
    \gvac{1}            \gwcm{3}                    \gnl
    \gcn{3}{1}{3}{2}    \linea                      \gnl
    \gcmu               \gvac{1}            \linea  \gnl
    \linea              \gcn{2}{1}{1}{2}    \linea  \gnl
    \linea              \etiqueta{+}\gcmu   \linea  \gnl
    \gmu                \linea              \linea  \gnl
    \gcn{2}{1}{2}{1}    \linea              \linea  \gnl
    \gend
    -
    \gbeg{4}{7}
    \gvac{1}            \gwcm{3}                    \gnl
    \gcn{3}{1}{3}{2}    \linea                      \gnl
    \gcmu               \gvac{1}            \linea  \gnl
    \linea              \gcn{2}{1}{1}{2}    \linea  \gnl
    \linea              \etiqueta{+}\gcmu   \linea  \gnl
    \linea              \gbr                \linea  \gnl
    \gmu                \linea              \linea  \gnl
    \gend
    +
    \gbeg{4}{7}
    \etiqueta{0}\gcmu               \gnl
    \gcl{1} \gcn{1}{1}{1}{2}        \gnl
    \gcl{1} \gcmu                   \gnl
    \gcl{1} \gcl{1} \gcn{1}{1}{1}{2}\gnl
    \gcl{1} \gcl{1} \gcmu           \gnl
    \linea  \cruce  \linea          \gnl
    \gmu    \linea  \linea          \gnl
    \gend
    +
    \gbeg{4}{7}
    \etiqueta{+}\gwcm{3}                                            \gnl
    \linea                  \gvac{1}            \gcn{1}{1}{1}{2}    \gnl
    \linea                  \gvac{1}            \gcmu               \gnl
    \linea                  \gcn{2}{1}{3}{2}    \linea              \gnl
    \linea                  \gcmu               \linea              \gnl
    \linea                  \cruce              \linea              \gnl
    \gmu                    \linea              \linea              \gnl
    \gend
    -
    \gbeg{4}{7}
    \gcmu                                                   \gnl
    \linea              \gcn{1}{1}{1}{2}                    \gnl
    \linea              \gcmu                               \gnl
    \gmu                \gcn{1}{1}{1}{2}                    \gnl
    \gcn{2}{1}{2}{1}    \gcmu                               \gnl
    \linea              \gvac{1}            \linea  \linea  \gnl
    \linea              \gvac{1}            \linea  \linea  \gnl
    \gend \\
    &=&
    \gbeg{4}{7}
    \gvac{1}            \gwcm{3}                        \gnl
    \gcn{3}{1}{3}{2}    \linea                          \gnl
    \gcmu               \gvac{1}            \linea      \gnl
    \linea              \gcn{2}{1}{1}{2}    \linea      \gnl
    \linea              \etiqueta{+}\gcmu   \linea      \gnl
    \gmu                \linea              \linea      \gnl
    \gcn{2}{1}{2}{1}    \linea              \linea      \gnl
    \gend
    -
    \gbeg{4}{7}
    \gvac{1}            \gwcm{3}                    \gnl
    \gcn{3}{1}{3}{2}    \linea                      \gnl
    \gcmu               \gvac{1}            \linea  \gnl
    \linea              \gcn{2}{1}{1}{2}    \linea  \gnl
    \linea              \etiqueta{+}\gcmu   \linea  \gnl
    \linea              \gbr                \linea  \gnl
    \gmu                \linea              \linea  \gnl
    \gend
    +
    \gbeg{4}{7}
    \etiqueta{0}\gcmu               \gnl
    \gcl{1} \gcn{1}{1}{1}{2}        \gnl
    \gcl{1} \gcmu                   \gnl
    \gcl{1} \gcl{1} \gcn{1}{1}{1}{2}\gnl
    \gcl{1} \gcl{1} \gcmu           \gnl
    \linea  \cruce  \linea          \gnl
    \gmu    \linea  \linea          \gnl
    \gend
    +
    \gbeg{4}{7}
    \etiqueta{+}\gwcm{3}                                    \gnl
    \linea  \gvac{1}            \gcn{1}{1}{1}{2}            \gnl
    \linea  \gvac{1}            \gcmu                       \gnl
    \linea  \gcn{2}{1}{3}{2}    \linea                      \gnl
    \linea  \etiqueta{+}        \gcmu               \linea  \gnl
    \linea  \cruce              \linea                      \gnl
    \gmu    \linea              \linea                      \gnl
    \gend
    +
    \gbeg{4}{7}
    \etiqueta{+}\gwcm{3}                                            \gnl
    \linea                  \gvac{1}            \gcn{1}{1}{1}{2}    \gnl
    \linea                  \gvac{1}            \gcmu               \gnl
    \linea                  \gcn{2}{1}{3}{2}    \linea              \gnl
    \linea                  \etiqueta{0}\gcmu   \linea              \gnl
    \linea                  \cruce              \linea              \gnl
    \gmu                    \linea              \linea              \gnl
    \gend
    -
    \gbeg{4}{7}
    \gcmu                                                   \gnl
    \linea              \gcn{1}{1}{1}{2}                    \gnl
    \linea              \gcmu                               \gnl
    \gmu                \gcn{1}{1}{1}{2}                    \gnl
    \gcn{2}{1}{2}{1}    \gcmu                               \gnl
    \linea              \gvac{1}            \linea  \linea  \gnl
    \linea              \gvac{1}            \linea  \linea  \gnl
    \gend
  \end{eqnarray*}
  We again use that $C_1 = \cdots = C_{n-1} = 0$ to obtain, modulo $h^{n+1}$,
  \begin{eqnarray*}
    D &\equiv&
    \gbeg{4}{7}
    \gvac{1}            \gwcm{3}                    \gnl
    \gcn{3}{1}{3}{2}    \linea                      \gnl
    \gcmu               \gvac{1}            \linea  \gnl
    \linea              \gcn{2}{1}{1}{2}    \linea  \gnl
    \linea              \etiqueta{+}\gcmu   \linea  \gnl
    \gmu                \linea              \linea  \gnl
    \gcn{2}{1}{2}{1}    \linea              \linea  \gnl
    \gend
    -
    \gbeg{4}{7}
    \gwcm{3}   \gnl
    \linea  \gvac{1}    \gcn{1}{1}{1}{2}    \gnl
    \linea  \gvac{1}    \gcmu   \gnl
    \linea  \gcn{2}{1}{3}{2}    \linea  \gnl
    \linea  \etiqueta{+}\gcmu   \linea  \gnl
    \linea  \cruce  \linea  \gnl
    \gmu  \linea  \linea    \gnl
    \gend
    +
    \gbeg{4}{7}
    \etiqueta{0}\gcmu \gnl
    \gcl{1} \gcn{1}{1}{1}{2}        \gnl
    \gcl{1} \gcmu                   \gnl
    \gcl{1} \gcl{1} \gcn{1}{1}{1}{2}\gnl
    \gcl{1} \gcl{1} \gcmu           \gnl
    \linea  \cruce  \linea          \gnl
    \gmu    \linea  \linea          \gnl
    \gend
    +
    \gbeg{4}{7}
    \etiqueta{+}\gwcm{3}   \gnl
    \linea  \gvac{1}    \gcn{1}{1}{1}{2}    \gnl
    \linea  \gvac{1}    \gcmu   \gnl
    \linea  \gcn{2}{1}{3}{2}    \linea  \gnl
    \linea  \etiqueta{+}\gcmu   \linea  \gnl
    \linea  \cruce  \linea  \gnl
    \gmu  \linea  \linea    \gnl
    \gend
    +
    \gbeg{4}{7}
    \etiqueta{+}\gwcm{3}   \gnl
    \linea  \gvac{1}    \gcn{1}{1}{1}{2}    \gnl
    \linea  \gvac{1}    \gcmu   \gnl
    \linea  \gcn{2}{1}{3}{2}    \linea  \gnl
    \linea  \etiqueta{0}\gcmu   \linea  \gnl
    \linea  \cruce  \linea  \gnl
    \gmu  \linea  \linea    \gnl
    \gend
    -
    \gbeg{4}{7}
    \gcmu \gnl
    \linea  \gcn{1}{1}{1}{2}    \gnl
    \linea  \gcmu   \gnl
    \gmu    \gcn{1}{1}{1}{2}    \gnl
    \gcn{2}{1}{2}{1}    \gcmu   \gnl
    \linea  \gvac{1}    \linea  \linea  \gnl
    \linea  \gvac{1}    \linea  \linea  \gnl
    \gend\\
    &=&
    \gbeg{4}{7}
    \gvac{1}\gwcm{3}       \gnl
    \gcn{3}{1}{3}{2}   \linea \gnl
    \gcmu       \gvac{1}\linea \gnl
    \linea      \gcn{2}{1}{1}{2}        \linea\gnl
    \linea      \etiqueta{+}\gcmu                   \linea \gnl
    \gmu        \linea                  \linea\gnl
    \gcn{2}{1}{2}{1}      \linea                  \linea\gnl
    \gend
    -
    \gbeg{4}{7}
    \etiqueta{0}\gwcm{3}   \gnl
    \linea  \gvac{1}    \gcn{1}{1}{1}{2}    \gnl
    \linea  \gvac{1}    \gcmu   \gnl
    \linea  \gcn{2}{1}{3}{2}    \linea  \gnl
    \linea  \etiqueta{+}\gcmu   \linea  \gnl
    \linea  \cruce  \linea  \gnl
    \gmu  \linea  \linea    \gnl
    \gend
    +
    \gbeg{4}{7}
    \etiqueta{0}\gcmu \gnl
    \gcl{1} \gcn{1}{1}{1}{2}        \gnl
    \gcl{1} \gcmu                   \gnl
    \gcl{1} \gcl{1} \gcn{1}{1}{1}{2}\gnl
    \gcl{1} \gcl{1} \gcmu           \gnl
    \linea  \cruce  \linea          \gnl
    \gmu    \linea  \linea          \gnl
    \gend
    +
    \gbeg{4}{7}
    \etiqueta{+}\gwcm{3}   \gnl
    \linea  \gvac{1}    \gcn{1}{1}{1}{2}    \gnl
    \linea  \gvac{1}    \gcmu   \gnl
    \linea  \gcn{2}{1}{3}{2}    \linea  \gnl
    \linea  \etiqueta{0}\gcmu   \linea  \gnl
    \linea  \cruce  \linea  \gnl
    \gmu  \linea  \linea    \gnl
    \gend
    -
    \gbeg{4}{7}
    \gcmu \gnl
    \linea  \gcn{1}{1}{1}{2}    \gnl
    \linea  \gcmu   \gnl
    \gmu    \gcn{1}{1}{1}{2}    \gnl
    \gcn{2}{1}{2}{1}    \gcmu   \gnl
    \linea  \gvac{1}    \linea  \linea  \gnl
    \linea  \gvac{1}    \linea  \linea  \gnl
    \gend\\
    &=&
    \gbeg{4}{7}
    \gvac{1}\gwcm{3}       \gnl
    \gcn{3}{1}{3}{2}   \linea \gnl
    \gcmu       \gvac{1}\linea \gnl
    \linea      \gcn{2}{1}{1}{2}        \linea\gnl
    \linea      \etiqueta{+}\gcmu                   \linea \gnl
    \gmu        \linea                  \linea\gnl
    \gcn{2}{1}{2}{1}      \linea                  \linea\gnl
    \gend
    -
    \gbeg{4}{7}
    \etiqueta{0}\gwcm{3}   \gnl
    \linea  \gvac{1}    \gcn{1}{1}{1}{2}    \gnl
    \linea  \gvac{1}    \gcmu   \gnl
    \linea  \gcn{2}{1}{3}{2}    \linea  \gnl
    \linea  \etiqueta{+}\gcmu   \linea  \gnl
    \linea  \cruce  \linea  \gnl
    \gmu  \linea  \linea    \gnl
    \gend
    +
    \gbeg{4}{7}
    \etiqueta{0}\gcmu \gnl
    \gcl{1} \gcn{1}{1}{1}{2}        \gnl
    \gcl{1} \gcmu                   \gnl
    \gcl{1} \gcl{1} \gcn{1}{1}{1}{2}\gnl
    \gcl{1} \gcl{1} \gcmu           \gnl
    \linea  \cruce  \linea          \gnl
    \gmu    \linea  \linea          \gnl
    \gend
    +
    \gbeg{4}{7}
    \etiqueta{+}\gwcm{3}   \gnl
    \linea  \gvac{1}    \gcn{1}{1}{1}{2}    \gnl
    \linea  \gvac{1}    \gcmu   \gnl
    \linea  \gcn{2}{1}{3}{2}    \linea  \gnl
    \linea  \etiqueta{0}\gcmu   \linea  \gnl
    \linea  \cruce  \linea  \gnl
    \gmu  \linea  \linea    \gnl
    \gend
    -
    \gbeg{4}{7}
    \etiqueta{+}\gcmu \gnl
    \linea  \gcn{1}{1}{1}{2}    \gnl
    \linea  \gcmu   \gnl
    \gmu    \gcn{1}{1}{1}{2}    \gnl
    \gcn{2}{1}{2}{1}    \gcmu   \gnl
    \linea  \gvac{1}    \linea  \linea  \gnl
    \linea  \gvac{1}    \linea  \linea  \gnl
    \gend
    -
    \gbeg{4}{7}
    \etiqueta{0}\gcmu \gnl
    \linea  \gcn{1}{1}{1}{2}    \gnl
    \linea  \gcmu   \gnl
    \gmu    \gcn{1}{1}{1}{2}    \gnl
    \gcn{2}{1}{2}{1}    \gcmu   \gnl
    \linea  \gvac{1}    \linea  \linea  \gnl
    \linea  \gvac{1}    \linea  \linea  \gnl
    \gend
  \end{eqnarray*}
  We use that $C_1 = \cdots = C_{n-1} = 0$ once more to get modulo $h^{n+1}$
  \begin{eqnarray*}
  D &\equiv&
  \gbeg{4}{7}
    \gvac{1}\gwcm{3}       \gnl
    \gcn{3}{1}{3}{2}   \linea \gnl
    \gcmu       \gvac{1}\linea \gnl
    \linea      \gcn{2}{1}{1}{2}        \linea\gnl
    \linea      \etiqueta{+}\gcmu                   \linea \gnl
    \gmu        \linea                  \linea\gnl
    \gcn{2}{1}{2}{1}      \linea                  \linea\gnl
    \gend
    -
    \gbeg{4}{7}
    \etiqueta{0}\gwcm{3}   \gnl
    \linea  \gvac{1}    \gcn{1}{1}{1}{2}    \gnl
    \linea  \gvac{1}    \gcmu   \gnl
    \linea  \gcn{2}{1}{3}{2}    \linea  \gnl
    \linea  \etiqueta{+}\gcmu   \linea  \gnl
    \linea  \cruce  \linea  \gnl
    \gmu  \linea  \linea    \gnl
    \gend
    +
    \gbeg{4}{7}
    \etiqueta{0}\gcmu \gnl
    \gcl{1} \gcn{1}{1}{1}{2}        \gnl
    \gcl{1} \gcmu                   \gnl
    \gcl{1} \gcl{1} \gcn{1}{1}{1}{2}\gnl
    \gcl{1} \gcl{1} \gcmu           \gnl
    \linea  \cruce  \linea          \gnl
    \gmu    \linea  \linea          \gnl
    \gend
    +
    \gbeg{4}{7}
    \etiqueta{+}\gwcm{3}   \gnl
    \linea  \gvac{1}    \gcn{1}{1}{1}{2}    \gnl
    \linea  \gvac{1}    \gcmu   \gnl
    \linea  \gcn{2}{1}{3}{2}    \linea  \gnl
    \linea  \etiqueta{0}\gcmu   \linea  \gnl
    \linea  \cruce  \linea  \gnl
    \gmu  \linea  \linea    \gnl
    \gend
    -
    \gbeg{4}{7}
    \etiqueta{+}\gwcm{3}   \gnl
    \linea  \gvac{1}    \gcn{1}{1}{1}{2}    \gnl
    \linea  \gvac{1}    \gcmu   \gnl
    \linea  \gcn{2}{1}{3}{2}    \linea  \gnl
    \linea  \gcmu   \linea  \gnl
    \linea  \linea  \linea  \linea \gnl
    \gmu  \linea  \linea    \gnl
    \gend
    -
    \gbeg{4}{7}
    \etiqueta{0}\gcmu \gnl
    \linea  \gcn{1}{1}{1}{2}    \gnl
    \linea  \gcmu   \gnl
    \gmu    \gcn{1}{1}{1}{2}    \gnl
    \gcn{2}{1}{2}{1}    \gcmu   \gnl
    \linea  \gvac{1}    \linea  \linea  \gnl
    \linea  \gvac{1}    \linea  \linea  \gnl
    \gend
    \\
    &\equiv&
    \gbeg{4}{7}
    \gwcm{3}   \gnl
    \linea  \gvac{1}    \gcn{1}{1}{1}{2}    \gnl
    \linea  \gvac{1}    \gcmu   \gnl
    \linea  \gcn{2}{1}{3}{2}    \linea  \gnl
    \linea  \etiqueta{+}\gcmu   \linea  \gnl
    \linea  \linea  \linea  \linea \gnl
    \gmu  \linea  \linea    \gnl
    \gend
    -
    \gbeg{4}{7}
    \etiqueta{0}\gwcm{3}   \gnl
    \linea  \gvac{1}    \gcn{1}{1}{1}{2}    \gnl
    \linea  \gvac{1}    \gcmu   \gnl
    \linea  \gcn{2}{1}{3}{2}    \linea  \gnl
    \linea  \etiqueta{+}\gcmu   \linea  \gnl
    \linea  \cruce  \linea  \gnl
    \gmu  \linea  \linea    \gnl
    \gend
    +
    \gbeg{4}{7}
    \etiqueta{0}\gcmu \gnl
    \gcl{1} \gcn{1}{1}{1}{2}        \gnl
    \gcl{1} \gcmu                   \gnl
    \gcl{1} \gcl{1} \gcn{1}{1}{1}{2}\gnl
    \gcl{1} \gcl{1} \gcmu           \gnl
    \linea  \cruce  \linea          \gnl
    \gmu    \linea  \linea          \gnl
    \gend
    +
    \gbeg{4}{7}
    \etiqueta{+}\gwcm{3}   \gnl
    \linea  \gvac{1}    \gcn{1}{1}{1}{2}    \gnl
    \linea  \gvac{1}    \gcmu   \gnl
    \linea  \gcn{2}{1}{3}{2}    \linea  \gnl
    \linea  \etiqueta{0}\gcmu   \linea  \gnl
    \linea  \cruce  \linea  \gnl
    \gmu  \linea  \linea    \gnl
    \gend
    -
    \gbeg{4}{7}
    \etiqueta{+}\gwcm{3}   \gnl
    \linea  \gvac{1}    \gcn{1}{1}{1}{2}    \gnl
    \linea  \gvac{1}    \gcmu   \gnl
    \linea  \gcn{2}{1}{3}{2}    \linea  \gnl
    \linea  \gcmu   \linea  \gnl
    \linea  \linea  \linea  \linea \gnl
    \gmu  \linea  \linea    \gnl
    \gend
    -
    \gbeg{4}{7}
    \etiqueta{0}\gcmu \gnl
    \linea  \gcn{1}{1}{1}{2}    \gnl
    \linea  \gcmu   \gnl
    \gmu    \gcn{1}{1}{1}{2}    \gnl
    \gcn{2}{1}{2}{1}    \gcmu   \gnl
    \linea  \gvac{1}    \linea  \linea  \gnl
    \linea  \gvac{1}    \linea  \linea  \gnl
    \gend
    \\
    &\equiv&
    \gbeg{4}{7}
    \etiqueta{0}\gwcm{3}   \gnl
    \linea  \gvac{1}    \gcn{1}{1}{1}{2}    \gnl
    \linea  \gvac{1}    \gcmu   \gnl
    \linea  \gcn{2}{1}{3}{2}    \linea  \gnl
    \linea  \etiqueta{+}\gcmu   \linea  \gnl
    \linea  \linea  \linea  \linea \gnl
    \gmu  \linea  \linea    \gnl
    \gend
    -
    \gbeg{4}{7}
    \etiqueta{0}\gwcm{3}   \gnl
    \linea  \gvac{1}    \gcn{1}{1}{1}{2}    \gnl
    \linea  \gvac{1}    \gcmu   \gnl
    \linea  \gcn{2}{1}{3}{2}    \linea  \gnl
    \linea  \etiqueta{+}\gcmu   \linea  \gnl
    \linea  \cruce  \linea  \gnl
    \gmu  \linea  \linea    \gnl
    \gend
    +
    \gbeg{4}{7}
    \etiqueta{0}\gcmu \gnl
    \gcl{1} \gcn{1}{1}{1}{2}        \gnl
    \gcl{1} \gcmu                   \gnl
    \gcl{1} \gcl{1} \gcn{1}{1}{1}{2}\gnl
    \gcl{1} \gcl{1} \gcmu           \gnl
    \linea  \cruce  \linea          \gnl
    \gmu    \linea  \linea          \gnl
    \gend
    -
    \gbeg{4}{7}
    \etiqueta{0}\gcmu \gnl
    \linea  \gcn{1}{1}{1}{2}    \gnl
    \linea  \gcmu   \gnl
    \gmu    \gcn{1}{1}{1}{2}    \gnl
    \gcn{2}{1}{2}{1}    \gcmu   \gnl
    \linea  \gvac{1}    \linea  \linea  \gnl
    \linea  \gvac{1}    \linea  \linea  \gnl
    \gend
  \end{eqnarray*}
  Since $\Delta_0 = \Delta^{\op}_0$ we finally obtain
  \begin{eqnarray*}
    D
    &\equiv&
    \gbeg{4}{7}
    \etiqueta{0}\gwcm{3}   \gnl
    \linea  \gvac{1}    \gcn{1}{1}{1}{2}    \gnl
    \linea  \gvac{1}    \gcmu   \gnl
    \linea  \gcn{2}{1}{3}{2}    \linea  \gnl
    \linea  \gcmu   \linea  \gnl
    \linea  \linea  \linea  \linea \gnl
    \gmu  \linea  \linea    \gnl
    \gend
    -
    \gbeg{4}{7}
    \etiqueta{0}\gwcm{3}   \gnl
    \linea  \gvac{1}    \gcn{1}{1}{1}{2}    \gnl
    \linea  \gvac{1}    \gcmu   \gnl
    \linea  \gcn{2}{1}{3}{2}    \linea  \gnl
    \linea  \etiqueta{0}\gcmu   \linea  \gnl
    \linea  \linea  \linea  \linea \gnl
    \gmu  \linea  \linea    \gnl
    \gend
    +
    \gbeg{4}{7}
    \etiqueta{0}\gwcm{3}   \gnl
    \linea  \gvac{1}    \gcn{1}{1}{1}{2}    \gnl
    \linea  \gvac{1}    \gcmu   \gnl
    \linea  \gcn{2}{1}{3}{2}    \linea  \gnl
    \linea  \etiqueta{0}\gcmu   \linea  \gnl
    \linea  \linea  \linea  \linea \gnl
    \gmu  \linea  \linea    \gnl
    \gend
    -
    \gbeg{4}{7}
    \etiqueta{0}\gwcm{3}   \gnl
    \linea  \gvac{1}    \gcn{1}{1}{1}{2}    \gnl
    \linea  \gvac{1}    \gcmu   \gnl
    \linea  \gcn{2}{1}{3}{2}    \linea  \gnl
    \linea  \gcmu   \linea  \gnl
    \linea  \cruce  \linea  \gnl
    \gmu  \linea  \linea    \gnl
    \gend
    +
    \gbeg{4}{7}
    \etiqueta{0}\gcmu \gnl
    \gcl{1} \gcn{1}{1}{1}{2}        \gnl
    \gcl{1} \gcmu                   \gnl
    \gcl{1} \gcl{1} \gcn{1}{1}{1}{2}\gnl
    \gcl{1} \gcl{1} \gcmu           \gnl
    \linea  \cruce  \linea          \gnl
    \gmu    \linea  \linea          \gnl
    \gend
    -
    \gbeg{4}{7}
    \etiqueta{0}\gcmu \gnl
    \linea  \gcn{1}{1}{1}{2}    \gnl
    \linea  \gcmu   \gnl
    \gmu    \gcn{1}{1}{1}{2}    \gnl
    \gcn{2}{1}{2}{1}    \gcmu   \gnl
    \linea  \gvac{1}    \linea  \linea  \gnl
    \linea  \gvac{1}    \linea  \linea  \gnl
    \gend\\
    &=&
    \gbeg{4}{7}
    \etiqueta{0}\gwcm{3}   \gnl
    \linea  \gvac{1}    \gcn{1}{1}{1}{2}    \gnl
    \linea  \gvac{1}    \gcmu   \gnl
    \linea  \gcn{2}{1}{3}{2}    \linea  \gnl
    \linea  \gcmu   \linea  \gnl
    \linea  \linea  \linea  \linea \gnl
    \gmu  \linea  \linea    \gnl
    \gend
    -
    \gbeg{4}{7}
    \etiqueta{0}\gwcm{3}   \gnl
    \linea  \gvac{1}    \gcn{1}{1}{1}{2}    \gnl
    \linea  \gvac{1}    \gcmu   \gnl
    \linea  \gcn{2}{1}{3}{2}    \linea  \gnl
    \linea  \gcmu   \linea  \gnl
    \linea  \cruce  \linea  \gnl
    \gmu  \linea  \linea    \gnl
    \gend
    +
    \gbeg{4}{7}
    \etiqueta{0}\gcmu \gnl
    \gcl{1} \gcn{1}{1}{1}{2}        \gnl
    \gcl{1} \gcmu                   \gnl
    \gcl{1} \gcl{1} \gcn{1}{1}{1}{2}\gnl
    \gcl{1} \gcl{1} \gcmu           \gnl
    \linea  \cruce  \linea          \gnl
    \gmu    \linea  \linea          \gnl
    \gend
    -
    \gbeg{4}{7}
    \etiqueta{0}\gcmu \gnl
    \linea  \gcn{1}{1}{1}{2}    \gnl
    \linea  \gcmu   \gnl
    \gmu    \gcn{1}{1}{1}{2}    \gnl
    \gcn{2}{1}{2}{1}    \gcmu   \gnl
    \linea  \gvac{1}    \linea  \linea  \gnl
    \linea  \gvac{1}    \linea  \linea  \gnl
    \gend
  \end{eqnarray*}
  Evaluating at $a \in \m$ and simplifying it follows that
\begin{displaymath}
    \gbeg{4}{7}
    \got{6}{a}           \gnl
    \gvac{2}            \gcmu                   \gnl
    \gcn{3}{1}{5}{4}    \gcl{1}                 \gnl
    \gvac{1}            \gcmu       \gcl{1}     \gnl
    \gcn{2}{1}{3}{2}    \gcl{1}     \gcl{1}     \gnl
    \gcmu               \gcl{1}     \gcl{1}     \gnl
    \gmu                \gcl{1}     \gcl{1}     \gnl
    \gend
    -
    \gbeg{4}{7}
    \got{4}{a}                                                  \gnl
    \gvac{1}            \gcmu                                   \gnl
    \gcn{1}{1}{3}{2}    \gcn{1}{1}{3}{4}                        \gnl
    \gcmu               \gcmu                                   \gnl
    \gmu                \linea              \linea              \gnl
    \gcn{2}{1}{2}{1}    \linea              \linea              \gnl
    \linea              \gvac{1}            \linea      \linea  \gnl
    \gend
    \equiv
    \gbeg{3}{5}
    \got{4}{a}                          \gnl
    \gvac{1}            \gcmu           \gnl
    \gcn{2}{1}{3}{2}    \linea          \gnl
    \gcmu               \linea          \gnl
    \linea              \linea   \linea \gnl
    \gend
    -
    \gbeg{4}{5}
    \got{3}{a}                  \gnl
    \gvac{1}            \gcmu   \gnl
    \gcn{2}{1}{3}{2}    \linea  \gnl
    \gcmu               \linea  \gnl
    \gbr                \linea \gnl
    \gend
    +
    \gbeg{3}{5}
    \got{2}{a}                  \gnl
    \gcmu                       \gnl
    \linea  \gcn{1}{1}{1}{2}    \gnl
    \linea  \gcmu               \gnl
    \gbr    \linea              \gnl
    \gend
    -
    \gbeg{3}{5}
    \got{2}{a}                  \gnl
    \gcmu                       \gnl
    \linea  \gcn{1}{1}{1}{2}    \gnl
    \linea  \gcmu               \gnl
    \linea  \linea      \linea  \gnl
    \gend
    \quad (\modulus h^{n+1})
\end{displaymath}
\qed
\end{proof}

\begin{lemma}\label{lem:skew}
  If $C_1 = \cdots = C_{n-1} = 0$ then for any $a \in \m$ we have
  \begin{displaymath}
    C_n(a) \in \m \wedge \m \otimes U(\m).
  \end{displaymath}
\end{lemma}
\begin{proof}
  Lemma~\ref{lem:P1} implies that, modulo $h^{n+1}$, the element
  \begin{displaymath}
    \gbeg{4}{7}
    \got{6}{a}           \gnl
    \gvac{2}            \gcmu                   \gnl
    \gcn{3}{1}{5}{4}    \gcl{1}                 \gnl
    \gvac{1}            \gcmu       \gcl{1}     \gnl
    \gcn{2}{1}{3}{2}    \gcl{1}     \gcl{1}     \gnl
    \gcmu               \gcl{1}     \gcl{1}     \gnl
    \gmu                \gcl{1}     \gcl{1}     \gnl
    \gend
    -
    \gbeg{4}{7}
    \got{4}{a}                                                  \gnl
    \gvac{1}            \gcmu                                   \gnl
    \gcn{1}{1}{3}{2}    \gcn{1}{1}{3}{4}                        \gnl
    \gcmu               \gcmu                                   \gnl
    \gmu                \linea              \linea              \gnl
    \gcn{2}{1}{2}{1}    \linea              \linea              \gnl
    \linea              \gvac{1}            \linea      \linea  \gnl
    \gend
  \end{displaymath}
  is skew-symmetric with respect to the first and second slots. Therefore
  \begin{displaymath}
    \gbeg{4}{7}
    \got{6}{a}                                                  \gnl
    \gvac{2}    \gcmu       \gnl
    \gcn{3}{1}{5}{4}        \gcl{1}     \gnl
    \gvac{1}    \gcmu       \gcl{1}     \gnl
    \gcn{2}{1}{3}{2}        \gcl{1}     \gcl{1} \gnl
    \gcmu       \gcl{1}     \gcl{1}     \gnl
    \gmu        \gcl{1}     \gcl{1}     \gnl
    \gend
    -
    \gbeg{4}{7}
    \got{4}{a}                                                  \gnl
    \gvac{1}\gcmu \gnl
    \gcn{1}{1}{3}{2} \gcn{1}{1}{3}{4} \gnl
    \gcmu   \gcmu   \gnl
    \gmu \linea    \linea   \gnl
    \gcn{2}{1}{2}{1} \linea  \linea \gnl
    \linea  \gvac{1}    \linea  \linea  \gnl
    \gend
    \equiv
    -
    \gbeg{4}{8}
    \got{5}{a}                                                  \gnl
    \gvac{1}    \gwcm{3}       \gnl
    \gcn{3}{1}{3}{2}        \gcl{1}     \gnl
    \gcmu       \gvac{1}    \gcl{1}     \gnl
    \gbr \gvac{1} \linea \gnl
    \gcl{1}  \gcn{2}{1}{1}{2}    \gcl{1} \gnl
    \linea \gcmu        \gcl{1}     \gnl
    \linea \gmu        \gcl{1}      \gnl
    \gend
    +
    \gbeg{4}{8}
    \got{2}{a}                                                  \gnl
    \gcmu \gnl
    \linea  \gcn{1}{1}{1}{3}    \gnl
    \linea \gwcm{3} \gnl
    \cruce \gvac{1}    \linea \gnl
    \linea \gcn{2}{1}{1}{2} \linea \gnl
    \linea \gcmu \linea \gnl
    \linea \gmu\linea \gnl
    \gend
  \end{displaymath}
  Using Lemma~\ref{lem:P1} and this equality we get
  \begin{eqnarray*}
    \gbeg{5}{8}
    \got{6}{a}                                                  \gnl
    \gvac{1}\gwcm{4} \gnl
    \gvac{1}\linea \gvac{2} \linea \gnl
     \gwcm{3} \gvac{1} \linea \gnl
    \gcn{1}{1}{1}{2} \gcn{3}{1}{3}{4} \linea \gnl
    \gcmu   \gcmu   \linea \gnl
    \gmu \gmu \linea \gnl
    \gcn{1}{1}{2}{2} \gcn{3}{1}{4}{4} \linea \gnl
    \gend
    -
    \gbeg{5}{8}
    \got{5}{a}\gnl
    \gvac{1} \gwcm{3} \gnl
    \gcn{3}{1}{3}{2} \linea \gnl
    \gcmu \gwcm{3} \gnl
    \gmu \gcn{2}{1}{1}{2} \linea \gnl
    \gcn{1}{1}{2}{2} \gvac{1}    \gcmu \linea \gnl
    \gcn{1}{1}{2}{2} \gvac{1} \gmu \linea \gnl
    \gcn{1}{1}{2}{2} \gcn{3}{1}{4}{4} \linea \gnl
    \gend
    &\equiv&
    \gbeg{4}{8}
    \got{5}{a}\gnl
    \gvac{1}    \gwcm{3}       \gnl
    \gcn{3}{1}{3}{2}        \gcl{1}     \gnl
    \gcmu       \gvac{1}    \gcl{1}     \gnl
    \linea \linea \gvac{1} \linea \gnl
    \gcl{1}  \gcn{2}{1}{1}{2}    \gcl{1} \gnl
    \linea \gcmu        \gcl{1}     \gnl
    \linea \gmu        \gcl{1}      \gnl
    \gend
    -
    \gbeg{4}{8}
    \got{5}{a}\gnl
    \gvac{1}    \gwcm{3}       \gnl
    \gcn{3}{1}{3}{2}        \gcl{1}     \gnl
    \gcmu       \gvac{1}    \gcl{1}     \gnl
    \gbr \gvac{1} \linea \gnl
    \gcl{1}  \gcn{2}{1}{1}{2}    \gcl{1} \gnl
    \linea \gcmu        \gcl{1}     \gnl
    \linea \gmu        \gcl{1}      \gnl
    \gend
    +
    \gbeg{4}{8}
    \got{2}{a}\gnl
    \gcmu \gnl
    \linea  \gcn{1}{1}{1}{3}    \gnl
    \linea \gwcm{3} \gnl
    \cruce \gvac{1}    \linea \gnl
    \linea \gcn{2}{1}{1}{2} \linea \gnl
    \linea \gcmu \linea \gnl
    \linea \gmu\linea \gnl
    \gend
    -
    \gbeg{4}{8}
    \got{2}{a}\gnl
    \gcmu \gnl
    \linea  \gcn{1}{1}{1}{3}    \gnl
    \linea \gwcm{3} \gnl
    \linea \linea \gvac{1}    \linea \gnl
    \linea \gcn{2}{1}{1}{2} \linea \gnl
    \linea \gcmu \linea \gnl
    \linea \gmu\linea \gnl
    \gend
    \\
    &\equiv&
    \gbeg{4}{7}
    \got{5}{a}\gnl
    \gvac{1}    \gwcm{3}       \gnl
    \gcn{3}{1}{3}{2}        \gcl{1}     \gnl
    \gcmu       \gvac{1}    \gcl{1}     \gnl
    \gcl{1}  \gcn{2}{1}{1}{2}    \gcl{1} \gnl
    \linea \gcmu        \gcl{1}     \gnl
    \linea \gmu        \gcl{1}      \gnl
    \gend
    +
    \gbeg{4}{7}
    \got{6}{a}\gnl
    \gvac{2}    \gcmu       \gnl
    \gcn{3}{1}{5}{4}        \gcl{1}     \gnl
    \gvac{1}    \gcmu       \gcl{1}     \gnl
    \gcn{2}{1}{3}{2}        \gcl{1}     \gcl{1} \gnl
    \gcmu       \gcl{1}     \gcl{1}     \gnl
    \gmu        \gcl{1}     \gcl{1}     \gnl
    \gend
    -
    \gbeg{4}{7}
    \got{4}{a}\gnl
    \gvac{1}\gcmu \gnl
    \gcn{1}{1}{3}{2} \gcn{1}{1}{3}{4} \gnl
    \gcmu   \gcmu   \gnl
    \gmu \linea    \linea   \gnl
    \gcn{2}{1}{2}{1} \linea  \linea \gnl
    \linea  \gvac{1}    \linea  \linea  \gnl
    \gend
    -
    \gbeg{4}{7}
    \got{2}{a}\gnl
    \gcmu \gnl
    \linea  \gcn{1}{1}{1}{3}    \gnl
    \linea \gwcm{3} \gnl
    \linea \gcn{2}{1}{1}{2} \linea \gnl
    \linea \gcmu \linea \gnl
    \linea \gmu\linea \gnl
    \gend
  \end{eqnarray*}
  Since $C_1 = \cdots = C_{n-1} = 0$ this implies that
  \begin{displaymath}
    \gbeg{5}{7}
    \got{6}{a}\gnl
    \gvac{1}\gwcm{4} \gnl
    \gvac{1}\linea \gvac{2} \linea \gnl
     \gwcm{3} \gvac{1} \linea \gnl
    \gcn{1}{1}{1}{2} \gcn{3}{1}{3}{4} \linea \gnl
    \etiqueta{0}\gcmu   \etiqueta{0}\gcmu   \linea \gnl
    \etiqueta{0}\gmu \etiqueta{0}\gmu \linea \gnl
    \gend
    -
    \gbeg{5}{7}
    \got{5}{a}\gnl
    \gvac{1} \gwcm{3} \gnl
    \gcn{3}{1}{3}{2} \linea \gnl
    \etiqueta{0}\gcmu \gwcm{3} \gnl
    \etiqueta{0}\gmu \gcn{2}{1}{1}{2} \linea \gnl
    \gcn{1}{1}{2}{2} \gvac{1}    \etiqueta{0}\gcmu \linea \gnl
    \gcn{1}{1}{2}{2} \gvac{1} \etiqueta{0}\gmu \linea \gnl
    \gend
    \equiv
    \gbeg{4}{7}
    \got{5}{a}\gnl
    \gvac{1}    \gwcm{3}       \gnl
    \gcn{3}{1}{3}{2}        \gcl{1}     \gnl
    \gcmu       \gvac{1}    \gcl{1}     \gnl
    \gcl{1}  \gcn{2}{1}{1}{2}    \gcl{1} \gnl
    \linea \etiqueta{0}\gcmu        \gcl{1}     \gnl
    \linea \etiqueta{0}\gmu        \gcl{1}      \gnl
    \gend
    +
    \gbeg{4}{7}
    \got{6}{a}\gnl
    \gvac{2}    \gcmu       \gnl
    \gcn{3}{1}{5}{4}        \gcl{1}     \gnl
    \gvac{1}    \gcmu       \gcl{1}     \gnl
    \gcn{2}{1}{3}{2}        \gcl{1}     \gcl{1} \gnl
    \etiqueta{0}\gcmu       \gcl{1}     \gcl{1}     \gnl
    \etiqueta{0}\gmu        \gcl{1}     \gcl{1}     \gnl
    \gend
    -
    \gbeg{4}{7}
    \got{4}{a}\gnl
    \gvac{1}\gcmu \gnl
    \gcn{1}{1}{3}{2} \gcn{1}{1}{3}{4} \gnl
    \etiqueta{0}\gcmu   \gcmu   \gnl
    \etiqueta{0}\gmu \linea    \linea   \gnl
    \gcn{2}{1}{2}{1} \linea  \linea \gnl
    \linea  \gvac{1}    \linea  \linea  \gnl
    \gend
    -
    \gbeg{4}{7}
    \got{2}{a}\gnl
    \gcmu \gnl
    \linea  \gcn{1}{1}{1}{3}    \gnl
    \linea \gwcm{3} \gnl
    \linea \gcn{2}{1}{1}{2} \linea \gnl
    \linea \etiqueta{0}\gcmu \linea \gnl
    \linea \etiqueta{0}\gmu\linea \gnl
    \gend
  \end{displaymath}
  i.e.,
  \begin{displaymath}
  C_n(a) \in \ker(Q \otimes Q \otimes \Id - Q \otimes \Id \otimes \Id - \Id \otimes Q \otimes \Id).
  \end{displaymath}
The space $U(\m) \otimes U(\m) \otimes U(\m)$ is spanned by $\{a^i \otimes b^j \otimes x \mid a,b \in \m, i,j\geq 0, x\in U(\m)\}$ and
  \begin{displaymath}
    (Q \otimes Q \otimes \Id - Q \otimes \Id \otimes \Id - \Id \otimes Q \otimes \Id)(a^i \otimes b^j \otimes x) = (2^{i+j}-2^i - 2^j) a^i \otimes b^j \otimes x.
  \end{displaymath}
  This shows that $C_n(a) \in \m \otimes \m  \otimes U(\m)$. Now, Lemma~\ref{lem:P1} implies that the element
  \begin{displaymath}
    2\left(
    \gbeg{3}{4}
    \got{4}{a}\gnl
    \gvac{1}    \gcmu        \gnl
    \gcn{2}{1}{3}{2}    \linea      \gnl
    \gcmu               \linea      \gnl
    \gend
    -
    \gbeg{3}{4}
    \got{2}{a}\gnl
    \gcmu       \gnl
    \linea      \gcn{1}{1}{1}{2}        \gnl
    \linea      \gcmu                   \gnl
    \gend
    \right)
    =
    \gbeg{4}{7}
    \got{6}{a}\gnl
    \gvac{2}    \gcmu       \gnl
    \gcn{3}{1}{5}{4}        \gcl{1}     \gnl
    \gvac{1}    \gcmu       \gcl{1}     \gnl
    \gcn{2}{1}{3}{2}        \gcl{1}     \gcl{1} \gnl
    \etiqueta{0}\gcmu       \gcl{1}     \gcl{1}     \gnl
    \etiqueta{0}\gmu        \gcl{1}     \gcl{1}     \gnl
    \gend
    -
    \gbeg{4}{7}
    \got{4}{a}\gnl
    \gvac{1}\gcmu \gnl
    \gcn{1}{1}{3}{2} \gcn{1}{1}{3}{4} \gnl
    \etiqueta{0}\gcmu   \gcmu   \gnl
    \etiqueta{0}\gmu \linea    \linea   \gnl
    \gcn{2}{1}{2}{1} \linea  \linea \gnl
    \linea  \gvac{1}    \linea  \linea  \gnl
    \gend
    =
    \gbeg{4}{7}
    \got{6}{a}\gnl
    \gvac{2}    \gcmu       \gnl
    \gcn{3}{1}{5}{4}        \gcl{1}     \gnl
    \gvac{1}    \gcmu       \gcl{1}     \gnl
    \gcn{2}{1}{3}{2}        \gcl{1}     \gcl{1} \gnl
    \gcmu       \gcl{1}     \gcl{1}     \gnl
    \gmu        \gcl{1}     \gcl{1}     \gnl
    \gend
    -
    \gbeg{4}{7}
    \got{4}{a}\gnl
    \gvac{1}\gcmu \gnl
    \gcn{1}{1}{3}{2} \gcn{1}{1}{3}{4} \gnl
    \gcmu   \gcmu   \gnl
    \gmu \linea    \linea   \gnl
    \gcn{2}{1}{2}{1} \linea  \linea \gnl
    \linea  \gvac{1}    \linea  \linea  \gnl
    \gend
  \end{displaymath}
  is skew-symmetric with respect to the first and second slot. Thus $C_n(a) \in \m \wedge \m \otimes U(\m)$. \qed
 \end{proof}

\begin{lemma}
\label{lem:P2}
  If $C_1 = \cdots = C_{n-1} = 0$ then for any $a \in \m$ we have
  \begin{displaymath}
    \gbeg{4}{7}
    \got{4}{a}\gnl
    \gvac{1}\gcmu \gnl
    \gcn{1}{1}{3}{2} \gcn{1}{1}{3}{4} \gnl
    \gcmu   \gcmu   \gnl
    \linea    \linea  \gmu  \gnl
    \linea  \linea \gcn{2}{1}{2}{3} \gnl
    \linea      \linea  \gvac{1} \linea  \gnl
    \gend
    -
    \gbeg{4}{7}
    \got{2}{a}\gnl
    \gcmu \gnl
    \linea \gcn{1}{1}{1}{2} \gnl
    \linea \gcmu \gnl
    \linea \linea \gcn{1}{1}{1}{2} \gnl
    \linea \linea \gcmu \gnl
    \linea \linea \gmu \gnl
    \gend
    \equiv
    \gbeg{3}{5}
    \got{4}{a}\gnl
    \gvac{1}\gcmu   \gnl
    \gcn{2}{1}{3}{2}    \linea \gnl
    \gcmu   \linea  \gnl
    \linea \linea   \linea \gnl
    \gend
    -
    \gbeg{3}{5}
   \got{2}{a}\gnl
    \gcmu \gnl
    \linea  \gcn{1}{1}{1}{2}    \gnl
    \linea  \gcmu   \gnl
    \linea    \linea  \linea \gnl
    \gend
    +
    \gbeg{3}{5}
    \got{2}{a}\gnl
    \gcmu \gnl
    \linea  \gcn{1}{1}{1}{2}    \gnl
    \linea  \gcmu   \gnl
    \linea  \cruce \gnl
    \gend
    -
    \gbeg{3}{5}
     \got{4}{a}\gnl
    \gvac{1}\gcmu   \gnl
    \gcn{2}{1}{3}{2}    \linea \gnl
    \gcmu   \linea  \gnl
    \linea \cruce \gnl
    \gend
    \quad (\modulus h^{n+1})
  \end{displaymath}
\end{lemma}
\begin{proof}
Let $D:=
    \gbeg{4}{7}
    \gvac{1}\gcmu \gnl
    \gcn{1}{1}{3}{2} \gcn{1}{1}{3}{4} \gnl
    \gcmu   \gcmu   \gnl
    \linea    \linea  \gmu  \gnl
    \linea  \linea \gcn{2}{1}{2}{3} \gnl
    \linea      \linea  \gvac{1} \linea  \gnl
    \gend
    -
    \gbeg{4}{7}
    \gcmu \gnl
    \linea \gcn{1}{1}{1}{2} \gnl
    \linea \gcmu \gnl
    \linea \linea \gcn{1}{1}{1}{2} \gnl
    \linea \linea \gcmu \gnl
    \linea \linea \gmu \gnl
    \gend$. Lemma~\ref{lem:basic} and the coassociativity of the comultiplication of $U(\m)$ imply that
    \begin{eqnarray*}
    D &=&
    \gbeg{4}{6}
    \gvac{2} \gcmu \gnl
    \gcn{3}{1}{5}{4} \linea \gnl
    \gvac{1} \gcmu \linea \gnl
    \gcn{2}{1}{3}{2} \gmu \gnl
    \gcmu \gcn{2}{1}{2}{3} \gnl
    \linea \linea \gvac{1} \linea \gnl
    \gend
    -
    \gbeg{4}{6}
    \gwcm{3} \gnl
    \linea \gcn{1}{1}{3}{4} \gnl
    \linea \gvac{1} \gcmu \gnl
    \linea \gcn{2}{1}{3}{2} \linea\gnl
    \linea \gcmu \linea \gnl
    \linea \linea \gmu \gnl
    \gend
    =
    \gbeg{4}{7}
    \gvac{2} \gcmu \gnl
    \gcn{3}{1}{5}{4} \linea \gnl
    \gvac{1} \gcmu \linea \gnl
    \gcn{2}{1}{3}{2} \gmu \gnl
    \gcmu \gcn{2}{1}{2}{3} \gnl
    \linea \linea \gvac{1} \linea \gnl
    \linea \linea \gvac{1} \linea \gnl
    \gend
    -
    \gbeg{4}{7}
    \gwcm{3} \gnl
    \linea \gcn{1}{1}{3}{4} \gnl
    \linea \gvac{1} \gcmu \gnl
    \linea \gcn{2}{1}{3}{2} \linea\gnl
    \linea \etiqueta{+}\gcmu \linea \gnl
    \linea \linea \gmu \gnl
    \linea \linea \gcn{1}{1}{2}{3} \gnl
    \gend
    -
    \gbeg{4}{7}
    \gwcm{3} \gnl
    \linea \gcn{1}{1}{3}{4} \gnl
    \linea \gvac{1} \gcmu \gnl
    \linea \gcn{2}{1}{3}{2} \linea\gnl
    \linea \etiqueta{0}\gcmu \linea \gnl
    \linea \cruce \linea \gnl
    \linea \linea \gmu \gnl
    \gend\\
    &=&
    \gbeg{4}{7}
    \gvac{2} \gcmu \gnl
    \gcn{3}{1}{5}{4} \linea \gnl
    \gvac{1} \gcmu \linea \gnl
    \gcn{2}{1}{3}{2} \gmu \gnl
    \gcmu \gcn{2}{1}{2}{3} \gnl
    \linea \linea \gvac{1} \linea \gnl
    \linea \linea \gvac{1} \linea \gnl
    \gend
    -
    \gbeg{4}{7}
    \gwcm{3} \gnl
    \linea \gcn{1}{1}{3}{4} \gnl
    \linea \gvac{1} \gcmu \gnl
    \linea \gcn{2}{1}{3}{2} \linea\gnl
    \linea \etiqueta{+}\gcmu \linea \gnl
    \linea \linea \gmu \gnl
    \linea \linea \gcn{1}{1}{2}{3} \gnl
    \gend
    +
    \gbeg{4}{7}
    \gwcm{3} \gnl
    \linea \gcn{1}{1}{3}{4} \gnl
    \linea \gvac{1} \gcmu \gnl
    \linea \gcn{2}{1}{3}{2} \linea\gnl
    \linea \etiqueta{+}\gcmu \linea \gnl
    \linea \cruce \linea \gnl
    \linea \linea \gmu \gnl
    \gend
    -
    \gbeg{4}{7}
    \gwcm{3} \gnl
    \linea \gcn{1}{1}{3}{4} \gnl
    \linea \gvac{1} \gcmu \gnl
    \linea \gcn{2}{1}{3}{2} \linea\gnl
    \linea \gcmu \linea \gnl
    \linea \cruce \linea \gnl
    \linea \linea \gmu \gnl
    \gend
    \end{eqnarray*}
    Using Lemma~\ref{lem:basic} again,
    \begin{eqnarray*}
    D &=&
    \gbeg{4}{7}
    \gvac{2} \gcmu \gnl
    \gcn{3}{1}{5}{4} \linea \gnl
    \gvac{1} \gcmu \linea \gnl
    \gcn{2}{1}{3}{2} \gmu \gnl
    \gcmu \gcn{2}{1}{2}{3} \gnl
    \linea \linea \gvac{1} \linea \gnl
    \linea \linea \gvac{1} \linea \gnl
    \gend
    -
    \gbeg{4}{7}
    \gwcm{3} \gnl
    \linea \gcn{1}{1}{3}{4} \gnl
    \linea \gvac{1} \gcmu \gnl
    \linea \gcn{2}{1}{3}{2} \linea\gnl
    \linea \etiqueta{+}\gcmu \linea \gnl
    \linea \linea \gmu \gnl
    \linea \linea \gcn{1}{1}{2}{3} \gnl
    \gend
    +
    \gbeg{4}{7}
    \gwcm{3} \gnl
    \linea \gcn{1}{1}{3}{4} \gnl
    \linea \gvac{1} \gcmu \gnl
    \linea \gcn{2}{1}{3}{2} \linea\gnl
    \linea \etiqueta{+}\gcmu \linea \gnl
    \linea \cruce \linea \gnl
    \linea \linea \gmu \gnl
    \gend
    -
    \gbeg{4}{7}
    \gvac{2}    \gcmu       \gnl
    \gcn{3}{1}{5}{4}        \gcl{1}     \gnl
    \gvac{1}    \gcmu       \gcl{1}     \gnl
    \gcn{2}{1}{3}{2}        \gcl{1}     \gcl{1} \gnl
    \gcmu       \gcl{1}     \gcl{1}     \gnl
    \gcl{1}     \gbr        \gcl{1}     \gnl
    \gcl{1}     \gcl{1}     \gmu        \gnl
    \gend
    \\
    &\equiv&
    \gbeg{4}{7}
    \gvac{2} \gcmu \gnl
    \gcn{3}{1}{5}{4} \linea \gnl
    \gvac{1} \gcmu \linea \gnl
    \gcn{2}{1}{3}{2} \gmu \gnl
    \gcmu \gcn{2}{1}{2}{3} \gnl
    \linea \linea \gvac{1} \linea \gnl
    \linea \linea \gvac{1} \linea \gnl
    \gend
    -
    \gbeg{4}{7}
    \gwcm{3} \gnl
    \linea \gcn{1}{1}{3}{4} \gnl
    \linea \gvac{1} \gcmu \gnl
    \linea \gcn{2}{1}{3}{2} \linea\gnl
    \linea \etiqueta{+}\gcmu \linea \gnl
    \linea \linea \gmu \gnl
    \linea \linea \gcn{1}{1}{2}{3} \gnl
    \gend
    +
    \gbeg{4}{7}
    \gwcm{3} \gnl
    \linea \gcn{1}{1}{3}{4} \gnl
    \linea \gvac{1} \gcmu \gnl
    \linea \gcn{2}{1}{3}{2} \linea\gnl
    \linea \etiqueta{+}\gcmu \linea \gnl
    \linea \cruce \linea \gnl
    \linea \linea \gmu \gnl
    \gend
    -
    \gbeg{4}{7}
    \gvac{2}    \etiqueta{0}\gcmu       \gnl
    \gcn{3}{1}{5}{4}        \gcl{1}     \gnl
    \gvac{1}    \gcmu       \gcl{1}     \gnl
    \gcn{2}{1}{3}{2}        \gcl{1}     \gcl{1} \gnl
    \gcmu       \gcl{1}     \gcl{1}     \gnl
    \gcl{1}     \gbr        \gcl{1}     \gnl
    \gcl{1}     \gcl{1}     \gmu        \gnl
    \gend
    -
    \gbeg{4}{7}
    \gvac{1}    \etiqueta{+}\gwcm{3}       \gnl
    \gcn{3}{1}{3}{2}        \gcl{1}     \gnl
    \gcmu      \gvac{1}     \linea     \gnl
    \linea      \gcn{1}{1}{1}{2} \gvac{1}       \linea \gnl
    \linea      \gcmu       \linea      \gnl
    \linea      \cruce       \linea     \gnl
    \linea      \linea      \gmu   \gnl
    \gend \quad (\modulus{h^{n+1}})\\
    &=&
    \gbeg{4}{7}
    \gvac{2} \gcmu \gnl
    \gcn{3}{1}{5}{4} \linea \gnl
    \gvac{1} \gcmu \linea \gnl
    \gcn{2}{1}{3}{2} \gmu \gnl
    \gcmu \gcn{2}{1}{2}{3} \gnl
    \linea \linea \gvac{1} \linea \gnl
    \linea \linea \gvac{1} \linea \gnl
    \gend
    -
    \gbeg{4}{7}
    \gwcm{3} \gnl
    \linea \gcn{1}{1}{3}{4} \gnl
    \linea \gvac{1} \gcmu \gnl
    \linea \gcn{2}{1}{3}{2} \linea\gnl
    \linea \etiqueta{+}\gcmu \linea \gnl
    \linea \linea \gmu \gnl
    \linea \linea \gcn{1}{1}{2}{3} \gnl
    \gend
    +
    \gbeg{4}{7}
    \gwcm{3} \gnl
    \linea \gcn{1}{1}{3}{4} \gnl
    \linea \gvac{1} \gcmu \gnl
    \linea \gcn{2}{1}{3}{2} \linea\gnl
    \linea \etiqueta{+}\gcmu \linea \gnl
    \linea \cruce \linea \gnl
    \linea \linea \gmu \gnl
    \gend
    -
    \gbeg{4}{7}
    \gvac{2}    \etiqueta{0}\gcmu       \gnl
    \gcn{3}{1}{5}{4}        \gcl{1}     \gnl
    \gvac{1}    \gcmu       \gcl{1}     \gnl
    \gcn{2}{1}{3}{2}        \gcl{1}     \gcl{1} \gnl
    \gcmu       \gcl{1}     \gcl{1}     \gnl
    \gcl{1}     \gbr        \gcl{1}     \gnl
    \gcl{1}     \gcl{1}     \gmu        \gnl
    \gend
    -
    \gbeg{4}{7}
    \gvac{1}    \etiqueta{+}\gwcm{3}       \gnl
    \gcn{3}{1}{3}{2}        \gcl{1}     \gnl
    \gcmu      \gvac{1}     \linea     \gnl
    \linea      \gcn{1}{1}{1}{2} \gvac{1}       \linea \gnl
    \linea      \etiqueta{+}\gcmu       \linea      \gnl
    \linea      \cruce       \linea     \gnl
    \linea      \linea      \gmu   \gnl
    \gend
    -
    \gbeg{4}{7}
    \gvac{1} \etiqueta{+}\gwcm{3} \gnl
    \gcn{3}{1}{3}{2} \linea \gnl
    \gcmu \gvac{1} \linea \gnl
    \linea \gcn{2}{1}{1}{2} \linea \gnl
    \linea \etiqueta{0}\gcmu \linea \gnl
    \linea \linea \linea \linea\gnl
    \linea \linea \gmu \gnl
    \gend
    \end{eqnarray*}
    where the second congruence follows from the fact that $C_1 = \cdots = C_{n-1} = 0$ and the last equality is a consequence of the coassociativity of $U(\m)$. Again, using $C_1 = \cdots = C_{n-1} = 0$ we deduce that modulo $h^{n+1}$
    \begin{eqnarray*}
    D &\equiv& \gbeg{4}{7}
    \gvac{2} \etiqueta{0}\gcmu \gnl
    \gcn{3}{1}{5}{4} \linea \gnl
    \gvac{1} \gcmu \linea \gnl
    \gcn{2}{1}{3}{2} \gmu \gnl
    \gcmu \gcn{2}{1}{2}{3} \gnl
    \linea \linea \gvac{1} \linea \gnl
    \linea \linea \gvac{1} \linea \gnl
    \gend
    +
    \gbeg{4}{7}
    \gvac{1} \etiqueta{+}\gwcm{3} \gnl
    \gcn{3}{1}{3}{2} \linea \gnl
    \gcmu \gvac{1} \linea \gnl
    \linea \gcn{2}{1}{1}{2} \linea \gnl
    \linea \etiqueta{+}\gcmu \linea \gnl
    \linea \linea \linea \linea\gnl
    \linea \linea \gmu \gnl
    \gend
    -
    \gbeg{4}{7}
    \gwcm{3} \gnl
    \linea \gcn{1}{1}{3}{4} \gnl
    \linea \gvac{1} \gcmu \gnl
    \linea \gcn{2}{1}{3}{2} \linea\gnl
    \linea \etiqueta{+}\gcmu \linea \gnl
    \linea \linea \gmu \gnl
    \linea \linea \gcn{1}{1}{2}{3} \gnl
    \gend
    +
    \gbeg{4}{7}
    \gwcm{3} \gnl
    \linea \gcn{1}{1}{3}{4} \gnl
    \linea \gvac{1} \gcmu \gnl
    \linea \gcn{2}{1}{3}{2} \linea\gnl
    \linea \etiqueta{+}\gcmu \linea \gnl
    \linea \cruce \linea \gnl
    \linea \linea \gmu \gnl
    \gend
    -
    \gbeg{4}{7}
    \gvac{2}    \etiqueta{0}\gcmu       \gnl
    \gcn{3}{1}{5}{4}        \gcl{1}     \gnl
    \gvac{1}    \gcmu       \gcl{1}     \gnl
    \gcn{2}{1}{3}{2}        \gcl{1}     \gcl{1} \gnl
    \gcmu       \gcl{1}     \gcl{1}     \gnl
    \gcl{1}     \gbr        \gcl{1}     \gnl
    \gcl{1}     \gcl{1}     \gmu        \gnl
    \gend
    -
    \gbeg{4}{7}
    \gvac{1}    \etiqueta{+}\gwcm{3}       \gnl
    \gcn{3}{1}{3}{2}        \gcl{1}     \gnl
    \gcmu      \gvac{1}     \linea     \gnl
    \linea      \gcn{1}{1}{1}{2} \gvac{1}       \linea \gnl
    \linea      \etiqueta{+}\gcmu       \linea      \gnl
    \linea      \cruce       \linea     \gnl
    \linea      \linea      \gmu   \gnl
    \gend
    \\
    &\equiv& \gbeg{4}{7}
    \gvac{2} \etiqueta{0}\gcmu \gnl
    \gcn{3}{1}{5}{4} \linea \gnl
    \gvac{1} \gcmu \linea \gnl
    \gcn{2}{1}{3}{2} \gmu \gnl
    \gcmu \gcn{2}{1}{2}{3} \gnl
    \linea \linea \gvac{1} \linea \gnl
    \linea \linea \gvac{1} \linea \gnl
    \gend
    +
    \gbeg{4}{7}
    \gvac{1} \etiqueta{+}\gwcm{3} \gnl
    \gcn{3}{1}{3}{2} \linea \gnl
    \gcmu \gvac{1} \linea \gnl
    \linea \gcn{2}{1}{1}{2} \linea \gnl
    \linea \etiqueta{+}\gcmu \linea \gnl
    \linea \linea \linea \linea\gnl
    \linea \linea \gmu \gnl
    \gend
    -
    \gbeg{4}{7}
    \gvac{1} \gwcm{3} \gnl
    \gcn{3}{1}{3}{2} \linea \gnl
    \gcmu \gvac{1} \linea \gnl
    \linea \gcn{2}{1}{1}{2} \linea \gnl
    \linea \etiqueta{+}\gcmu \linea \gnl
    \linea \linea \linea \linea\gnl
    \linea \linea \gmu \gnl
    \gend
    +
    \gbeg{4}{7}
    \gvac{1}   \gwcm{3}       \gnl
    \gcn{3}{1}{3}{2}        \gcl{1}     \gnl
    \gcmu      \gvac{1}     \linea     \gnl
    \linea      \gcn{1}{1}{1}{2} \gvac{1}       \linea \gnl
    \linea      \etiqueta{+}\gcmu       \linea      \gnl
    \linea      \cruce       \linea     \gnl
    \linea      \linea      \linea \linea   \gnl
    \gend
    -
    \gbeg{4}{7}
    \gvac{2}    \etiqueta{0}\gcmu       \gnl
    \gcn{3}{1}{5}{4}        \gcl{1}     \gnl
    \gvac{1}    \gcmu       \gcl{1}     \gnl
    \gcn{2}{1}{3}{2}        \gcl{1}     \gcl{1} \gnl
    \gcmu       \gcl{1}     \gcl{1}     \gnl
    \gcl{1}     \gbr        \gcl{1}     \gnl
    \gcl{1}     \gcl{1}     \gmu        \gnl
    \gend
    -
    \gbeg{4}{7}
    \gvac{1}    \etiqueta{+}\gwcm{3}       \gnl
    \gcn{3}{1}{3}{2}        \gcl{1}     \gnl
    \gcmu      \gvac{1}     \linea     \gnl
    \linea      \gcn{1}{1}{1}{2} \gvac{1}       \linea \gnl
    \linea      \etiqueta{+}\gcmu       \linea      \gnl
    \linea      \cruce       \linea     \gnl
    \linea      \linea      \gmu   \gnl
    \gend
    \end{eqnarray*}
    Simplifying we get
    \begin{displaymath}
    D \equiv \gbeg{4}{7}
    \gvac{2} \etiqueta{0}\gcmu \gnl
    \gcn{3}{1}{5}{4} \linea \gnl
    \gvac{1} \gcmu \linea \gnl
    \gcn{2}{1}{3}{2} \gmu \gnl
    \gcmu \gcn{2}{1}{2}{3} \gnl
    \linea \linea \gvac{1} \linea \gnl
    \linea \linea \gvac{1} \linea \gnl
    \gend
    -
    \gbeg{4}{7}
    \gvac{1} \etiqueta{0}\gwcm{3} \gnl
    \gcn{3}{1}{3}{2} \linea \gnl
    \gcmu \gvac{1} \linea \gnl
    \linea \gcn{2}{1}{1}{2} \linea \gnl
    \linea \etiqueta{+}\gcmu \linea \gnl
    \linea \linea \linea \linea\gnl
    \linea \linea \gmu \gnl
    \gend
    +
    \gbeg{4}{7}
    \gvac{1}    \etiqueta{0}\gwcm{3}       \gnl
    \gcn{3}{1}{3}{2}        \gcl{1}     \gnl
    \gcmu      \gvac{1}     \linea     \gnl
    \linea      \gcn{1}{1}{1}{2} \gvac{1}       \linea \gnl
    \linea      \etiqueta{+}\gcmu       \linea      \gnl
    \linea      \cruce       \linea     \gnl
    \linea      \linea      \linea \linea   \gnl
    \gend
    -
    \gbeg{4}{7}
    \gvac{2}    \etiqueta{0}\gcmu       \gnl
    \gcn{3}{1}{5}{4}        \gcl{1}     \gnl
    \gvac{1}    \gcmu       \gcl{1}     \gnl
    \gcn{2}{1}{3}{2}        \gcl{1}     \gcl{1} \gnl
    \gcmu       \gcl{1}     \gcl{1}     \gnl
    \gcl{1}     \gbr        \gcl{1}     \gnl
    \gcl{1}     \gcl{1}     \gmu        \gnl
    \gend \quad (\modulus{h^{n+1}})
    \end{displaymath}
    Since $\Delta_0 = \Delta^{\op}_0$ then
    \begin{displaymath}
      D \equiv
      \gbeg{4}{7}
    \gvac{2} \etiqueta{0}\gcmu \gnl
    \gcn{3}{1}{5}{4} \linea \gnl
    \gvac{1} \gcmu \linea \gnl
    \gcn{2}{1}{3}{2} \gmu \gnl
    \gcmu \gcn{2}{1}{2}{3} \gnl
    \linea \linea \gvac{1} \linea \gnl
    \linea \linea \gvac{1} \linea \gnl
    \gend
    -
    \gbeg{4}{7}
    \gvac{1} \etiqueta{0}\gwcm{3} \gnl
    \gcn{3}{1}{3}{2} \linea \gnl
    \gcmu \gvac{1} \linea \gnl
    \linea \gcn{2}{1}{1}{2} \linea \gnl
    \linea \gcmu \linea \gnl
    \linea \linea \linea \linea\gnl
    \linea \linea \gmu \gnl
    \gend
    +
    \gbeg{4}{7}
    \gvac{1}    \etiqueta{0}\gwcm{3}       \gnl
    \gcn{3}{1}{3}{2}        \gcl{1}     \gnl
    \gcmu      \gvac{1}     \linea     \gnl
    \linea      \gcn{1}{1}{1}{2} \gvac{1}       \linea \gnl
    \linea      \gcmu       \linea      \gnl
    \linea      \cruce       \linea     \gnl
    \linea      \linea      \linea \linea   \gnl
    \gend
    -
    \gbeg{4}{7}
    \gvac{2}    \etiqueta{0}\gcmu       \gnl
    \gcn{3}{1}{5}{4}        \gcl{1}     \gnl
    \gvac{1}    \gcmu       \gcl{1}     \gnl
    \gcn{2}{1}{3}{2}        \gcl{1}     \gcl{1} \gnl
    \gcmu       \gcl{1}     \gcl{1}     \gnl
    \gcl{1}     \gbr        \gcl{1}     \gnl
    \gcl{1}     \gcl{1}     \gmu        \gnl
    \gend\quad (\modulus{h^{n+1}})
    \end{displaymath}
   The result follows by evaluating this expression at $a$.\qed
\end{proof}

\begin{lemma}
  If $C_1 = \cdots = C_{n-1} = 0$ then for any $a \in \m$ we have
  \begin{displaymath}
    C_n(a) \in \m \wedge \m \wedge \m
  \end{displaymath}
\end{lemma}
\begin{proof}
  Similar to the proof of Lemma~\ref{lem:skew}. \qed
\end{proof}

\begin{theorem}
  Any bialgebra deformation of $U(\M(\alpha,\beta,\gamma))$ satisfying the left and right co-Moufang identities is coassociative and cocommutative.
\end{theorem}
\begin{proof}
  We have proved that $U(\M(\alpha,\beta,\gamma))$ is cocommutative, so we only have to prove that it is coassociative. We will proceed by induction. Assume that $C_1 = \cdots = C_{n-1} = 0$. Since $C_n(a) \in \m \wedge \m \wedge \m$ then it is an eigenvector of $Q \otimes \Id \otimes \Id$ and $\Id \otimes \Id \otimes Q$ of eigenvalue $2$. Hence, Lemma~\ref{lem:P1} and Lemma~\ref{lem:P2} imply  that
  \begin{displaymath}
    2\left(
    \gbeg{3}{5}
    \got{4}{a}                          \gnl
    \gvac{1}            \gcmu           \gnl
    \gcn{2}{1}{3}{2}    \linea          \gnl
    \gcmu               \linea          \gnl
    \linea              \linea   \linea \gnl
    \gend
    -
    \gbeg{3}{5}
    \got{2}{a}                  \gnl
    \gcmu                       \gnl
    \linea  \gcn{1}{1}{1}{2}    \gnl
    \linea  \gcmu               \gnl
    \linea  \linea      \linea  \gnl
    \gend
    \right)
    \equiv
    \gbeg{3}{5}
    \got{4}{a}                          \gnl
    \gvac{1}            \gcmu           \gnl
    \gcn{2}{1}{3}{2}    \linea          \gnl
    \gcmu               \linea          \gnl
    \linea              \linea   \linea \gnl
    \gend
    -
    \gbeg{4}{5}
    \got{4}{a}                  \gnl
    \gvac{1}            \gcmu   \gnl
    \gcn{2}{1}{3}{2}    \linea  \gnl
    \gcmu               \linea  \gnl
    \gbr                \linea \gnl
    \gend
    +
    \gbeg{3}{5}
    \got{2}{a}                  \gnl
    \gcmu                       \gnl
    \linea  \gcn{1}{1}{1}{2}    \gnl
    \linea  \gcmu               \gnl
    \gbr    \linea              \gnl
    \gend
    -
    \gbeg{3}{5}
    \got{2}{a}                  \gnl
    \gcmu                       \gnl
    \linea  \gcn{1}{1}{1}{2}    \gnl
    \linea  \gcmu               \gnl
    \linea  \linea      \linea  \gnl
    \gend
    \quad (\modulus h^{n+1})
  \end{displaymath}
  and
  \begin{displaymath}
    2\left(
    \gbeg{3}{5}
    \got{4}{a}                          \gnl
    \gvac{1}            \gcmu           \gnl
    \gcn{2}{1}{3}{2}    \linea          \gnl
    \gcmu               \linea          \gnl
    \linea              \linea   \linea \gnl
    \gend
    -
    \gbeg{3}{5}
    \got{2}{a}                  \gnl
    \gcmu                       \gnl
    \linea  \gcn{1}{1}{1}{2}    \gnl
    \linea  \gcmu               \gnl
    \linea  \linea      \linea  \gnl
    \gend
    \right)
    \equiv
    \gbeg{3}{5}
    \got{4}{a}\gnl
    \gvac{1}\gcmu   \gnl
    \gcn{2}{1}{3}{2}    \linea \gnl
    \gcmu   \linea  \gnl
    \linea \linea   \linea \gnl
    \gend
    -
    \gbeg{3}{5}
    \got{2}{a}\gnl
    \gcmu \gnl
    \linea  \gcn{1}{1}{1}{2}    \gnl
    \linea  \gcmu   \gnl
    \linea    \linea  \linea \gnl
    \gend
    +
    \gbeg{3}{5}
    \got{2}{a}\gnl
    \gcmu \gnl
    \linea  \gcn{1}{1}{1}{2}    \gnl
    \linea  \gcmu   \gnl
    \linea  \cruce \gnl
    \gend
    -
    \gbeg{3}{5}
    \got{4}{a}\gnl
    \gvac{1}\gcmu   \gnl
    \gcn{2}{1}{3}{2}    \linea \gnl
    \gcmu   \linea  \gnl
    \linea \cruce \gnl
    \gend \quad (\modulus h^{n+1})
  \end{displaymath}
  From the cocommutativity of $\Delta_h$ we get
  \begin{displaymath}
    \gbeg{3}{5}
    \got{4}{a}                          \gnl
    \gvac{1}            \gcmu           \gnl
    \gcn{2}{1}{3}{2}    \linea          \gnl
    \gcmu               \linea          \gnl
    \linea              \linea   \linea \gnl
    \gend
    -
    \gbeg{3}{5}
    \got{2}{a}                  \gnl
    \gcmu                       \gnl
    \linea  \gcn{1}{1}{1}{2}    \gnl
    \linea  \gcmu               \gnl
    \linea  \linea      \linea  \gnl
    \gend
    \equiv
    \gbeg{3}{5}
    \got{2}{a}                  \gnl
    \gcmu                       \gnl
    \linea  \gcn{1}{1}{1}{2}    \gnl
    \linea  \gcmu               \gnl
    \gbr    \linea              \gnl
    \gend
    -
    \gbeg{3}{5}
    \got{4}{a}                          \gnl
    \gvac{1}            \gcmu           \gnl
    \gcn{2}{1}{3}{2}    \linea          \gnl
    \gcmu               \linea          \gnl
    \linea              \linea   \linea \gnl
    \gend
,\quad
    \gbeg{3}{5}
    \got{4}{a}                          \gnl
    \gvac{1}            \gcmu           \gnl
    \gcn{2}{1}{3}{2}    \linea          \gnl
    \gcmu               \linea          \gnl
    \linea              \linea   \linea \gnl
    \gend
    -
    \gbeg{3}{5}
    \got{2}{a}                  \gnl
    \gcmu                       \gnl
    \linea  \gcn{1}{1}{1}{2}    \gnl
    \linea  \gcmu               \gnl
    \linea  \linea      \linea  \gnl
    \gend
    \equiv
     \gbeg{3}{5}
     \got{2}{a}\gnl
    \gcmu \gnl
    \linea  \gcn{1}{1}{1}{2}    \gnl
    \linea  \gcmu   \gnl
    \linea    \linea  \linea \gnl
    \gend
    -
    \gbeg{3}{5}
    \got{4}{a}\gnl
    \gvac{1}\gcmu   \gnl
    \gcn{2}{1}{3}{2}    \linea \gnl
    \gcmu   \linea  \gnl
    \linea \cruce \gnl
    \gend
     \quad (\modulus h^{n+1})
  \end{displaymath}
and
  \begin{displaymath}
    \gbeg{3}{5}
    \got{2}{a}                  \gnl
    \gcmu                       \gnl
    \linea  \gcn{1}{1}{1}{2}    \gnl
    \linea  \gcmu               \gnl
    \gbr    \linea              \gnl
    \gend
    \equiv
    \gbeg{3}{6}
    \got{4}{a}                  \gnl
    \gvac{1}\gcmu \gnl
    \gvac{1}\cruce \gnl
    \gvac{1}\gcn{1}{1}{1}{0}    \linea \gnl
    \gcmu \linea \gnl
    \linea \cruce \gnl
    \gend
    \equiv
    \gbeg{3}{5}
    \got{4}{a}                  \gnl
    \gvac{1}\gcmu   \gnl
    \gcn{2}{1}{3}{2}    \linea \gnl
    \gcmu   \linea  \gnl
    \linea \cruce \gnl
    \gend  \quad (\modulus h^{n+1})
    \end{displaymath}
    so
    \begin{displaymath}
    2\left(
    \gbeg{3}{5}
    \got{4}{a}                          \gnl
    \gvac{1}            \gcmu           \gnl
    \gcn{2}{1}{3}{2}    \linea          \gnl
    \gcmu               \linea          \gnl
    \linea              \linea   \linea \gnl
    \gend
    -
    \gbeg{3}{5}
    \got{2}{a}                  \gnl
    \gcmu                       \gnl
    \linea  \gcn{1}{1}{1}{2}    \gnl
    \linea  \gcmu               \gnl
    \linea  \linea      \linea  \gnl
    \gend
    \right)
    \equiv
    -
    \gbeg{3}{5}
    \got{4}{a}                          \gnl
    \gvac{1}            \gcmu           \gnl
    \gcn{2}{1}{3}{2}    \linea          \gnl
    \gcmu               \linea          \gnl
    \linea              \linea   \linea \gnl
    \gend
    +
    \gbeg{3}{5}
    \got{2}{a}                  \gnl
    \gcmu                       \gnl
    \linea  \gcn{1}{1}{1}{2}    \gnl
    \linea  \gcmu               \gnl
    \linea  \linea      \linea  \gnl
    \gend  \quad (\modulus h^{n+1})
    \end{displaymath}
    Therefore, $\gbeg{3}{5}
    \got{4}{a}                          \gnl
    \gvac{1}            \gcmu           \gnl
    \gcn{2}{1}{3}{2}    \linea          \gnl
    \gcmu               \linea          \gnl
    \gend
    -
    \gbeg{3}{5}
    \got{2}{a}                  \gnl
    \gcmu                       \gnl
    \linea  \gcn{1}{1}{1}{2}    \gnl
    \linea  \gcmu               \gnl
    \gend = 0$, i.e., $C_n(a) = 0$ for all $a \in \m$. The result follows from Lemma~\ref{lem:induction}. \qed
\end{proof}

\begin{theorem}\label{thm:kernel}
Let $U_h(\m)$ be a bialgebra deformation of $U(\m)$ satisfying the left and right co-Moufang identities. Then, for any $x \in U_h(\m)$ the element $C_h(x)$ belongs to the kernel of the map
  \begin{eqnarray*}
    U_h(\m) \tilde{\otimes} U_h(\m) \tilde{\otimes} U_h(\m) & \rightarrow &  U_h(\m) \tilde{\otimes} U_h(\m) \tilde{\otimes} U_h(\m) \\
    x  \tilde{\otimes} y \tilde{\otimes} z &\mapsto & \sum x_{\si} \bullet y  \tilde{\otimes} x_{\sii}  \tilde{\otimes} z + x \bullet z_{\si}  \tilde{\otimes} y  \tilde{\otimes} z_{\sii}
  \end{eqnarray*}
  \end{theorem}
\begin{proof}
 Let $R \colon x \tilde{\otimes} y \tilde{\otimes} z \mapsto \sum x_{\si} \bullet y \tilde{\otimes} x_{\sii} \tilde{\otimes} z$ and $S \colon x \tilde{\otimes} y \tilde{\otimes} z \tilde{\otimes} \sum x \bullet z_{\si} \tilde{\otimes} y \tilde{\otimes} z_{\sii}$. The left co-Moufang identity
 \begin{displaymath}
        \gbeg{4}{8}
        \gvac{2}    \gcmu       \gnl
        \gcn{3}{1}{5}{4}        \gcl{1}     \gnl
        \gvac{1}    \gcmu       \gcl{1}     \gnl
        \gcn{2}{1}{3}{2}        \gcl{1}     \gcl{1} \gnl
        \gcmu       \gcl{1}     \gcl{1}     \gnl
        \gcl{1}     \gbr        \gcl{1}     \gnl
        \gbox{4}{4}\gmu        \gcl{1}     \gcl{1}     \gnl
        \gend
        =
        \gbeg{4}{8}
        \gcmu \gnl
        \gcl{1} \gcn{1}{1}{1}{2}        \gnl
        \gcl{1} \gcmu                   \gnl
        \gcl{1} \gcl{1} \gcn{1}{1}{1}{2}\gnl
        \gcl{1} \gcl{1} \gcmu           \gnl
        \gcl{1} \gbr    \gcl{1}         \gnl
        \gbox{4}{4} \gmu    \gcl{1} \gcl{1}         \gnl
        \gend
 \end{displaymath}
 can be written as
 $R(C_l(x)) = S(C_r(x))$ where $C_l(x) = (\Delta_h \tilde{\otimes} \Id) \Delta_h (x)$ and $C_r(x) = (\Id \tilde{\otimes} \Delta_h)\Delta_h(x)$. We also observe that
 \begin{displaymath}
    \gbeg{4}{7}
    \gvac{1}\gcmu \gnl
    \gcn{1}{1}{3}{2}  \gcn{1}{2}{3}{4} \gnl
    \gcmu    \gnl
    \linea \linea \gcmu \gnl
    \linea \cruce \linea \gnl
    \gmu \linea    \linea   \gnl
    \gbox{4}{4}\gcn{2}{1}{2}{1} \linea  \linea \gnl
    \gend
    =
    \gbeg{4}{7}
    \gvac{1}\gcmu \gnl
    \gcn{1}{2}{3}{2}  \gcn{1}{1}{3}{4} \gnl
    \gvac{2}    \gcmu    \gnl
    \gcmu \linea \linea \gnl
    \linea \cruce \linea \gnl
    \gmu \linea    \linea   \gnl
    \gbox{4}{4}\gcn{2}{1}{2}{1} \linea  \linea \gnl
    \gend
 \end{displaymath}
 so $S(C_l(x)) = R(C_r(x))$. Therefore, $(R+S)(C_l(x)) = (R+S)(C_r(x))$ and
 \begin{displaymath}
   (R+S)(C_h(x)) = (R+S)(C_l(x)-C_r(x)) = 0.
 \end{displaymath} \qed
\end{proof}

\begin{theorem}
  Let $\g$ be a finite-dimensional central simple Lie algebra and $U_h(\g)$ a bialgebra deformation of $U(\g)$ satisfying the left and right co-Moufang identities. Then $U_h(\g)$ is coassociative.
\end{theorem}
\begin{proof}
  Let us assume that $C_1 = \cdots = C_{n-1} = 0$. We will prove that $C_n = 0$. The following Sweedler like notation for $C_n(x)$ will be very useful:
  \begin{displaymath}
    C_n(x) = \sum x' \otimes x'' \otimes x''',\quad x \in U(\g)
  \end{displaymath}
    With this notation Theorem~\ref{thm:kernel} gives
  \begin{equation}
    \label{eq:squares}
    \sum x'_{(1)} x'' \otimes x'_{(2)} \otimes x''' + \sum x'x'''_{(1)} \otimes x'' \otimes x'''_{(2)} = 0.
  \end{equation}
  Lemma~\ref{lem:induction} with $\varphi_n = \sum_{i+j = n} (\Delta_j\otimes \Id)\Delta_i$ and $\psi_n = \sum_{i+j = n}(\Id \otimes \Delta_j)\Delta_i$ implies that
  \begin{eqnarray*}
    C_n(xy) &=& \sum C_n(x) (y_{(1)} \otimes y_{(2)} \otimes y_{(3)}) + \sum (x_{(1)} \otimes x_{(2)} \otimes x_{(3)}) C_n(y) \\
    &=& \sum x'y_{(1)} \otimes x'' y_{(2)} \otimes x''' y_{(3)} + \sum x_{(1)}y' \otimes x_{(2)}y'' \otimes x_{(3)}y'''.
  \end{eqnarray*}
  We can apply Theorem~\ref{thm:kernel} to the element $C_h(xy)$ to get
  \begin{eqnarray*}
    && \sum (x'_{(1)}y_{(1)})(x''y_{(2)} )\otimes x'_{(2)}y_{(3)} \otimes x'''y_{(4)}\\
    &&\quad\quad +
    \sum (x_{(1)}y'_{(1)})(x_{(2)}y'') \otimes x_{(3)}y'_{(2)} \otimes x_{(4)}y'''\\
    && \quad\quad +
    \sum (x'y_{(1)})(x'''_{(1)}y_{(2)}) \otimes x''y_{(3)} \otimes x'''_{(2)}y_{(4)}\\
    &&\quad\quad +
     \sum (x_{(1)}y')(x_{(2)}y'''_{(1)}) \otimes x_{(3)}y'' \otimes x_{(4)}y'''_{(2)} = 0.
  \end{eqnarray*}
  In fact, since $C_n(x)$ and $C_n(y)$ also satisfy (\ref{eq:squares}) then all the summands where $y_{(1)}$ or $x_{(2)}$ are scalars vanish. Moreover, since $C_n(\g) \subseteq \g \wedge \g \wedge \g$ then substituting $x$ and $y$ for elements $a$ and $b$ of $\m$ we get 
  \begin{eqnarray*}
    && \sum (a'ba'' \otimes 1 \otimes a''' + ba'' \otimes a' \otimes a''' + b'ab'' \otimes 1 \otimes b'''\\
     && \quad\quad  + ab'' \otimes b'\otimes b''' + a'ba''' \otimes a'' \otimes 1 + a'b \otimes a'' \otimes a''' \\
   && \quad\quad + b'ab'''\otimes b'' \otimes 1 + b'a\otimes b'' \otimes b''') =0.
  \end{eqnarray*}
  Projecting onto $\ker \epsilon \otimes \ker \epsilon \otimes \ker \epsilon$ this equality leads to
  \begin{displaymath}
    \sum ba'' \otimes a' \otimes a''' + ab'' \otimes b' \otimes b''' + a'b \otimes a'' \otimes a''' + b'a \otimes b'' \otimes b''' = 0.
  \end{displaymath}
  The skew-symmetry of the tensors in $C_n(\g)$ implies that
  \begin{displaymath}
    \sum [a',b]\otimes a'' \otimes a''' - [a,b']\otimes b'' \otimes b''' = 0.
  \end{displaymath}
  This equality with the substitution $b = a$ is
  \begin{displaymath}
    \sum [a,a']\otimes a'' \otimes a''' = 0.
  \end{displaymath}
  We can also obtain 
  \begin{displaymath}
    \sum a'\otimes[a,a''] \otimes a'''  = 0, \quad \sum a'\otimes a'' \otimes [a,a'''] = 0
  \end{displaymath}
  by the skew-symmetry of $C_n(\g)$.
  Therefore, if $\lambda\colon \g \rightarrow \Endo(\g \wedge \g \wedge \g)$ denotes the representation of $\g$ on the $\g$-module $\g \wedge \g \wedge \g$ then
  \begin{equation}\label{eq:kill}
    \lambda_a \cdot C_n(a) = 0.
  \end{equation}
 Lemma~\ref{lem:induction} implies that $C_n$ is a 1-cocycle of $\g$ with values in $\g \wedge \g \wedge \g$. Since $\g$ is simple this cocycle is a cobundary so there exists $r\in \g \wedge \g \wedge \g$ such that $C_n(a) = \lambda_a \cdot r$. By (\ref{eq:kill}), we finally obtain that $\lambda^2_a \cdot r = 0$ for all $a \in \g$. In particular, the Casimir operator of $\g$ kills the element $r$. However,  the Casimir operator acts as a non-zero scalar multiple of the identity on any non-trivial irreducible module of $\g$. It follows that $r$ spans a trivial submodule and that $C_n(a) = \lambda_a \cdot r = 0$ for all $a\in \g$. Hence $C_n = 0$ by Lemma~\ref{lem:induction}. \qed
  \end{proof}

  \begin{theorem}
    Let $\g$ be a finite-dimensional central simple Lie algebra. If $U_h(\g)$ is a bialgebra deformation that satisfies the left and right Moufang identities then $U_h(\g)$ is associative.
  \end{theorem}
  \begin{proof}
    We will write the associator in $U_h(\g)$ of elements $x,y,z \in U(\g)$ as
    \begin{displaymath}
      (x,y,z)_{\bullet} = (x \bullet y) \bullet z - x \bullet (y \bullet z) = A_1(x,y,z) h + A_2(x,y,z) h^2 + \cdots
    \end{displaymath}
    and we will prove by induction that $A_n = 0$. Assume that $A_1 = \cdots = A_{n-1} = 0$, then
    \begin{displaymath}
    \gbeg{4}{7}
    \etiqueta{0}\gcmu       \gcl{1}     \gcl{1}     \gnl
    \gcl{1}     \gbr        \gcl{1}     \gnl
    \gcl{1}     \gcl{1}     \gmu        \gnl
    \gcl{1}     \gcl{1}     \gcn{1}{1}{2}{1}    \gnl
    \gcl{1}     \gmu                    \gnl
    \gcl{1}     \gcn{1}{1}{2}{1}        \gnl
    \gmu                                \gnl
    \gend
    +
    \gbeg{4}{7}
    \etiqueta{+}\gcmu       \gcl{1}     \gcl{1}     \gnl
    \gcl{1}     \gbr        \gcl{1}     \gnl
    \gcl{1}     \gcl{1}     \gmu        \gnl
    \gcl{1}     \gcl{1}     \gcn{1}{1}{2}{1}    \gnl
    \gcl{1}     \gmu                    \gnl
    \gcl{1}     \gcn{1}{1}{2}{1}        \gnl
    \gmu                                \gnl
    \gend
    =
    \gbeg{4}{7}
    \etiqueta{0}\gcmu       \linea      \linea      \gnl
    \linea      \cruce      \linea      \gnl
    \gmu        \linea      \linea      \gnl
    \gcn{2}{1}{2}{3}        \linea      \linea  \gnl
    \gvac{1}    \gmu        \linea      \gnl
    \gvac{1}    \gcn{2}{1}{2}{3}        \linea  \gnl
    \gvac{2}    \gmu         \gnl
    \gend
    +
     \gbeg{4}{7}
    \etiqueta{+}\gcmu       \linea      \linea      \gnl
    \linea      \cruce      \linea      \gnl
    \gmu        \linea      \linea      \gnl
    \gcn{2}{1}{2}{3}        \linea      \linea  \gnl
    \gvac{1}    \gmu        \linea      \gnl
    \gvac{1}    \gcn{2}{1}{2}{3}        \linea  \gnl
    \gvac{2}    \gmu         \gnl
    \gend
    \end{displaymath}
   implies that
    \begin{displaymath}
    \gbeg{4}{7}
    \etiqueta{0}\gcmu       \gcl{1}     \gcl{1}     \gnl
    \gcl{1}     \gbr        \gcl{1}     \gnl
    \gcl{1}     \gcl{1}     \gmu        \gnl
    \gcl{1}     \gcl{1}     \gcn{1}{1}{2}{1}    \gnl
    \gcl{1}     \gmu                    \gnl
    \gcl{1}     \gcn{1}{1}{2}{1}        \gnl
    \gmu                                \gnl
    \gend
    \equiv
    \gbeg{4}{7}
    \etiqueta{0}\gcmu       \linea      \linea      \gnl
    \linea      \cruce      \linea      \gnl
    \gmu        \linea      \linea      \gnl
    \gcn{2}{1}{2}{3}        \linea      \linea  \gnl
    \gvac{1}    \gmu        \linea      \gnl
    \gvac{1}    \gcn{2}{1}{2}{3}        \linea  \gnl
    \gvac{2}    \gmu         \gnl
    \gend \quad (\modulus{h^{n+1}}).
    \end{displaymath}
    Hence, modulo $h^{n+1}$,
    the element $ a \in \g$ verifies that
    \begin{displaymath}
      a\bullet (y \bullet z) +y \bullet (a  \bullet z) \equiv (z  \bullet y)  \bullet z + (y  \bullet a)  \bullet z
    \end{displaymath}
    i.e.,
    \begin{displaymath}
      (a,y,z)_{\bullet} \equiv - (y,a,z)_{\bullet}.
    \end{displaymath}
    Having also considered the right Moufang identity, we would have obtained that
    $\bar{a} \in \Nalt(\overline{U_h})$, where $\bar{x}$ and $\overline{U_h}$ stand for $x + h^{n+1}U_h(\g)$  and $U_h(\g)/ h^{n+1} U_h(\g)$ respectively. Thus, the jacobian of $\bar{a},\bar{b},\bar{c} \in \g + h^{n+1}U_h$ is 
    \begin{displaymath}
   \Jac(\bar{a},\bar{b},\bar{c}) = 6(\bar{a},\bar{b},\bar{c})_{\bullet} = 6 \overline{A_n(a,b,c)}h^{n}
    \end{displaymath}

    Let $\tilde{\m}$ be the Malcev subalgebra generated by  $\g + h^{n+1}U_h$ inside $\overline{U_h}$ but where the product is understood to be the commutator product $\overline{x \bullet y - y \bullet x}$. Since
    \begin{displaymath}
      \tilde{\m}/( \tilde{\m} \cap h \overline{U_h}) \cong ( \tilde{\m} + h \overline{U_h})/h\overline{U_h} \subseteq \overline{U_h}/h\overline{U_h} \cong U(\g)
    \end{displaymath}
     then $\tilde{\m}/(\tilde{\m} \cap h \overline{U_h})$ is a Lie algebra isomorphic to $\g$. Since $\tilde{\m} \cap h \overline{U_h}$ is a nilpotent ideal of $\tilde{\m}$, the radical in fact, we can find a Lie subalgebra $\tilde{\mathfrak{s}}$ of $\tilde{\m}$ such that $\tilde{\m} = \tilde{\mathfrak{s}} \oplus (\tilde{\m} \cap h \overline{U_h})$, a direct sum of vector spaces \cite{G77,K77}. Let $S$ be the subalgebra of $\overline{U_h}$ generated by $\{1\}\cup \tilde{\mathfrak{s}} \cup \{h\}$. Clearly $S/(S\cap h\overline{U_h}) \cong (S + h \overline{U_h})/h\overline{U_h} \cong U(\g)$ which implies that $h^n S = h^n \overline{U_h}$ since $h^{n+1}\overline{U_h} = 0 = h^{n+1}S$. We recursively get $h^iS = h^i\overline{U_h}$ for $i=n,n-1,\dots, 1,0$ so $S = \overline{U_h}$. Since $\tilde{\mathfrak{s}} \subseteq \Nalt(\overline{U_h})$ and $\tilde{\mathfrak{s}}$ is a Lie algebra then $S$ is associative. Thus $A_n(x,y,z) = 0$. \qed
  \end{proof}

\appendix
\section{Proof of Theorem\ref{thm:cocommutative}}
We include a non-graphical proof of Theorem~\ref{thm:cocommutative} so that the reader can check the correspondence between the movements performed in the diagrams and the algebraic manipulations of the equalities.

Recall that $\Delta_h = \sum_n h^n \Delta_n$ and $p_h = \sum_n h^n p_n$ stand for the comultiplication and multiplication of $U_h(\m)$ respectively. The map $x \tilde{\otimes} y \mapsto y \tilde{\otimes} x$ will be denoted by $\tau$. We rewrite part of the statement of Lemma~\ref{lem:basic} for convenience:
\begin{displaymath}
   \gbeg{3}{4}
            \gcmu                               \gnl
            \linea      \gcn{1}{1}{1}{2}        \gnl
            \linea      \gcmu                   \gnl
            \gmu        \linea                  \gnl
            \gend
            =
            \gbeg{3}{4}
            \gvac{1}            \gcmu       \gnl
            \gcn{2}{1}{3}{2}    \linea      \gnl
            \gcmu               \linea      \gnl
            \gmu                \linea      \gnl
			\gend \mbox{corresponds to }
            (p_h \tilde{\otimes} \Id)(\Id \tilde{\otimes} \Delta_h)\Delta_h=((p_h\Delta_h) \tilde{\otimes} \Id) \Delta_h
\end{displaymath}
and
\begin{displaymath}
    \gbeg{3}{5}
            \gcmu       \gnl
            \linea      \gcn{1}{1}{1}{2}        \gnl
            \linea      \gcmu                   \gnl
            \linea      \cruce                  \gnl
            \gmu        \linea                  \gnl
            \gend
            =
            \gbeg{3}{5}
            \gvac{1}            \gcmu       \gnl
            \gcn{2}{1}{3}{2}    \linea      \gnl
            \gcmu               \linea      \gnl
            \linea              \cruce      \gnl
            \gmu                \linea      \gnl
            \gend
    \mbox{corresponds to }
    (p_h \tilde{\otimes} \Id)(\Id \tilde{\otimes} \Delta^{\op}_h)\Delta_h = (p_h \tilde{\otimes} \Id)(\Id \tilde{\otimes} \tau)(\Delta_h \tilde{\otimes} \Id)\Delta_h.
\end{displaymath}
Part 2) of Lemma~\ref{lem:basic} establishes that $(\Id \tilde{\otimes} p_h)(\Delta_h \tilde{\otimes} \Id)\Delta_h = (\Id \tilde{\otimes} (p_h\Delta_h))\Delta_h$. After composing with $\tau$ we get $\tau (\Id \tilde{\otimes} (p_h\Delta_h)) \Delta_h= \tau (\Id \tilde{\otimes} p_h)(\Delta_h \tilde{\otimes} \Id)\Delta_h$. The left-hand side of this equality is $((p_h\Delta_h) \tilde{\otimes} \Id)\Delta^{\op}_h$. The right-hand side of the equality is $\tau(\Id \tilde{\otimes} p_h)(\tau \tilde{\otimes} \Id)(\Delta^{\op}_h \tilde{\otimes} \Id) \Delta_h = (p_h \tilde{\otimes} \Id)(\Id \tilde{\otimes} \tau)(\Delta^{\op}_h \tilde{\otimes} \Id)\Delta_h$. Therefore,
\begin{displaymath}
  ((p_h\Delta_h) \tilde{\otimes} \Id)\Delta^{\op}_h = (p_h \tilde{\otimes} \Id)(\Id \tilde{\otimes} \tau)(\Delta^{\op}_h \tilde{\otimes} \Id)\Delta_h
\end{displaymath}

The proof of Theorem \ref{thm:cocommutative} goes as follows. By Lemma \ref{lem:basic} we have
\begin{eqnarray*}
  D&:=& (p_h\Delta_h \tilde{\otimes} \Id) \Delta_h - (p_h\Delta_h \tilde{\otimes} \Id)\Delta^{\op}_h \\
  &=&
  (p_h \tilde{\otimes} \Id)(\Id \tilde{\otimes} \Delta_h)\Delta_h - (p_h\Delta_h \tilde{\otimes} \Id)\Delta^{\op}_h \\
  &=&
  (p_h \tilde{\otimes} \Id)(\Id \tilde{\otimes} \Delta_h)\Delta_h - (p_h \tilde{\otimes} \Id)(\Id \tilde{\otimes} \tau)(\Delta^{\op}_h \tilde{\otimes} \Id)\Delta_h
\end{eqnarray*}
Since $\Delta_j = \Delta^{\op}_j$ $j = 0,\dots,n-1$ then, modulo $h^{n+1}$,
\begin{eqnarray*}
  D&\equiv&
  (p_h \tilde{\otimes} \Id)(\Id \tilde{\otimes} \Delta_h) \Delta_h - (p_h \tilde{\otimes} \Id)(\Id \tilde{\otimes} \tau)(\Delta^{\op} \tilde{\otimes} \Id)\Delta_0\\
  && \quad - \sum_{\substack{i+j+k \leq n\\ 1\leq i}} h^{i+j+k} (p_k \otimes \Id)(\Id \otimes \tau)(\Delta_j \otimes \Id)\Delta_i\\
  &\equiv&
  (p_h \tilde{\otimes} \Id)(\Id \tilde{\otimes} \Delta_h) \Delta_h - (p_h \tilde{\otimes} \Id)(\Id \tilde{\otimes} \tau)(\Delta^{\op}_h \tilde{\otimes} \Id)\Delta_0\\
  && \quad +(p_h \tilde{\otimes} \Id)(\Id \tilde{\otimes} \tau)(\Delta_h \tilde{\otimes} \Id)\Delta_0- (p_h \tilde{\otimes} \Id)(\Id \tilde{\otimes} \tau)(\Delta_h \tilde{\otimes} \Id)\Delta_h
\end{eqnarray*}
By Lemma~\ref{lem:basic}
\begin{eqnarray*}
  D &=& (p_h \tilde{\otimes} \Id)(\Id \tilde{\otimes} \Delta_h) \Delta_h - (p_h \tilde{\otimes} \Id)(\Id \tilde{\otimes} \tau)(\Delta^{\op}_h \tilde{\otimes} \Id)\Delta_0\\
  && \quad +(p_h \tilde{\otimes} \Id)(\Id \tilde{\otimes} \tau)(\Delta_h \tilde{\otimes} \Id)\Delta_0- (p_h \tilde{\otimes} \Id)(\Id \tilde{\otimes} \Delta^{\op}_h)\Delta_h\\
  &\equiv& (p_h \tilde{\otimes} \Id)(\Id \tilde{\otimes} \Delta_h) \Delta_h - (p_h \tilde{\otimes} \Id)(\Id \tilde{\otimes} \tau)(\Delta^{\op}_h \tilde{\otimes} \Id)\Delta_0\\
  && \quad +(p_h \tilde{\otimes} \Id)(\Id \tilde{\otimes} \tau)(\Delta_h \tilde{\otimes} \Id)\Delta_0- (p_h \tilde{\otimes} \Id)(\Id \tilde{\otimes} \Delta^{\op}_h)\Delta_0 \\
  && \quad - \sum_{\substack{i+j+k \leq n \\ 1\leq i}} h^{i+j+k}(p_k \otimes \Id)(\Id \otimes \Delta^{\op}_j)\Delta_i.
\end{eqnarray*}
Since $\Delta_j = \Delta^{\op}_j$ $j=0,\dots,n-1$ then, modulo $h^{n+1}$,
\begin{eqnarray*}
  D &\equiv&
  (p_h \tilde{\otimes} \Id)(\Id \tilde{\otimes} \Delta_h) \Delta_h - (p_h \tilde{\otimes} \Id)(\Id \tilde{\otimes} \tau)(\Delta^{\op}_h \tilde{\otimes} \Id)\Delta_0\\
  && \quad +(p_h \tilde{\otimes} \Id)(\Id \tilde{\otimes} \tau)(\Delta_h \tilde{\otimes} \Id)\Delta_0- (p_h \tilde{\otimes} \Id)(\Id \tilde{\otimes} \Delta^{\op}_h)\Delta_0 \\
  && \quad - \sum_{\substack{i+j+k \leq n \\ 1\leq i}} h^{i+j+k}(p_k \otimes \Id)(\Id \otimes \Delta_j)\Delta_i\\
  &\equiv&
  (p_h \tilde{\otimes} \Id)(\Id \tilde{\otimes} \Delta_h) \Delta_0 - (p_h \tilde{\otimes} \Id)(\Id \tilde{\otimes} \tau)(\Delta^{\op}_h \tilde{\otimes} \Id)\Delta_0\\
  && \quad +(p_h \tilde{\otimes} \Id)(\Id \tilde{\otimes} \tau)(\Delta_h \tilde{\otimes} \Id)\Delta_0- (p_h \tilde{\otimes} \Id)(\Id \tilde{\otimes} \Delta^{\op}_h)\Delta_0.
\end{eqnarray*}
Any $a \in \m$ is primitive, i.e., $\Delta_0(a) = a \otimes 1 + 1 \otimes a$ thus
\begin{eqnarray*}
    &&(p_h \tilde{\otimes} \Id)(\Id \tilde{\otimes} \Delta_h) \Delta_0(a) = \\
    &&\quad\quad (p_h \tilde{\otimes} \Id)(\Id \tilde{\otimes} \Delta_h)(a \otimes 1 + 1 \otimes a) = \Delta_h(a) + a \otimes 1\\
    && (p_h \tilde{\otimes} \Id)(\Id \tilde{\otimes} \tau)(\Delta^{\op}_h \tilde{\otimes} \Id)\Delta_0(a) = \\
    && \quad\quad (p_h \tilde{\otimes} \Id)(\Id \tilde{\otimes} \tau)(\Delta^{\op}_h \tilde{\otimes} \Id)(a \otimes 1 + 1 \otimes a) = \Delta^{\op}_h(a) + a \otimes 1\\
    && (p_h \tilde{\otimes} \Id)(\Id \tilde{\otimes} \tau)(\Delta_h \tilde{\otimes} \Id)\Delta_0(a) =\\
    &&\quad\quad (p_h \tilde{\otimes} \Id)(\Id \tilde{\otimes} \tau)(\Delta_h \tilde{\otimes} \Id)(a \otimes 1 + 1 \otimes a) = \Delta_h(a) + a \otimes 1\\
    && (p_h \tilde{\otimes} \Id)(\Id \tilde{\otimes} \Delta^{\op}_h)\Delta_0(a) =\\
    && \quad\quad (p_h \tilde{\otimes} \Id)(\Id \tilde{\otimes} \Delta^{\op}_h)(a \otimes 1 + 1 \otimes a) =
    \Delta^{\op}_h(a) + a \otimes 1
\end{eqnarray*}
and
\begin{displaymath}
  (p_h\Delta_h \tilde{\otimes} \Id) \Delta_h(a) - (p_h\Delta_h \tilde{\otimes} \Id)\Delta^{\op}_h(a) \equiv 2(\Delta_h(a) - \Delta^{\op}_h(a)) \quad (\modulus{h^{n+1}})
\end{displaymath}
Since $\Delta_i = \Delta^{\op}_i$ $i =0,\dots, n-1$ we finally get
\begin{displaymath}
   (p_0\Delta_0 \tilde{\otimes} \Id) \Delta_n(a) - (p_0\Delta_0 \tilde{\otimes} \Id)\Delta^{\op}_n(a) = 2(\Delta_n(a) - \Delta^{\op}_n(a)).
\end{displaymath}
This shows that $(\Delta_n - \Delta^{\op}_n)(a)$ is an eigenvector of $Q$ of eigenvalue $2$ so \begin{displaymath}
  (\Delta_n - \Delta^{\op}_n)(a) \in \m \wedge \m
\end{displaymath}
for any $a \in \m$. The result follows from \cite{MP}.

\def\cprime{$'$}

\end{document}